%% file: main.tex
\documentclass[a4paper, 11pt]{article}
\usepackage[utf8]{inputenc}
\usepackage{amsmath}
\usepackage{mathtools}
\usepackage{amssymb}
\usepackage[breakable]{tcolorbox}
\usepackage{amsthm}
\usepackage{graphicx}
\usepackage{hyperref}
\usepackage{tikz}
\usepackage{pgfplots}
\usepackage{breakcites}
\usepackage{fancyhdr}
\usepackage{lastpage}
\usepackage[normalem]{ulem}
\usepackage{actuarialangle}
\usepackage[nottoc]{tocbibind}
\usepackage[mathscr]{euscript}
\usepackage{geometry}
\usepackage{bm}
\usepackage{bbm}
\newtcolorbox{ntcolorbox}[1][]{%
    breakable
}

\allowdisplaybreaks

\usepackage{todonotes}

\usepackage{makecell}
\begin{document}
\newtheorem{thm}{Theorem}[section]
\newtheorem{lemma}[thm]{Lemma}
\newtheorem{cor}[thm]{Corollary}
\newtheorem{prop}[thm]{Proposition}
\newtheorem{defin}[thm]{Definition}
\newtheorem{example}[thm]{Example}
\newtheorem{calc}[thm]{Calculation}
\newtheorem{choice}[thm]{Choice}
\newtheorem{rmk}[thm]{Remark}
\newtheorem{conj}[thm]{Conjecture}
\newtheorem{ques}[thm]{Question}
\newtheorem{claim}[thm]{Claim}
\let\oldcalc\calc
\renewcommand{\calc}{\oldcalc\normalfont}
\let\oldexample\example
\renewcommand{\example}{\oldexample\normalfont}
\let\oldrmk\rmk
\renewcommand{\rmk}{\oldrmk\normalfont}
\numberwithin{equation}{subsection}

\newcommand{\bR}{\mathbb{R}}
\newcommand{\bT}{\mathbb{T}}

\newcommand\xqed[1]{%
  \leavevmode\unskip\penalty9999 \hbox{}\nobreak\hfill
  \quad\hbox{#1}}
\newcommand\demo{\xqed{$\circ$}}
\fancyfoot[C]{\thepage{}}
\title{Signature Reconstruction from Randomized Signatures}
\author{\textbf{Mie Glückstad} \\ Mathematical Institute \\ University of Oxford \\ \emph{mie.gluckstad@exeter.ox.ac.uk} \and
\textbf{Nicola Muça Cirone}\\ Department of Mathematics \\ Imperial College London \\ \emph{n.muca-cirone22@imperial.ac.uk}
\and
\textbf{Josef Teichmann}\\ Department of Mathematics \\ ETH Zürich \\ \emph{jteichma@math.ethz.ch}}
\date{}
\maketitle
\begin{abstract}
\noindent Controlled ordinary differential equations driven by continuous bounded variation curves can be considered a continuous time analogue of recurrent neural networks for the construction of expressive features of the input curves. We ask up to which extent well known signature features of such curves can be reconstructed from controlled ordinary differential equations with (untrained) random vector fields. The answer turns out to be algebraically involved, but essentially the number of signature features, which can be reconstructed from the non-linear flow of the controlled ordinary differential equation, is exponential in its hidden dimension, when the vector fields are chosen to be neural with depth two. Moreover, we characterize a general linear independence condition on arbitrary vector fields, under which the signature features up to some fixed order can always be reconstructed. Algebraically speaking this complements in a quantitative manner several well known results from the theory of Lie algebras of vector fields and puts them in a context of machine learning.
\end{abstract}
\newpage
\tableofcontents
\newpage
\section{Introduction}
Controlled ordinary differential equation (CDEs)
$$
Y_t = y + \sum_{i=1}^d \int_0^t V_i(Y_s) d X^i_s
$$
driven by a continuous bounded variation input path $X: [0,T] \to \mathbb{R}^d $ are a continuous time analogue of recurrent neural networks. Here $V_i: \mathbb{R}^N \to \mathbb{R}^N$, $i=1,\ldots,d$ are smooth vector fields with globally bounded derivatives on $\mathbb{R}^N$. $N$ is referred to as the \emph{hidden dimension}. The vector fields correspond to the (recurrent) cell of the network. Each component of $Y$ is called a feature of $X$. In machine learning applications the regressions on features are performed to approximate non-linear functionals of $X$ up to time $T$. \\

It is an important question whether solutions of a fixed controlled differential equation at time $T$ provide, when one varies over all intial values $Y_0 = y$, an exhaustive set of (non-linear) features of $X$, where exhaustive refers to density properties in spaces of path space functionals. A more specific question is whether one can approximate by linear combinations of solutions $Y_T$ for different intial values of a given controlled differential equation, the well known signature features of $X$ (at time $T$), which are dense in multiple senses (see, e.g., the recent work \cite{CucSchTei:23}).\\


In \cite{AkyildirimTeichmann} it has been conjectured that the signature features, see Subsection \ref{subs:sigs}, of a bounded variation path $X: [0,T] \rightarrow \mathbb{R}^d$ can almost surely be reconstructed from its randomized signature, i.e., from the collection of solutions $(Y_T^y)_{y \in \mathbb{R}^N}$ at time $T$ to the controlled differential equation with randomly chosen vector fields of neural network type
$$ 
Y_t := y + \sum_{i=1}^d\int_0^t \sigma(A_i Y_s + b_i) dX_s^i \quad \text{ for } t \in [0,T],
$$
considered across all different initial values $y \in \mathbb{R}^N$. Here $A_i \in \mathbb{R}^{N \times N}$ and $b_i \in \mathbb{R}^N$ are samples of random matrices, and $\sigma$ is an activation function applied entry-wise. 
\\
Our original aim was to give a detailed proof of this conjecture, based on a sketch from \cite{AkyildirimTeichmann}, which suggested using a Taylor expansion to explicitly relate the signature components to the solutions of the randomized signature (see Section \ref{subsec:TaylorExp}). The sketch, however, relied on a crucial assumption of linear independence between the iterated vector fields generated by $V_i := \sigma(A_i \cdot + b_i)$ for $i \in \{1, \dots, d\}$, and this question of linear independence turned out to be algebraically richer and more delicate, than we had originally anticipated. In particular it is only true up to a certain depth of iterations but in general wrong (on finite dimensional input spaces). We also refer here to important algebraic or geometric work by \cite{Bah:21} and in particular \cite{Mol:87}, where natural identities of Lie polynomials are described and derived. Our work provides an alternative, more quantitative perspective on these seminal results. \\


In Section \ref{sec:linindep} we study the structure of the iterated vector fields for general choices of vector fields $V_1, \dots, V_d: \mathbb{R}^N \rightarrow \mathbb{R}^N$, which need not be of the 'randomized signature form', and observe that they can be represented graphically as vertex- and edge-labelled planar recursive trees, inspired by the similar tree-like representations from \cite{mclachlan2017butcherseriesstoryrooted} and \cite{GUBINELLI}. Indeed, the question of linear independence of the iterated vector fields, boils down to a question of linear independence between the associated \emph{tree-like vector fields}, as defined in Section \ref{subsec:TreeVF}.\\

We illustrate by an example in Remark \ref{rmk:Dimensions}, that having linear independence between the iterated vector fields up to order $m$ may be entirely impossible (no matter the choice of vector fields) if the dimension $N$ of the input space is not chosen to be sufficiently large compared to $m$. We also see in Remark \ref{rmk:NoDepth} that the tree-like vector fields are, perhaps surprisingly, \emph{not} linearly independent, when $V_1, \dots, V_d$ are chosen to be of the classical randomized signature form.\\

In order to break the algebraic relations which arise between the tree-like vector fields in the randomized signature case, we add \emph{depth} to the randomized signature (see also Section \ref{subsec:RanSig}). Namely, we show that when $N \geq m-1$, choosing $V_1, \dots, V_d: \mathbb{R}^N \rightarrow \mathbb{R}^N$ to be randomized signatures of \emph{depth two} with exponential activation function, we are able to obtain linear independence of the tree-like vector fields up to order $m$. In Section \ref{sec:SigRec} we outline the signature reconstruction results, inspired by the conjecture from \cite{AkyildirimTeichmann}, and see that it is possible to reconstruct the signature components up to order $m$ if linear independence between the tree-like vector fields up to order $m$ is satisfied. This, of course, in particular implies that we are able to reconstruct the signature components up to order $m$, for the randomized signature with depth two and exponential activation when $N \geq m-1$. We emphasize here that since the number of signature components of order $m$ is $d^m$, the bound $N \geq m-1$ is indeed logarithmically small in terms of the number of reconstructed siganture components (when $d \geq 2$). In other words: as the hidden dimension $N$ increases linearly, the number of signature components which can be reconstructed by this method increases exponentially.\\

In Section \ref{subsec:sigreconstructLie} we briefly discuss modifications of the signature reconstruction result to the case where the vector fields $V_1, \dots, V_d: G \rightarrow TG$ are defined on some finite-dimensional Lie group. Proving the linear independence results in this case should follow by analogous arguments to what we have done in Section \ref{sec:linindep}, and provides the added benefit of having the solutions of the differential equations contained on a (possibly compact) Lie group.
\section{Setting}
\subsection{Signatures and rough paths}
\label{subs:sigs}

The study of signatures originates from the theory of rough paths, see \cite{LyonsRCDE}, or the more introductory resources \cite{Allan}, \cite{LyonsRoughDE} \cite{FrizHairer} for an overview of this. 
Its reconstruction is of particular interest since by the classical Stone-Weierstrass theorem, linear functionals are dense in the space of continuous real-valued functions defined on compact sets of un-parameterized paths. 
This density property makes signature coefficients well-suited as feature representations for machine learning tasks involving sequential data \cite{fermanian2023new, cass2024lecture}, and as a consequence related techniques have been implemented across diverse domains. Their applications span deep learning \cite{kidger2019deep,  cirone2024theoretical, issa2024non, barancikova2024sigdiffusions}, kernel methods \cite{salvi2021signature, lemercier2021distribution}, and quantitative finance \cite{AkyildirimTeichmann, arribas2020sigsdes, horvath2023optimal, pannier2024path, cirone2025rough}. 
\\
\\
Let $X: [0,T] \rightarrow \mathbb{R}^d$ be a Lipschitz continuous path and consider its real-valued coordinate functions $X^i: [0,T] \rightarrow \mathbb{R}$. The study of signatures deals with the iterated integals of these coordinate functions. Let $w_1, \dots, w_n \in \{1, \dots, d\}$ be indices and consider the corresponding \emph{word} $w= (w_1, \dots, w_n)$ given through the associated $n$-tuple. Then the iterated integral of $X$ associated to the word $w$ over the subinterval $[s,t] \subseteq [0,T]$ is given by
$$ \int_{\Delta_{[s,t]}^n} dX_r^w := \int_s^t \int_s^{r_n} \cdots \int_s^{r_3} \int_s^{r_2} dX_{r_1}^{w_1} dX_{r_2}^{w_2} \cdots dX_{r_n}^{w_n}, $$
where we denote the set of possible $n$-step interval partitions by
$$ \Delta_{[s,t]}^n :=  \{ s \leq r_1 \leq r_2 \leq \dots \leq r_n \leq t \} \quad \textrm{for each } n \in \mathbb{N}. $$
We denote by $\mathscr{W}_n := \{ w = (w_1, \dots, w_n) ~|~ w_1, \dots, w_n \in \{1, \dots, d\}\}$ the collection of words of a fixed length $n$, and set for each $w = (w_1, \dots, w_n) \in \mathscr{W}_n$
$$ e_w := e_{w_1} \otimes e_{w_2} \otimes \dots \otimes e_{w_n}, $$
where $e_{w_i}$ is the standard basis vector in $\mathbb{R}^d$ in direction $w_i$ for each $i \in \{1, \dots, n\}$. In this case, we can consider the collection of all $n$'th order iterated integrals as an element in the $n$'th order tensor space $(\mathbb{R}^d)^{\otimes n}$, given by
$$ \int_{\Delta_{[s,t]}^n} dX_{r_1} \otimes \cdots \otimes dX_{r_n} = \sum_{w \in \mathscr{W}_n} \int_{\Delta_{[s,t]}^n} dX_r^w~ e_w, $$
where we have considered the canonical tensor product on $\mathbb{R}^d$. The collection of all finite order iterated integrals of $X$ is referred to as the \emph{signature} of $X$ (with the convention that the zero'th order iterated integral is just 1). 
\begin{defin}[Signature]
The signature $S(X)$ of the path $X$ is a two-parameter function $S(X): \Delta_{[0,T]}^2 \rightarrow T((\mathbb{R}^d))$ given by
$$ S(X)_{s,t} := 1 + \sum_{n=1}^{\infty} \int_{\Delta_{[s,t]}^n} dX_{r_1} \otimes \dots \otimes dX_{r_n}, $$
for every $(s,t) \in \Delta_{[0,T]}^2$, where $T((\mathbb{R}^d))$ denotes the extended tensor algebra on $\mathbb{R}^d$. We refer to $\int_{\Delta_{[s,t]}^n} dX_{r_1} \otimes \cdots \otimes dX_{r_n}$ as the \emph{$n$'th order signature component of $X$}.
\end{defin}
Here, the extended tensor algebra $T((\mathbb{R}^d))$ should not be confused with the usual tensor algebra $T(\mathbb{R}^d)$ which is the countable direct sum $ T(\mathbb{R}^d) = \bigoplus_{n=0}^{\infty} (\mathbb{R}^d)^{\otimes n} = \mathbb{R} \oplus \mathbb{R}^d \oplus (\mathbb{R}^d)^{\otimes 2} \oplus (\mathbb{R}^d)^{\otimes 3} \oplus \dots$. Instead, the \emph{extended} tensor algebra can be considered as the closure of $T(\mathbb{R}^d)$, and we can in particular express it by
$$ T((\mathbb{R}^d)) := \left\{ (x_0, x_1, x_2, \dots) = \sum_{n=0}^{\infty} x_n ~\Big{|}~ \forall n \in \mathbb{N}_0: ~ x_n \in (\mathbb{R}^d)^{\otimes n} \right\}.$$
The most common setting is that of level two rough paths, in which a meaningful notion of integration is introduced for paths of Hölder-regularity $\alpha \in \left( \frac{1}{3}, \frac{1}{2} \right]$ to which both the Riemann-Stieltjes and Young integration theories do not apply. This is done by equipping these paths with an extra piece of information, the \emph{enhancement}, mimicking the behaviour of a second order iterated integral.
More precisely, the idea goes as follows: consider first a path $X$ of Hölder-regularity $\alpha \in \left( \frac{1}{2}, 1 \right]$, an appropriate twice continuously differentiable function $f$ and a sequence of partitions
$$ \pi^n = \{ s = r_0^n < r_1^n < \dots < r_{N_n}^n = t \} $$
with mesh $|\pi^n| \rightarrow 0$. Applying a second order Taylor expansion to $f$ and using the usual approximation results for Young integrals, it is seen that
\begin{equation}
\label{eq:roughpathidea}
 \int_s^t f(X_r) dX_r = \lim_{n \rightarrow \infty} \sum_{i=0}^{N_n - 1} f(X_{r_i^n})(X_{r_{i+1}^n} - X_{r_i^n}) + Df(X_{r_i^n})\int_{r_i^n}^{r_{i+1}^n} \int_{r_i^n}^u dX_u \otimes dX_r.
\end{equation}
This gives us an expression for the Young integral which does not just depend on the usual approximation terms in the first part of the sum, but also on terms which are expressed through the second order iterated integral of $X$.\\
When the path $X$ is of lower Hölder regularity ($\alpha \in \left( \frac{1}{3}, \frac{1}{2} \right]$) and the usual approximation of the integral via the sum $\sum_{i=0}^{N_n - 1} f(X_{r_i^n}) (X_{r_{i+1}^n} - X_{r_i^n})$ does not work, the idea is then to equip $X$ with an enhancement $\mathbb{X}$ (setting $\mathbf{X} = (X, \mathbb{X})$), and letting $\mathbb{X}$ play the role of the second order iterated integral in \eqref{eq:roughpathidea}. This is done, first of all, by ensuring that $\mathbb{X}$ behaves similarly to a second order iterated integral (through \emph{Chen's relation}), and secondly by imposing sufficient regularity conditions on $\mathbb{X}$ to ensure that the limit $\lim_{n \rightarrow \infty} \sum_{i=0}^{N_n - 1} f(X_{r_i^n})(X_{r_{i+1}^n} - X_{r_i^n}) + Df(X_{r_i^n}) \mathbb{X}_{r_i^n, r_{i+1}^n}$ exists and is well-defined, in which case this limit is referred to as a \emph{rough integral} and is denoted by $\int_s^t f(X_r) d\mathbf{X}_r$.\\ 

For paths of even lower Hölder regularity, $\alpha \in \left( 0, \frac{1}{3} \right]$, a similar approach can be implemented, this time using a higher order Taylor expansion and thus equipping the path $X$ with higher order enhancements, say, $\mathbb{X}^2, \dots, \mathbb{X}^n$ (where $\mathbb{X}^2$ is the second order enhancement from before), with each $\mathbb{X}^i$ mimicking the behaviour of the $i$'th order iterated integral $\int_{\Delta_{[\cdot, \cdot]}^i} dX_{r_1} \otimes \cdots \otimes dX_{r_i} $. In this case, the level $n$ rough path will precisely be the collection $\mathbf{X} = (X, \mathbb{X}^2, \dots, \mathbb{X}^n)$, mimicking the signature components up to order $n$.\\

 From the initial study of rough paths, the relationship between a path and its signature has been further explored, and it has been shown that many of the geometric properties of a path $X$ can in most cases be recovered from its signature $S(X)$. In some sense, the signature can be seen to \emph{encode} the information of the full underlying path $X$, when evaluated in its terminal value, up to pieces of $X$ "backtracking" onto themselves \cite{HamblyLyons}. It is also shown in \cite{LyonsInversion} that if $X: [0,T] \rightarrow \mathbb{R}^d$ is a $\mathcal{C}^1$-path, then the path $X$ can approximately be reconstructed from the signature $S(X)_{0,T}$ by an inversion method - and it is possible to explicitly quantify the error of this approximation. Inversion of signatures is also studied in \cite{chang2019insertion} and \cite{lyons2017hyperbolic}. Results such as these suggest that it will often be useful to study the signature $S(X)_{0,T}$ if we want to understand the underlying path $X$, and indeed signature-based methods have gained traction for applications in various fields, as outlined earlier. For an introductory overview of some applications of signatures in machine learning, see also \cite{ChevyrevK}.
\subsection{Taylor expansions}
\label{subsec:TaylorExp}
\subsubsection{The $\mathbb{R}^N$-case}
Let $X: [0,T] \rightarrow \mathbb{R}^d$ be a Lipschitz continuous path, and let $V_1, \dots, V_d: \mathbb{R}^N \rightarrow \mathbb{R}^N$ be smooth vector fields. We consider for initial values $y \in \mathbb{R}^N$, controlled differential equations (CDEs) of the form
\begin{equation}
\label{eq:RNDE}
 Y_t = y + \sum_{i=1}^d \int_0^t V_i(Y_s) dX_s^i,
\end{equation}
and denote by $Y^y: [0,T] \rightarrow \mathbb{R}^N$ the unique solution associated to the initial value $y$ (supposing that the vector fields are chosen in a way such that a unique solution exists).\\

Let $g: \mathbb{R}^N \rightarrow \mathbb{R}$ be a smooth function, such that $g \in C^{\infty}(\mathbb{R}^N)$ with the notation from Appendix \ref{appendix:vectorfields}. By Proposition \ref{prop:mapXf}, we can then for each $i \in \{1, \dots, d\}$ apply $g$ to the vector field $V_i: \mathbb{R}^N \rightarrow \mathbb{R}^N$, in order to obtain a map $V_ig \in C^{\infty}(\mathbb{R}^N)$, which evaluates to:
$$ V_ig(y) = \underbrace{\nabla g(y)^T}_{1 \times n} \underbrace{V_i(y)}_{n \times 1} \quad \textrm{for every } y \in \mathbb{R}^N. $$
As $V_i g \in C^{\infty}(\mathbb{R}^N)$ we can again apply a vector field $V_j$ to this function, and this can be done repeatedly as many times as wanted. For any word $w=(w_1, \dots, w_k) \in \mathscr{W}_k$ and $g \in C^{\infty}(\mathbb{R}^N)$, we can thus construct a well-defined map $V_wg := V_{w_1} \cdots V_{w_k} g  \in C^{\infty}(\mathbb{R}^N)$ by this procedure. Indeed, $V_w: C^{\infty}(\mathbb{R}^N) \rightarrow C^{\infty}(\mathbb{R}^N)$ is seen to be the operator obtained by composing the derivations of the associated vector fields in the iteration, i.e. 
$$V_w = \mathscr{D}_{V_{w_1}} \circ \dots \circ \mathscr{D}_{V_{w_n}},$$
using the notation from Appendix \ref{appendix:vectorfields}.\\

Using a Picard iteration, e.g., as in \cite{BaudoinZhang}, we now obtain the following Taylor expansion.
\begin{thm}
\label{thm:BaudoinZhang}
Let $Y: [0,T] \rightarrow \mathbb{R}^N$ denote the unique solution to the differential equation \eqref{eq:RNDE} with initial value $y \in \mathbb{R}^N$, and let $g \in C^{\infty}(\mathbb{R}^N)$. Then:
\begin{align*} 
g(Y_t) = g(y) + \sum_{k=1}^{\infty} \sum_{w \in \mathscr{W}_k} V_w g(y)\int_{\Delta_{[0,t]}^k} dX_r^w \quad \textrm{for every } t \in [0,T].
\end{align*}
\end{thm}
\begin{rmk}
\label{rmk:TaylorConv}
Convergence of the Taylor expansion in Theorem \ref{thm:BaudoinZhangLie} is not needed for our results in Section \ref{subsec:sigreconstruct}, as any remainder terms from the Picard iteration disappear when taking derivatives and evaluating in $r =0$ under our reparametrization method proposed in Section \ref{subsec:sigreconstruct}. As such, we will refrain from going into a lengthy discussion of any such convergence. For a discussion of the convergence radius of the Taylor expansion, see for instance \cite{BaudoinZhang}.
\demo
\end{rmk}
\subsubsection{The Lie group case}
A similar construction can be made when considering smooth vector fields $V_1, \dots, V_d: G \rightarrow TG$ on a finite-dimensional Lie group $G$. In this case, we consider for initial values $z \in G$, CDEs of the form
\begin{equation}
\label{eq:LieDE} 
Z_t = z + \sum_{i=1}^d \int_0^t V_i(Z_s) dX_s^i \quad \textrm{for every } t \in [0,T].
\end{equation}
If the vector fields are chosen appropriately, such that a unique solution exists, it is well-known that this solution takes value on the Lie group $G$. We denote by $Z^z: [0,T] \rightarrow G$ the unique solution associated to the initial value $z$. When $G$ is a matrix Lie group, and the vector fields $V_1, \dots, V_d$ are chosen to be linear maps, the solutions $(Z^z)_{z \in G}$ are referred to as \emph{path-developments} on the Lie group $G$, see for instance \cite{Lou}.\\

As previously, we can apply the vector fields $V_1, \dots, V_d$ repeatedly to any function $g \in C^{\infty}(G)$ using the construction from Proposition \ref{prop:mapXf}. This yields, following a Picard iteration similar to the one in \cite{BaudoinZhang}, a Taylor expansion as in the $\mathbb{R}^N$-case.
\begin{thm}
\label{thm:BaudoinZhangLie}
Let $Z: [0,T] \rightarrow G$ denote the unique solution to the differential equation \eqref{eq:LieDE} with initial value $z \in G$. Then it holds for any $g \in C^{\infty}(G)$ that
$$ Z_t = z + \sum_{k=1}^{\infty} \sum_{w \in \mathscr{W}_k} V_w g (z) \int_{\Delta_{[0,t]}^k} dX_r^w \quad \textrm{for every } t \in [0,T]. $$
\end{thm}
\begin{rmk}
    As in the $\mathbb{R}^N$-case, we have here slightly abused notation without taking into account the convergence of the Taylor expansion. As in the $\mathbb{R}^N$-case, our approach in Section \ref{subsec:sigreconstructLie} does not require the Taylor expansion to be convergent.
    \demo
\end{rmk}
\subsection{Randomized signatures with depth}
\label{subsec:RanSig}
Let $X: [0,T] \rightarrow \mathbb{R}^d$ be a Lipschitz continuous path, and let $\sigma: \mathbb{R} \rightarrow \mathbb{R}$ be a real analytic function with infinite radius of convergence, different from the null-map (see e.g., \cite{Krantz} for a detailed discussion of real analytic functions). We apply $\sigma$ component-wise to vectors, such that
$$ \sigma(x_1, \dots, x_N) = \left( \sigma(x_1), \dots, \sigma(x_N)\right) \quad \textrm{for every } x = (x_1, \dots, x_N) \in \mathbb{R}^N.$$
Suppose that $A_1, \dots, A_d$ are random matrices in $\mathbb{R}^{N \times N}$ and $b_1, \dots, b_d$ are random vectors in $\mathbb{R}^N$, generated in a way such that all entries are i.i.d. with distribution absolutely continuous with respect to the Lebesgue measure. Define for each $i \in \{1, \dots, d\}$, the random neural vector field $V_i: \mathbb{R}^N \rightarrow \mathbb{R}^N$ by:
    $$ V_i(x) = \sigma(A_i x + b_i) \quad \text{for every } x \in \mathbb{R}^N. $$
Consider for each $y \in \mathbb{R}^N$ the random differential equation from \eqref{eq:RNDE}:
$$ Y_t^y = y + \sum_{i=1}^d \int_0^t \sigma(A_i Y_s + b_i) ~dX_s^i \quad \text{for } t \in [0,T].$$
Then we call the collection of solutions $(Y^y)_{y \in \mathbb{R}^N}$ a \emph{randomized signature of $X$}. Randomized signatures were originally introduced in \cite{CuchieroTeichmann1} and \cite{CuchieroTeichmann2}, and have later been studied and applied in works such as \cite{AkyildirimTeichmann}, \cite{Gonon24}, \cite{CompagnoniTeichmann}, \cite{MucaCirone23}, and \cite{Gruber23}. \\

For our purposes here, the shift-vectors $b_1, \dots, b_d$ will not be of much use, and so we ignore them. Inspired by neural activation functions, we can of course add depth to the vector fields $V_1, \dots, V_d$. Let $\sigma$ and $A_1, \dots, A_d$ be given as before, and let $D_1, \dots, D_d$ be random \emph{diagonal} matrices in $\mathbb{R}^{N \times N}$ with i.i.d. entries having distribution absolutely continuous with respect to the Lebesgue measure. Then we can for each $i \in \{1, \dots, d\}$, associate a random \emph{depth two} neural vector field $V_i: \mathbb{R}^N \rightarrow \mathbb{R}^N$ given by:
\begin{equation}
\label{eq:VFDepth}
    V_i(x) = \sigma(A_i \sigma(D_ix)) \quad \text{for every } x \in \mathbb{R}^N.
\end{equation}
Considering as before the random differential equations
$$ Y_t^y =  y + \sum_{i=1}^d \int_0^t \sigma(A_i\sigma(D_iY_s)) ~dX_s^i \quad \text{for } t \in [0,T],$$
we refer to the collection of solutions $(Y^y)_{y \in \mathbb{R}^N}$ as a randomized signature of $X$ with depth two, or simply as a \emph{randomized signature with depth}.

\newpage
\section{Linear independence of iterated vector fields on $\mathbb{R}^N$}
\label{sec:linindep}

From now on consider $d$ fixed smooth vector fields $\{V_i : \bR^N \to \bR^N ~|~ i=1,\dots,d ~\}$, and let $\mathscr{W}$ be the set of words in the alphabet $\{1,\dots,d\}$. Our main results will rely on utilizing the Taylor expansions from Section \ref{subsec:TaylorExp} to relate the collection of solutions $(Y^y)_{y \in \mathbb{R}^N}$ of a CDE \eqref{eq:RNDE} to the iterated integrals (i.e., signature components) of the underlying path $X$. In order to do this, we will need to study the iterated vector fields $V_w: C^{\infty}(\mathbb{R}^N) \rightarrow C^{\infty}(\mathbb{R}^N)$ which arise in the Taylor expansions. In particular, we will want to show that these operators are \emph{linearly independent} for words of the same length. We note of course, that the iterated vector fields can be defined recursively on word length, by setting for $i \in \{1, \dots, d\}$, $x \in \mathbb{R}^N$, $w \in \mathscr{W}$, and $\phi \in C^{\infty}(\bR^N)$:
        \begin{gather*}
            V_{i}\phi(x) = (\nabla \phi(x))^T V_i(x), \quad 
            V_{wi}\phi(x) := V_{i}(V_w\phi)(x)= (\nabla (V_w\phi)(x))^T V_i(x),
        \end{gather*}
        with the convention that in the one-letter case (i.e., $w=i$ for some $i$), $V_i$ may refer to either a vector field $\mathbb{R}^N \rightarrow \mathbb{R}^N$ or an operator $C^{\infty}(\mathbb{R}^N) \rightarrow C^{\infty}(\mathbb{R}^N)$, depending on the context.\\

As it is clear from an attempt at writing the explicit formulation of $V_{w}$ where $|w| \geq 3$, the form of the $V_{w}$ quickly spirals out of control. 
Fortunately such terms present a rich combinatorial structure which can be leveraged to effectively handle the constituent parts.
The natural language for this is that of \emph{trees}, as is customary in the differential equations literature under the guise of \emph{Butcher} expansions \cite{mclachlan2017butcherseriesstoryrooted}, or in the related Branched Rough Paths of \cite{GUBINELLI}.
\subsection{Tree-like vector fields}
\label{subsec:TreeVF}
Start by recalling the notion of a \emph{labeled recursive tree of order $m$}: this is a connected graph $\tau$ with vertex set $\{1, \dots, m\}$, and precisely $m-1$ edges, which is built recursively by letting vertex 1 denote the root, and adding in step $l+1$ the vertex labeled $l+1$ to the existing tree by connecting it with an edge to one of the vertices $\{1, \dots, l\}$. In addition to the natural labeling of a recursive tree, we will add extra vertex labels corresponding to letters from the alphabet $\{1, \dots, d\}$. Given a recursive tree of order $m$, we will typically just highlight these letter-labels by denoting them by $(w_1, \dots, w_m) \in \mathscr{W}_m$ such that the letter $w_l$ is the label of vertex $l$ for each $l \in \{1, \dots, m\}$. We denote by $\bT$ be the set of letter-labeled recursive trees, and note that these are {planar} graphs meaning that the order of the branches at any vertex is ignored (\emph{cf.} Figure \ref{fig:planar_tree_ex}).
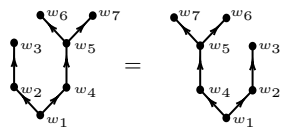
\begin{figure}[h!]
        \centering
        \input{Figures/planar_tree_ex2}
        \caption{Example of a planar rooted tree in $\bT$ with vertices $\{1, \dots, 7\}$ labeled by the word $w=(w_1, \dots, w_7) \in \mathscr{W}_7$.}
        \label{fig:planar_tree_ex}
\end{figure}
\ \\
Given trees $\tau_1,\dots,\tau_k \in \bT$, let $[\tau_1 \cdots \tau_k]_{\bullet_i}$ be the tree obtained by attaching the roots of $\tau_1,\dots,\tau_k$ to a new vertex $\bullet_i$; then every $\tau \in \bT$ can be constructed by starting with its leaves and recursively using the "graphical" operation $[\cdot]$ (only attaching vertices of higher index to lower index vertices, from the top down). For example the tree of Figure \ref{fig:planar_tree_ex} can be constructed as 
$[[\bullet_{w_3}]_{\bullet_{w_2}} ~ [[\bullet_{w_6} ~ \bullet_{w_7}]_{\bullet_{w_5}}]_{\bullet_{w_4}}]_{\bullet_{w_1}}$. For any given word $w \in \mathscr{W}$, we denote by $\mathbb{T}_w$ the collection of letter-labeled recursive trees of order $|w|$ with letter-labels given by the letters of the word $w$. Formally, we define it as follows.
\begin{defin}[Letter-labeled recursive trees]
    Define the following sets of recursive trees with vertices labeled by $\{0, 1, \dots, d\}$: 
    \begin{itemize}
        \item For $i \in \{1, \dots, d\}$ let $\bT_{i} := \{ \bullet_i \}$.
        \item For $i \in \{1, \dots, d\}$ and $w \in \mathscr{W}$ define $\bT_{wi}$ from $\bT_{w}$ by attaching, one by one, to the vertices of its trees a new edge having leaf $\bullet_i$ \emph{cf.} Figure \ref{fig:tree_spaces}.
    \end{itemize}
\end{defin}

\begin{figure}[h!]
        \centering
        \input{Figures/tree_spaces2}
        \caption{Left to right, the spaces $\bT_{w_1}$, $\bT_{w_1w_2}$, $\bT_{w_1w_2w_3}$, and $\bT_{w_1w_2w_3w_4}$. }
        \label{fig:tree_spaces}
\end{figure}
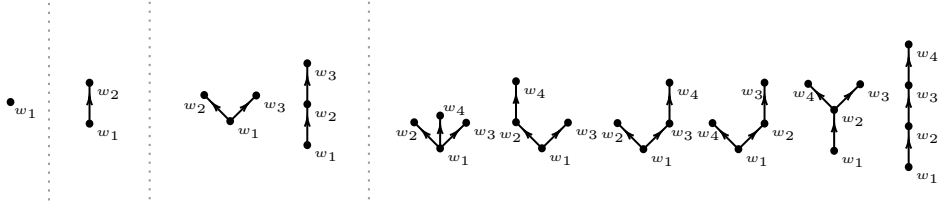
Denote by $\pi_1, \dots, \pi_N: \mathbb{R}^N \rightarrow \mathbb{R}$ the canonical coordinate projections, $\pi_i: x \mapsto x_i$.

\begin{defin}[Tree-like vector fields]
\label{def:TreeLikeVF}
    To each tree $\tau \in \bT$ we associate the vector field $V_{\tau}: \bR^N \to \bR^N$ defined recursively, for $i \in \{1, \dots, d\}$, and  $\tau_1, \dots, \tau_k \in \mathbb{T}$, as 
    \begin{align}
        V_{\bullet_i}(x) &= V_i(x),
        \\ \label{eq:TreeLikeVF}
        V_{[\tau_1 \cdots \tau_k]_{\bullet_i}}(x) &= \sum_{j_1,\dots, j_k=1}^N \frac{\partial^k}{\partial x_{j_k} \cdots \partial x_{j_1}}V_i(x) \cdot \pi_{j_1}(V_{\tau_1}(x)) \cdots \pi_{j_k}(V_{\tau_k}(x))        .
    \end{align}

    For each fixed $m \in \mathbb{N}$ we call $\{ V_{\tau} ~|~ \tau \in \mathbb{T}_w, w \in \mathscr{W}_m \}$ the \emph{collection of tree-like vector fields of order $m$}. Note how $V_{[\tau_1 \cdots \tau_k]_{\bullet_i}}(x) = ( d^k V_i )_x [V_{\tau_1}(x), \cdots, V_{\tau_k}(x)]$.
\end{defin}
Letting $\tau$ be given as in Figure \ref{fig:planar_tree_ex}, we see that the associated vector field is:
\begin{align*}
    V_{\tau}(x) = \sum_{j_1, \dots, j_6=1}^N& \frac{\partial^2}{\partial x_{j_3} \partial x_{j_1}} V_{w_1}(x) \cdot \pi_{j_1}\left(\frac{\partial}{\partial x_{j_2}} V_{w_2}(x)\right) \cdot \pi_{j_2}(V_{w_3}(x)) \cdot \pi_{j_3}\left(\frac{\partial}{\partial x_{j_4}}V_{w_3}(x)\right)\\
    &\quad \cdot \pi_{j_4}\left(\frac{\partial^2}{\partial x_{j_5} \partial x_{j_6}} V_{w_5}(x) \right) \cdot \pi_{j_5}(V_{w_6}(x)) \cdot \pi_{j_6}(V_{w_7}(x))
\end{align*}
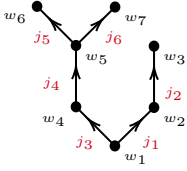
\begin{figure}[h!]
  \centering
  \input{Figures/treeSketch1}
  \caption{The planar rooted tree from Figure \ref{fig:planar_tree_ex} with assigned edge-directions $j_1, \dots, j_6 \in \{1, \dots, N\}$.}
\end{figure}
\\
Let $w \in \mathscr{W}_m$, let $\tau \in \mathbb{T}_m$ and let $\mathbf{j} = (j_1, \dots, j_{m-1}) \in \{1, \dots, N\}^{m-1}$ be a set of edge directions arising in the expression of $V_{\tau}$ obtained from \ref{eq:TreeLikeVF}. We will order the elements in $\mathbf{j}$ such that $j_k$ denotes the $k$'th edge added to the tree in its recursive construction. The edge directions $\mathbf{j}$ then have the following interpretation in the expression for $V_{\tau}$: Consider for some $l \in \{1, \dots, m\}$, the $l$'th vertex of the tree $\tau$, with associated label $w_l \in \{1, \dots, d\}$. Then:
\begin{itemize}
    \item The ingoing edge into vertex $l$ (label $w_l$) will always have direction $j_{l-1}$, and determines the direction of the coordinate projection $\pi_{j_{l-1}}$ associated to $V_{w_{l}}$ in the expression of $V_{\tau}$;
    \item The directions of the outgoing edges of a vertex $w_l$ determine the directions in which derivatives are taken with respect to $V_{w_l}$ in the expression of $V_{\tau}$.
\end{itemize}
\subsection{Linear independence of iterated vector fields}
We will use the association between trees and vector fields to gain an explicit expansion of the operators $V_w$. Here, the tree-like vector fields $V_{\tau}$ will come in handy, but first we will need to appropriately 'lift' them into actual operators on $C^{\infty}(\mathbb{R}^N)$. Let $w =(w_1, \dots, w_m) \in \mathscr{W}_m$ be a fixed word, and denote by $\mathbb{T}_w^0$ the collection of letter-labeled recursive trees of order $m+1$ with the root vertex having label $0$ and the remaining vertices having labels $(w_1, \dots, w_m)$. Let $\mathbb{T}^0 = \{ \tau ~|~ \tau \in \mathbb{T}_w^0 \text{ for some } w \in \mathscr{W}\}$ be the collection of all such letter-labeled recursive trees.
\begin{figure}[h!]
        \centering
        \input{Figures/tree_spaces}
        \caption{Left to right, the spaces $\bT_{w_1}^0$, $\bT_{(w_1,w_2)}^0$, $\bT_{(w_1, w_2, w_3)}^0$.
        }
        \label{fig:tree_spaces}
\end{figure}
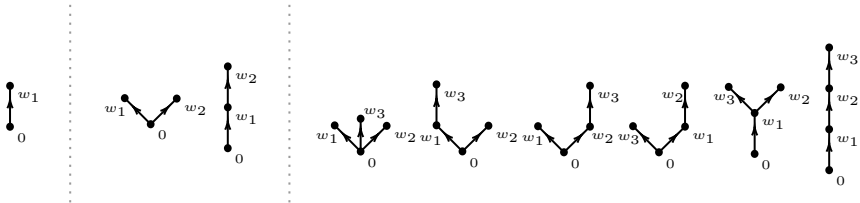
\begin{defin}[Tree-like operators]
Let $\tau^0 \in \mathbb{T}^0$. Then $\tau^0 = [\tau_1  \cdots  \tau_k]_{\bullet_0}$ for some $\tau_1, \dots, \tau_k \in \mathbb{T}$. We define an associated operator $V_{\tau^0}: C^{\infty}(\mathbb{R}^N) \rightarrow C^{\infty}(\mathbb{R}^N)$ by setting:
\begin{align*}
    V_{\bullet_0} \phi(x) &= \phi(x) \quad \text{and} \\ V_{[\tau_1 \cdots \tau_k]_{\bullet_0}}\phi(x) &= \sum_{j_0, \dots, j_{k-1} = 1}^N \frac{\partial^k}{\partial x_{j_{k-1}} \cdots \partial x_{j_0}} \phi(x) \cdot \pi_{j_0}(V_{\tau_1}(x)) \cdots \pi_{j_{k-1}}(V_{\tau_k}(x))
\end{align*}
for all $\phi \in C^{\infty}(\mathbb{R}^N)$ and $x \in \mathbb{R}^N$. Here, $V_{\tau_1}, \dots, V_{\tau_k}$ denote the tree-like vector fields obtained from Definition \ref{def:TreeLikeVF}.
\end{defin}
With the addition of a root vertex $0$, we see that if $\tau^0 \in \mathbb{T}^0$ is a letter-labeled recursive tree of order $m+1$ with labels $(0, w_1, \dots, w_m)$, then it can be expressed as a sum over edge directions $\mathbf{j} = (j_0, \dots, j_{m-1}) \in \{1, \dots, N\}^m$ where $j_1, \dots, j_{m-1}$ play the same role as before, and $j_0$ denotes the first edge going out of the root $0$. The following lemma allows us to express the operator $V_w$ as an expansion of tree-like operators associated to the word $w$.

\begin{lemma}
\label{lemma:Nicola1}
We can expand, for $w \in \mathscr{W}$, the operators $V_w$ in terms of tree-like operators as
\begin{equation}
    V_w = \sum_{\tau^0 \in \bT_w^0} V_{\tau^0}.
\end{equation}
\end{lemma}
\begin{proof}
    This follows from a simple application of the chain rule and an induction on word length.
\end{proof}

\begin{example}
As an example, we see that for $w_1, w_2 \in \{1, \dots, d\}$ and $\phi \in C^{\infty}(\mathbb{R}^N)$:
    $$
    V_{w_1}\phi(x) = (\nabla \phi(x))^T V_{w_1}(x) = \sum_{j_0=1}^N \frac{\partial}{\partial x_{j_0}}\phi(x) \cdot \pi_{j_0}(V_{w_1}(x)) = V_{[\bullet_{w_1}]_{\bullet_0}}\phi(x)$$
    and
    \begin{align*}
    V_{w_1w_2}\phi(x) &= (\nabla V_{w_1}\phi(x))^T V_{w_2}(x) \\
    &= \sum_{j_0,j_1=1}^N \frac{\partial^2}{\partial x_{j_1} \partial x_{j_0}} \phi(x) \cdot \pi_{j_0}(V_{w_1}(x)) \cdot \pi_{j_1}(V_{w_2}(x))\\ 
    &\qquad \qquad + \frac{\partial}{\partial x_{j_0}}\phi(x) \cdot \pi_{j_0}\left( \frac{\partial}{\partial x_{j_1}} V_{w_1}(x) \right) \cdot \pi_{j_1}(V_{w_2}(x)) \\
    &= V_{[\bullet_{w_1}~ \bullet_{w_2}]_{\bullet_0}}\phi(x) +  V_{[[\bullet_{w_2}]_{\bullet_{w_1}}]_{\bullet_0}}\phi(x)  
    \end{align*}
    \demo 
\end{example}
\begin{prop}
\label{prop:LinIndep}
    Fix integers $d, N, m \geq 1$. Let $V_1, \dots, V_d: \mathbb{R}^N \rightarrow \mathbb{R}^N$ be smooth vector fields, and suppose that the tree-like vector fields
    \begin{equation} 
    \label{eq:assumption}
    \{ V_{\tau}: \mathbb{R}^N \rightarrow \mathbb{R}^N ~|~ \tau \in \mathbb{T}_w, ~ w \in \mathscr{W}_m \}
    \end{equation}
    are linearly independent. Then it follows that the operators
    $$ \{V_w: C^{\infty}(\mathbb{R}^N) \rightarrow C^{\infty}(\mathbb{R}^N) ~|~ w \in \mathscr{W}_m \}$$
    are linearly independent.
\end{prop}
\begin{proof}
For any word $w \in \mathscr{W}_m$, we have by the "tree expansion" of $V_w$ from Lemma \ref{lemma:Nicola1} that:
$$ V_w = \sum_{\tau^0 \in \mathbb{T}_w^0} V_{\tau^0}$$
Start by observing that if $\tau^0 \in \mathbb{T}_w^0$, then the degree of the tree at the root (denoted by $\deg_{\tau^0}(0)$) determines the order of the differential operator applied to $\phi \in C^{\infty}(\mathbb{R}^N)$ in the expression of $V_{\tau^0}$. As the differential operators $\partial^k: C^{\infty}(\mathbb{R}^N) \rightarrow C^{\infty}(\mathbb{R}^N)$ for different $k \in \mathbb{N}$ are linearly independent, it will be sufficient for us to establish linear independence of the tree-like operators associated to the first order differential operator $\partial^1 : C^{\infty}(\mathbb{R}^N) \rightarrow C^{\infty}(\mathbb{R}^N)$. Set $\Pi := \{\tau^0 \in \mathbb{T}^0 ~|~ \tau^0 \in \mathbb{T}_w^0, ~ w \in \mathscr{W}_m, ~ \deg_{\tau^0}(0) = 1 \}$. Then the tree-like operators associated to $\partial^1$ are precisely given by the collection $\{V_{\tau^0} ~|~ \tau^0 \in \Pi \}$.
\begin{figure}[h!]
  \centering
  \includegraphics[width=0.7\textwidth]{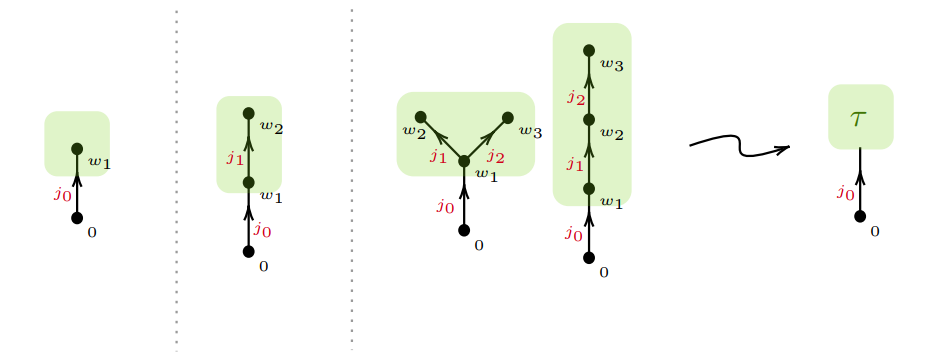}
  \caption{Form of the trees $\tau^0 \in \Pi$. }
  \label{fig:TreeSketch3}
\end{figure}
\ \\
By the recursive tree construction, the first order differential operator must always arise from first edge attached to the root, and thus have direction $j_0$. As no more edges can be added to the root without changing the order of the differential operator, a recursive tree is now built from the first vertex attached to the root. In other words, if $\tau^0 \in \mathbb{T}_w^0 \cap \Pi$, then $\tau^0 = [\tau]_{\bullet_0}$ for some $\tau \in \mathbb{T}_w$ which uniquely determines $\tau^0$, as seen on Figure \ref{fig:TreeSketch3}. We thus see that for all $\phi \in C^{\infty}(\mathbb{R}^N)$ and $x \in \mathbb{R}^N$:
$$ V_{\tau^0}\phi(x) = (\nabla \phi(x))^T V_{\tau}(x) = \sum_{j_0=1}^N \frac{\partial}{\partial x_{j_0}} \phi(x) \cdot \pi_{j_0}(V_{\tau}(x)) \quad \text{for all } \tau^0 \in \Pi,$$
and so linear independence of the vector fields $\{V_{\tau} ~|~ \tau \in \mathbb{T}_w, ~ w \in \mathscr{W}_m\}$ directly implies linear independence of the operators $\{V_{\tau^0} ~|~ \tau^0 \in \Pi\}$. As
$$ V_w = \underbrace{\sum_{\tau^0 \in \mathbb{T}_w^0 \cap \Pi} V_{\tau^0}}_{(1)} + \underbrace{\sum_{\tau^0 \in \mathbb{T}_w^0 \cap \Pi^c} V_{\tau^0}}_{(2)}$$
with all operators of type $(1)$ being linearly independent from all operators of type $(2)$, and all operators of type $(1)$ being linearly independent across different words $w \in \mathscr{W}_m$, we conclude that the operators $\{V_w ~|~ w \in \mathscr{W}_m\}$ must be linearly independent, as desired.
\end{proof}
In Section \ref{subsec:NestedExpo} we will introduce a class of vector fields, namely the ones obtained through \emph{nested exponentials}, which satisfy assumption \eqref{eq:assumption} when chosen such that $N \geq m-1$. The overall approach here should be generally applicable to other choices of vector fields as well. For every word $w \in \mathscr{W}_m$ and every associated letter-labeled recursive tree $\tau \in \mathbb{T}_w$, denote for each $\mathbf{j} = (j_1, \dots, j_{m-1}) \in \{1, \dots, N\}^{m-1}$ the associated term in the expression of $V_{\tau}$ from Definition \ref{def:TreeLikeVF} by $V_{\tau}^{\mathbf{j}}$, such that
\begin{equation} 
\label{eq:JExpansion}
V_{\tau} = \sum_{\mathbf{j} = (j_1, \dots, j_{m-1})} V_{\tau}^{\mathbf{j}}.
\end{equation}
The approach is then as follows:
\begin{enumerate}
    \item Take $N \geq m-1$, and consider the (non-empty) set of edge directions $\mathbf{j} = (j_1, \dots, j_{m-1}) \in \{1, \dots, N\}^{m-1}$ with $j_1, \dots, j_{m-1}$ distinct;
    \item Let $\tau \in \mathbb{T}_w^0$, and let $\mathbf{j} = (j_1, \dots, j_{m-1})$ have distinct directions. Then even if the word $w$ has the same letter appearing multiple times, the associated vertices can be distinguished in their vector-field form, since the in- and out-going edge directions will be distinct (corresponding to having different directions of the projections and derivatives of the associated vector fields);
    \item If the vector fields $V_1, \dots, V_d$ are chosen suitably, then the previous point should be enough to show that the vector fields
     $$ \left\{ V_{\tau}^{\mathbf{j}}: \mathbb{R}^N \rightarrow \mathbb{R}^N ~\Big{|}~ \begin{array}{l}
    \mathbf{j} = (j_1, \dots, j_{m-1}) \in \{1, \dots, N\}^{m-1} \text{ with } j_1, \dots, j_{m-1} \text{ distinct} \\
    \tau \in \mathbb{T}_w \text{ for some } w \in \mathscr{W}_m
    \end{array} \right\}$$
    are linearly independent among themselves, and indeed also linearly independent from 
    $$\left\{ V_{\tau}^{\mathbf{j}}: \mathbb{R}^N \rightarrow \mathbb{R}^N ~\Big{|}~ \begin{array}{l}
    \mathbf{j} = (j_1, \dots, j_{m-1}) \in \{1, \dots, N\}^{m-1} \text{ with } j_1, \dots, j_{m-1} \text{ not distinct} \\
    \tau \in \mathbb{T}_w \text{ for some } w \in \mathscr{W}_m
    \end{array} \right\}.$$
\end{enumerate}
The third point will then be sufficient to conclude that the vector fields $\{V_{\tau} ~|~ \tau \in \mathbb{T}_w, ~ w \in \mathscr{W}_m\}$ are linearly independent, using the expansion \eqref{eq:JExpansion} and a similar argument to the one used in the proof of Proposition \ref{prop:LinIndep}.
\begin{rmk}
\label{rmk:Dimensions}
    In general a dimensional requirement relating $N$ and $m$ is needed in order for the vector fields $V_1, \dots, V_d: \mathbb{R}^N \rightarrow \mathbb{R}^N$ to satisfy the assumption \eqref{eq:assumption}. Consider for instance the case where $m = 3$, $N = 1$ and $d = 3$. Then for $i,j,k \in \{1,2,3\}$ we have 
    \begin{align*}
        V_{kji} =& ( V_i'' V_j V_k +  V_i' V_j' V_k) \partial_x  
        + ( 2V_i' V_j V_k + V_i V_j' V_k) \partial^2_x
    + (V_i V_j V_k) \partial^3_x
    \end{align*}
    and the non-trivial relationship
    \[
    V_{123}+V_{231}+V_{312} = V_{132}+V_{213}+V_{321} \, ,
    \]
    so $\{V_w: C^{\infty}(\mathbb{R}) \rightarrow C^{\infty}(\mathbb{R}) ~|~ w \in \mathscr{W}_3\}$ is \emph{not} linearly independent (no matter our choice of vector fields $V_1, \dots, V_d: \mathbb{R} \rightarrow \mathbb{R}$).\\
    
    The relationship needed between $N$ and $m$ to avoid breaking the possibility of linear independence between the operators $\{V_w: C^{\infty}(\mathbb{R}^N) \rightarrow C^{\infty}(\mathbb{R}^N) ~|~ w \in \mathscr{W}_m\}$ is yet to be exactly determined - moreover, this relationship should also to some extent be dependent on the dimension $d$. For instance, if $d = 1$ then linear independence can easily be obtained for all $N, m \in \mathbb{N}$ (given suitable choices of vector fields). As noted earlier, we shall see in Section \ref{subsec:NestedExpo} that the assumption $N \geq m-1$ (for any choice of $d \in \mathbb{N}$) will also be sufficient to construct a working example where the operators $\{V_w ~|~ w \in \mathscr{W}_m\}$ become linearly independent.
    \demo
\end{rmk}
\subsection{Example: Nested exponentials}
\label{subsec:NestedExpo}
First, recall the definition of algebraic independence over the field of rational numbers $\mathbb{Q}$. For a more general algebraic definition, see e.g., \cite{Morandi}[Def. 19.1].
\begin{defin}[Algebraic independence over $\mathbb{Q}$]
Let $\alpha_1, \dots, \alpha_n \in \mathbb{R}$. We say that the set $\{\alpha_1, \dots, \alpha_n\}$ is \emph{algebraically independent over $\mathbb{Q}$} if
$$\forall p \in \mathbb{Q}[x_1, \dots, x_n] \setminus \{0\}: \quad p(\alpha_1, \dots, \alpha_n) \neq 0,$$
i.e., if $\alpha_1, \dots, \alpha_n$ cannot be made to cancel out through any polynomial equation with coefficients in $\mathbb{Q}$, other than of course the trivial zero polynomial $p \equiv 0$.\\

More generally, a set $A \subseteq \mathbb{R}$ is said to be \emph{algebraically independent over $\mathbb{Q}$}, if any finite subset of $A$ is algebraically independent over $\mathbb{Q}$. If a set is not algebraically independent over $\mathbb{Q}$, we say that it is \emph{algebraically dependent (over $\mathbb{Q}$)}.
\end{defin}
We will start by examining the random neural depth two vector fields given by \eqref{eq:VFDepth}, and then show that when $N \geq m-1$ and $\sigma(x) = \exp(x)$, these vector fields indeed satisfy the assumption of linear independence between tree-like vector fields of order $m$ from Proposition \ref{prop:LinIndep}. We note that if the random matrices $A_1, \dots, A_d$ and $D_1, \dots, D_d$ have entries which are i.i.d. with distribution absolutely continuous with respect to the Lebesgue measure, then all entries must almost surely be algebraically independent. It will thus be sufficient for us to study the setting where $A_i, D_i \in \mathbb{R}^{N \times N}$ are deterministic, and given by:
$$ A_i = \left( \begin{array}{cccc}
\alpha_{11}(i) & \alpha_{12}(i) & \dots & \alpha_{1N}(i) \\
\alpha_{21}(i) & \alpha_{22}(i) & \dots & \alpha_{2N}(i) \\
\vdots & \vdots & \ddots & \vdots \\
\alpha_{N1}(i) & \alpha_{N2}(i) & \dots & \alpha_{NN}(i)
\end{array}\right) \quad \text{and} \quad D_i = \left( \begin{array}{cccc}
d_1(i) & 0 & \dots & 0 \\
0 & d_2(i) & \dots & 0 \\
\vdots & \vdots & \ddots & \vdots \\
0 & 0 & \dots & d_N(i)
\end{array}\right),$$
where the coefficients $\{ \alpha_{jk}(i) ~|~ i \in \{1, \dots, d\}, j,k \in \{1, \dots, N\}\} \cup \{d_j(i) ~|~ i \in \{1, \dots, d\}, j \in \{1, \dots, N\} \}$ are algebraically independent over $\mathbb{Q}$.
\subsubsection{Neural depth-two vector fields}
\label{subsub:depth-two}
Start by letting $\sigma: \mathbb{R} \rightarrow \mathbb{R}$ be a real analytic function with infinite radius of convergence, and define the vector fields $V_1, \dots, V_d: \mathbb{R}^N \rightarrow \mathbb{R}^N$ by:
$$ V_i(x) = \sigma(A_i \sigma(D_i x)) \quad \text{for } x \in \mathbb{R}^N.$$
As we know that the tree-like vector fields $\{V_{\tau} ~|~ \tau \in \mathbb{T}\}$ will be products of different order partial derivatives (and coordinate projections) of the vector fields $V_1, \dots, V_d$, we first make a brief observation about the form of these partial derivatives. Suppose that $N \geq m-1$, such that the edge directions $j_1, \dots, j_{m-1} \in \{1, \dots, N\}$ can be chosen to be distinct. Then we have for all $k \in \{1, \dots, N\}$, $i \in \{1, \dots, d\}$, and $j_1, \dots, j_{m-1} \in \{1, \dots, N\}$ distinct, that:
\begin{align*}
    \pi_k(V_i)(x) &= \sigma(\pi_k(A_i \sigma(D_ix))), \\
    \partial_{j_1} \pi_k(V_i)(x) &= \sigma'(\pi_k(A_i \sigma(D_i x))) \alpha_{kj_1} \sigma'(d_{j_1}(i) x_{j_1}), \\
    \frac{\partial^2}{\partial j_2 \partial j_1} \pi_k(V_i)(x) &= \sigma''(\pi_k(A_i\sigma(D_ix))) \alpha_{kj_1} \alpha_{kj_2} \sigma'(d_{j_1}(i)x_{j_1}) \sigma'(d_{j_2}(i)x_{j_2}), \\
    \vdots \quad &\qquad \vdots \\
    \frac{\partial^{(m-1)}}{\partial j_{m-1} \cdots \partial j_1} \pi_k(V_i)(x) &= \sigma^{(m-1)}(\pi_k(A_i\sigma(D_ix))) \prod_{l=1}^{m-1} \alpha_{kj_l} \sigma'(d_{j_l}(i) x_{j_l})
\end{align*}
Take some word $w=(w_1, \dots, w_m)$, some tree $\tau \in \mathbb{T}_w$, and a tuple of \emph{distinct} edge directions $\mathbf{j} = (j_1, \dots, j_{m-1}) \in \{1, \dots, N\}^{m-1}$. Then as the tree $\tau$ has $m-1$ edges, there will be $m-1$ partial derivatives taken in the expression of $V_{\tau, \mathbf{j}}$, and each of these will be taken in a different direction associated to one of the entries of $\mathbf{j}$. Denote by $\mathbf{i} = (i_1, \dots, i_{m-1})$ the indices of the vertices (i.e., the indices of the letters in $w \in \mathscr{W}_m$) with respect to which the derivative in each direction $(j_1, \dots, j_{m-1})$ is taken (see Figure \ref{fig:TreeSketch4}). Then, as we have by the recursive construction of trees that the edge with direction $j_1$ must necessarily be attached to the root vertex $w_1$, the edge with direction $j_2$ must be attached to one of the vertices $\{w_1, w_2\}$, the edge with direction $j_3$ must be attached to one of the vertices $\{w_1, w_2, w_3\}$, and so on, we must have that the vertex indices $\mathbf{i} = (i_1, \dots, i_{m-1})$ must be contained in the collection:
\begin{equation}
\label{eq:derivvector}
    \{(i_1, \dots, i_{m-1}) ~|~ i_1 \in \{1\}, i_2 \in \{1,2\}, i_3 \in \{1,2,3\}, \dots, i_{m-1} \in \{1, \dots, m-1\} \}.
\end{equation}
\begin{figure}[h!]
  \centering
  \input{Figures/treeSketch4}
  \caption{Illustration of how the indices $\mathbf{i} = (i_1, \dots, i_{m-1})$ are obtained in a specific recursive tree with $m = 5$.}
  \label{fig:TreeSketch4}
\end{figure}
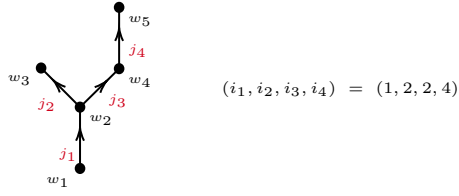
\ \\
As each element in the collection \eqref{eq:derivvector} precisely determines the structure of a recursive tree (as it explicitly pinpoints the vertices to which each edge in the recursive construction is attached), the cardinality of the set $\eqref{eq:derivvector}$ is exactly the number of recursive trees with $m$ vertices. In particular a tree $\tau \in \mathbb{T}_w$, $w \in \mathscr{W}_m$ can be uniquely represented by the pair $(\mathbf{i}, w)$ (with the possible add-on of edge directions $\mathbf{j}$).\\

Suppose that we are given a pair $(\mathbf{i}, w) = ((i_1, \dots, i_{m-1}), (w_1, \dots, w_m))$ uniquely determining a letter-labeled recursive tree $\tau \in \mathbb{T}_w$ of order $m$, and let $\mathbf{j} = (j_1, \dots, j_{m-1})$ be distinct edge directions (supposing $N \geq m-1$). 
\begin{itemize}
    \item We denote for each vertex index $i \in \{1, \dots, m\}$ by $n_i$ the number of times the index appears in the $(m-1)$-tuple $\mathbf{i}=(i_1, \dots, i_{m-1})$. Then $n_i$ is precisely the number of out-going edges of the vertex with index $i$ in the tree $\tau$, and so $\sum_{i=1}^{m-1} n_i = m-1$. In the example on Figure \ref{fig:TreeSketch4} we would have $(n_1, n_2, n_3, n_4, n_5) = (1, 2, 0, 1, 0)$.

    \item We split the edge directions of $\mathbf{j}$ into tuples associated to each vertex, such that $\mathbf{j}^{(1)}, \dots, \mathbf{j}^{(m)}$ denote the edge directions going out of vertices $w_1, \dots, w_m$, respectively (with $\mathbf{j}^{(m)} = ()$ by construction). In the example on Figure \ref{fig:TreeSketch4} this would correspond to having $\mathbf{j}^{(1)} = (j_1)$, $\mathbf{j}^{(2)} = (j_2, j_3)$, $\mathbf{j}^{(3)} = ()$, $\mathbf{j}^{(4)} = (j_4)$, $\mathbf{j}^{(5)} = ()$, and by construction we naturally have that $|\mathbf{j}^{(i)}| = n_i$.
\end{itemize}
For every index $i \in \{1, \dots, m\}$ we can then write $\frac{\partial^{n_i}}{\partial x_{\mathbf{j}^{(i)}}}$ for the $n_i$'th order partial derivative operator taken in the directions determined by $\mathbf{j}^{(i)}$. If $n_i = 0$ (and so $\mathbf{j}^{(i)} = ()$) we let $\frac{\partial^{n_i}}{\partial x_{\mathbf{j}^{(i)}}}$ be the identity operator. By the recursive construction of $V_{\tau}$ from Definition \ref{def:TreeLikeVF} and the form of the partial derivatives from earlier, we now see that for any $j_0 \in \{1, \dots, N\}$:
\begin{align}
\label{eq:TreeVFExpression}
\begin{split}
    \pi_{j_0}(V_{\tau}^{\mathbf{j}}) &= \prod_{i=1}^m \frac{\partial^{n_i}}{\partial x_{\mathbf{j}^{(i)}}} \pi_{j_{i-1}}(V_{w_i}) \\
    &= \prod_{i=1}^m \sigma^{(n_i)}(\pi_{j_{i-1}}(A_{w_i} \sigma(D_{w_i}x))) \prod_{j \in \mathbf{j}^{(i)}} \alpha_{j_{i-1}j} \sigma'(d_j(w_i)x_j)\\
    &=  \sigma^{{\color{blue}(n_1)}}(\pi_{j_0}(A_{w_1}\sigma(D_{w_1}x))) \sigma^{{\color{blue}(n_2)}}(\pi_{j_1}(A_{w_2}\sigma(D_{w_2}x))) \cdots \sigma^{{\color{blue}(n_m)}}(\pi_{j_{m-1}}(A_{w_m}\sigma(D_{w_m}x)))\\
    &\qquad \times \alpha_{j_{i_1 -1}j_1}({\color{blue}w_{i_1}}) \alpha_{j_{i_2 - 1}j_2}({\color{blue}w_{i_2}}) \cdots \alpha_{j_{i_{m-1}-1}j_{m-1}}({\color{blue}w_{i_{m-1}}})\\
    &\qquad \times \sigma'(d_{j_1}({\color{blue}w_{i_1}})x_{j_1}) \sigma'(d_{j_2}({\color{blue}w_{i_2}})x_{j_2}) \cdots \sigma'(d_{j_{m-1}}({\color{blue}w_{i_{m-1}}})x_{j_{m-1}})
    \end{split}
\end{align}
\subsubsection{Nested exponentials}
Suppose now that $\sigma(x) = e^x$, such that the vector fields $V_1, \dots, V_d: \mathbb{R}^N \rightarrow \mathbb{R}^N$ are given by
$$ V_i(x) = \exp(A_i \exp(D_i x)) \quad \text{for } x \in \mathbb{R}^N. $$
Then the first line in \eqref{eq:TreeVFExpression} will be the same for all $\tau \in \mathbb{T}_w$ (given fixed distinct edge directions $\mathbf{j}$), since $\sigma^{(n)}(x) = \exp(x) = \sigma(x)$ for all $n \in \mathbb{N}_0$. In particular, we get that:
\begin{align*}
    \pi_{j_0}(V_{\tau}^{\mathbf{j}}) &=  \exp\left( \sum_{i=1}^ m \pi_{j_{i-1}}(A_{w_i}\exp(D_{w_i}x))\right) \\
    &\qquad \times \alpha_{j_{i_1 -1}j_1}({\color{blue}w_{i_1}}) \alpha_{j_{i_2 - 1}j_2}({\color{blue}w_{i_2}}) \cdots \alpha_{j_{i_{m-1}-1}j_{m-1}}({\color{blue}w_{i_{m-1}}})\\
    &\qquad \times \exp\left( \sum_{n=1}^{m-1} d_{j_n}({\color{blue}w_{i_n}})x_{j_n} \right).
\end{align*}
The important parts of the expression here are the exponential factors (especially the one on the last line), as we do not care particularly about the coefficients in front when dealing with linear independence. As the constants $\{d_j(i) ~|~ j \in \{1, \dots, N\}, i \in \{1, \dots, d\}\} \cup \{\alpha_{kj}(i)~|~ k,j \in \{1, \dots, N\}, i \in \{1, \dots, d\}\}$ are algebraically independent over $\mathbb{Q}$, we see that each pair of $(m-1)$-tuples $(\mathbf{i}, \mathbf{j}) = ((i_1, \dots, i_{m-1}), (j_1, \dots, j_{m-1}))$ with $j_1, \dots, j_{m-1}$ distinct, will yield a \emph{unique} pair, consisting of the sum $\sum_{i=1}^m \pi_{j_{i-1}}(A_{w_i} \exp(D_{w_i}x)) = \sum_{j=1}^N \sum_{i=1}^m \alpha_{j_{i-1}j}(w_i) \exp(d_j(w_i)x_j)$ and a linear combination $\sum_{n=1}^{m-1} d_{j_n}(w_{i_n}) x_{j_n}$. As the functions $\{e^{c_1x_1 + \dots + c_{m-1} x_{m-1}} ~|~ c_1, \dots, c_{m-1} \in \mathbb{R}\}$ and $$\{ e^{\alpha_1 e^{\beta_1x_j} + \dots + \alpha_{(m-1)} e^{\beta_{m-1}x_j}} ~|~ \alpha_{1}, \dots, \alpha_{(m-1)}, \beta_{1}, \dots, \beta_{(m-1)}\in \mathbb{R} \}$$ are linearly independent for all $j \in \{1, \dots, N\}$, this implies that the above maps are indeed linearly independent - and indeed they are also linearly independent from any maps for which the edge directions $\mathbf{j} = (j_1, \dots, j_{m-1})$ are \emph{not} distinct. In other words, whenever $N \geq m-1$, the maps
$$ \left\{ \pi_{j_0}(V_{\tau}^{\mathbf{j}}): \mathbb{R}^N \rightarrow \mathbb{R} ~\Big{|}~ \begin{array}{l}
\tau \in \mathbb{T}_w, ~ w \in \mathscr{W}_m, ~ j_0 \in \{1, \dots, N\}\\
\mathbf{j} = (j_1, \dots, j_{m-1}) \in \{1, \dots, N\}^{m-1} \text{ with distinct entries}
\end{array} \right\}$$
are linearly independent. To briefly see that these maps are indeed also linearly independent from any $\pi_{j_0}(V_{\tau}^{\mathbf{j}})$ with $\mathbf{j}$ \emph{not} distinct, fix $\mathbf{j} = (j_1, \dots, j_{m-1})$ with \emph{distinct} entries and $\mathbf{j}' = (j_1', \dots, j_{m-1}')$ with \emph{non-distinct} entries. Then as $\mathbf{j}'$ will have at least one repeat of entries, we must necessarily have that there exists an entry, say $j^*$ in $\mathbf{j}$, such that $j^*$ does \emph{not} appear anywhere in $\mathbf{j}'$, i.e., $j^* \neq j_n'$ for all $n \in \{1, \dots, m-1\}$. As the vector fields $V_{\tau}^{\mathbf{j}'}$ associated to $\mathbf{j}'$ will thus \emph{never} take derivatives in direction $j^*$, there is no way of obtaining a factor of the type $\exp\left( d_{j^*}(w_{i^*}) x_{j^*} \right)$ in the expression of $V_{\tau}^{\mathbf{j}'}$ (with $i^*$ denoting whatever index is associated to $j^*$ in $\mathbf{i} = (i_1, \dots, i_{m-1})$). Utilizing the linear independence of exponential functions from before, this gives the desired.
\\

Removing the initial coordinate projection yields of course also that the vector fields
$$ \left\{ V_{\tau}^{\mathbf{j}}: \mathbb{R}^N \rightarrow \mathbb{R}^N ~\Big{|}~ \begin{array}{l}
\tau \in \mathbb{T}_w, ~ w \in \mathscr{W}_m\\
\mathbf{j} = (j_1, \dots, j_{m-1}) \in \{1, \dots, N\}^{m-1} \text{ with distinct entries}
\end{array} \right\}$$
are linearly independent among themselves, and also linearly independent from the collection of vector fields
$$ \left\{ V_{\tau}^{\mathbf{j}}: \mathbb{R}^N \rightarrow \mathbb{R}^N ~\Big{|}~ \begin{array}{l}
\tau \in \mathbb{T}_w, ~ w \in \mathscr{W}_m\\
\mathbf{j} = (j_1, \dots, j_{m-1}) \in \{1, \dots, N\}^{m-1} \text{ with non-distinct entries}
\end{array} \right\}.$$
Utilizing finally the fact that
$$ V_{\tau} = \sum_{\mathbf{j} = (j_1, \dots, j_{m-1})} V_{\tau}^{\mathbf{j}},$$
we conclude that the collection $\{V_{\tau} ~|~ \tau \in \mathbb{T}_w, ~ w \in \mathscr{W}_m\}$ is linearly independent when $N \geq m-1$ and $\sigma(x) = e^x$. This implies by Proposition \ref{prop:LinIndep} that the iterated vector fields $\{V_w ~|~ w \in \mathscr{W}_m\}$ must also be linearly independent for every fixed $m \leq N+1$.
\begin{rmk}[The general depth-two case]
    Note here that we have only specified $\sigma(x) = \exp(x)$ in order to make quick and direct conclusions about linear independence, utilizing well-known results for exponential functions. It should be possible to extend this to many other suitable choices of real analytic activation functions (for instance chosen \emph{generically}, i.e., such that all coefficients in the power series expansion are algebraically independent over $\mathbb{Q}$).
    \demo
\end{rmk}
\begin{rmk}[The no-depth case]
\label{rmk:NoDepth}
    One might wonder whether the strategy outlined here would not also work in the case where the vector fields $V_1, \dots, V_d: \mathbb{R}^N \rightarrow \mathbb{R}^N$ are chosen \emph{without depth}, i.e., in the classical 'randomized signature form':
    $$ V_i (x) = \sigma(A_i x) \quad \text{for } x \in \mathbb{R}^N.$$
    In this case, however, one sees that for $j, k \in \{1, \dots, N\}$ the partial derivatives will be of the form:
    $$ \partial_j \pi_k(V_i)(x) = \sigma'(\pi_k(A_ix)) \alpha_{kj}(i),$$
    i.e., the only way in which the directions of the derivatives is recorded is through the coefficients $\alpha_{kj}$ appearing out front. These are of course not very helpful when establishing linear independence, and indeed it can be checked that the tree-like vector fields $\{ V_{\tau}^{\mathbf{j}} ~|~ \tau \in \mathbb{T}_w, ~ w \in \mathscr{W}_m, ~ \mathbf{j} \text{ distinct} \}$ are \underline{not} linearly independent in this case.\\
    
    To see this, take for instance two words in $\mathscr{W}_m$:
    $$ w = (w_1, \dots, w_m) \quad \text{and} \quad w' = (w_1, w_{m-1}, w_{m-2}, \dots, w_3, w_2, w_m),$$
    and let $\tau \in \mathbb{T}_w$ and $\tau' \in \mathbb{T}_{w'}$ denote the fixed recursive trees corresponding to a ladder with letter-labels given by each of these words. I.e., the structure of $\tau$ and $\tau'$ is determined uniquely by the indices $(i_1, i_2, \dots, i_{m-1}) = (1, 2, \dots, m-1)$ and $(n_1, n_2, \dots, n_{m-1}, n_m) = (1,1, \dots, 1,0)$. Set 
    $$\mathbf{j} = (1,2,3,4,\dots, m-1) \quad \text{and} \quad \mathbf{j}' = (m-2, m-3, \dots, 2, 1, m-1).$$ 
    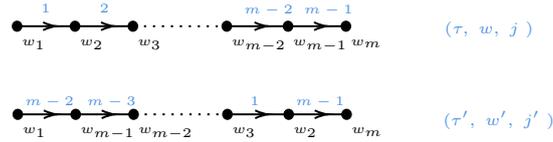
\begin{figure}[h!]
  \centering
  \input{Figures/TreeSketch5}
  \caption{The 'ladder'-like trees $(\tau, w, \mathbf{j})$ and $(\tau', w', \mathbf{j}')$ }
  \label{fig:TreeSketch5}
\end{figure}
\ \\
    Then:
    \begin{align*}
        V_{\tau}^{\mathbf{j}} &= c_{\tau, \mathbf{j}} \sigma'(A_{w_1}x) \sigma'(\pi_1(A_{w_2}x)) \cdots \sigma'(\pi_{m-2}(A_{w_{m-1}}x))\sigma(\pi_{m-1}(A_{w_m}x)) \quad \text{and}\\
        V_{\tau'}^{\mathbf{j}'} &= c_{\tau', \mathbf{j}'} \sigma'(A_{w_1}x) \sigma'(\pi_{m-2}(A_{w_{m-1}}x)) \cdots \sigma'(\pi_1(A_{w_2}x)) \sigma(\pi_{m-1}(A_{w_m}x)),
    \end{align*}
    and so $V_{\tau}^{\mathbf{j}} = \frac{c_{\tau, \mathbf{j}}}{c_{\tau', \mathbf{j}'}} V_{\tau'}^{\mathbf{j'}}$, showing that the two tree-like vector fields are \emph{linearly dependent}, even though they are associated to different words and have distinct edge directions.
    \demo
\end{rmk}
\newpage
\section{Signature reconstruction}
\label{sec:SigRec}

%
%


\subsection{Signature reconstruction on $\mathbb{R}^N$}
\label{subsec:sigreconstruct}
Consider the CDE \eqref{eq:RNDE} from the $\mathbb{R}^N$-case. It was first suggested in \cite{AkyildirimTeichmann}
that the information obtained by considering such CDEs across all different initial values $y \in \mathbb{R}^N$ given a specified choice of random neural vector fields $V_1, \dots, V_d$ would allow for the reconstruction of the components of the signature $S(X)$ almost surely. We provide here a detailed proof of this fact for general vector fields satisfying an assumption of linear independence between the associated tree-like vector fields up to a certain order $L$, but note that the reconstruction is then only possible up to the $L$'th order signature component. We note also that the assumption of linear independence between the associated tree-like vector fields implicitly depends on the dimension of the hidden space $\mathbb{R}^N$ of the CDE (see Remark \ref{rmk:Dimensions}).
\begin{thm}[Signature reconstruction on $\mathbb{R}^N$]
\label{thm:sigreconRN}
Fix some $N, L \in \mathbb{N}$, let $V_1, \dots, V_d: \mathbb{R}^N \rightarrow \mathbb{R}^N$ be smooth vector fields, and suppose that for all fixed $m \leq L$, the associated tree-like vector fields
$$ \{ V_{\tau} ~|~ \tau \in \mathbb{T}_w, ~ w \in \mathscr{W}_m\}$$
are linearly independent. Consider for each $y \in \mathbb{R}^N$, the CDE:
\begin{equation}
\label{eq:SigDE}
Y_t = y + \sum_{i=1}^d \int_0^t V_i(Y_s) dX_s^i \quad \text{for every } t \in [0,T]
\end{equation}
from \eqref{eq:RNDE}, and denote the solution associated to the initial value $y$ by $Y^{y}: [0,T]\rightarrow \mathbb{R}^N$. Then the signature components of $S(X)$ up to order $L$ can be uniquely reconstructed from the collection of solutions $(Y^y)_{y \in \mathbb{R}^N}$
\end{thm}
\begin{proof}
We start by slightly modifying the differential equation \eqref{eq:SigDE}. In particular, we let $\eta \in \mathbb{R}^N$ and $r \in \mathbb{R}$ and consider the differential equation
$$ Y_t^{\eta,r}  = \eta + \sum_{i=1}^d \int_0^t r V_i(Y_s^{\eta,r}) dX_s^i \quad \textrm{for } t \in [0,T]. $$
We immediately see that for every $\eta \in \mathbb{R}^N$ and $r \neq 0$ it holds that $Y_t^{\eta, r} = rY_t^{\frac{\eta}{r}}$, so each solution $Y^{\eta, r}$ (for $r \neq 0$) can indeed be obtained from the collection of solutions $(Y^y)_{y \in \mathbb{R}^N}$ to the differential equation \eqref{eq:SigDE} under appropriate scalar-multiplication with $r$. We of course also have the trivial correspondence $Y^{\eta, 0} \equiv \eta$. \\

Now, let $g \in C^{\infty}(\mathbb{R}^N)$, and define the $r$-scaled vector fields $\tilde{V}_1, \dots, \tilde{V}_d: \mathbb{R}^N \rightarrow \mathbb{R}^N$ by
$$\tilde{V}_i(x) = r V_i(x) \quad \textrm{for every } x \in \mathbb{R}^N. $$
Then it is easily verified for any $w = (w_1, \dots, w_k) \in \mathscr{W}_k$ that $ \tilde{V}_w g = r^k  V_w g$, so in particular Theorem \ref{thm:BaudoinZhang} gives the Taylor expansion
$$ g(Y_t^{\eta, r}) = g(\eta) + \sum_{k=1}^{\infty} r^k \sum_{w \in \mathscr{W}_k} V_w g(\eta) \int_{\Delta_{[0,t]}^k} dX_r^w\quad \textrm{for every } t \in [0,T].$$
For any $m\in \mathbb{N}$ we can thus take the $m$'th derivative with respect to $r$, and evaluate in $r =0$, in order to obtain that
$$ \frac{d^m}{dr^m}(g(Y_t^{\eta, r}))|_{r=0} = \sum_{w \in \mathscr{W}_m} V_w g(\eta) \int_{\Delta_{[0,t]}^m} dX_r^w \quad \textrm{for every } t \in [0,T]. $$
We now obtain from Proposition \ref{prop:LinIndep} that whenever $m \leq L$, the operators $\{V_w ~|~ w \in \mathscr{W}_m\}$ are linearly independent, due to our assumption of linear independence of the tree-like vector fields up to order $L$. Hence, choosing some appropriate distinct initial values $\eta_1, \dots, \eta_{d^m}$ (noting that $|\mathscr{W}_m| = d^m$), will by linear independence allow us to consider the above expression as a system of linear equations which is uniquely solvable with respect to the iterated integrals of $X$. Solving this linear equation system will then allow us to express all elements of the $m$'th signature component of $X$, in terms of the solutions $Y^{\eta_1, r}, \dots, Y^{\eta_{d^m}, r}$ for $r \in \mathbb{R}$, concluding the proof.
\end{proof}
The result can of course, among other examples, be applied to the randomized nested exponential vector fields from Section \ref{subsec:NestedExpo}, i.e., the vector fields given by
$$ V_i(x) = \exp(A_i \exp(D_i x)) \quad \text{for } x \in \mathbb{R}^N,$$
with the matrices $A_1, \dots, A_d$ and diagonal matrices $D_1, \dots, D_d$ being randomly generated by sampling the entries independently according to a distribution which is absolutely continuous with respect to the Lebesgue measure. In this case Theorem \ref{thm:sigreconRN} states that if $N \geq L-1$, then the signature components up to order $L$ of a path $X$ can be almost surely uniquely reconstructed from the randomized signatures of $X$ with depth two and exponential activation function.
\subsection{Signature reconstruction on Lie groups}
\label{subsec:sigreconstructLie}
Consider the CDE \eqref{eq:LieDE} from the Lie group case. We wish to extend the signature reconstruction result from Theorem \ref{thm:sigreconRN} to this setting. Letting $G$ be a Lie group of dimension $N$, and letting $V_1, \dots, V_d: G \rightarrow TG$ be smooth vector fields, Proposition \ref{prop:LinIndep} (linear independence of tree-like vector fields) should be possible to extend to the Lie group setting. In this case we can utilize the Taylor expansion from Theorem \ref{thm:BaudoinZhangLie} as before.
\begin{thm}[Signature reconstruction on $G$]
\label{thm:sigreconLie}
Fix some $N,L \in \mathbb{N}$, let $V_1, \dots, V_d: G \rightarrow TG$ be smooth vector fields, and suppose that for all fixed $m \leq L$, the associated tree-like vector fields
$$ \{ V_{\tau}: G \rightarrow TG ~|~ \tau \in \mathbb{T}_w, ~ w \in \mathscr{W}_m \}$$
are linearly independent. Consider for each $\zeta \in G$ and $r \in \mathbb{R}$, the $r$-scaled CDE
\begin{equation} 
Z_t^{\zeta, r} = \zeta + \sum_{i=1}^d \int_0^t rV_i(Z_s^{\zeta, r}) dX_s^i \quad \text{for every } t \in [0,T]
\end{equation}
from \eqref{eq:LieDE}. Then the signature components of $S(X)$ up to order $L$ can be uniquely reconstructed from the collection of solutions $(Z^{\zeta, r})_{\zeta \in G, r \in \mathbb{R}}$.
\end{thm}
\begin{rmk}
It is sufficient to consider $r \in (-\varepsilon, \varepsilon)$ for some $\varepsilon > 0$, as we are only interested in choices of $r$ close to zero.
\demo
\end{rmk}
\begin{proof}
Write $\tilde{V}_i(x) = rV_i(x)$ for every $i \in \{1, \dots, d\}$ and $x \in G$, and consider some $g \in C^{\infty}(G)$ non-constant. Then it can easily be verified for any word $w = (w_1, \dots, w_k) \in \mathscr{W}_k$ that $ \tilde{V}_w g = r^k V_wg$, so in particular we get the Taylor expansion
\begin{align*}
g(Z_t^{\zeta, r}) = g(\zeta) + \sum_{k=1}^{\infty} r^k \sum_{w \in \mathscr{W}_k} V_w g(\zeta) \int_{\Delta_{[0,t]}^k} dX_r^w \quad \textrm{for every } t \in [0,T],
\end{align*}
from Theorem \ref{thm:BaudoinZhangLie}. For any $m \in \mathbb{N}$ we can thus  take the $m$'th derivative with respect to $r$ and evaluate in $r=0$, in order to obtain that
$$ \frac{d^m}{dr^m} (g(Z_t^{\zeta, r}))|_{r=0} = \sum_{w \in \mathscr{W}_m} V_w g(\zeta) \int_{\Delta_{[0,t]}^k} dX_r^w \quad \textrm{for every } t \in [0,T].$$
Under a similar claim of linear independence of the operators $\{V_w ~|~ w \in \mathscr{W}_m\}$ whenever $m \leq L$ as in Proposition \ref{prop:LinIndep}, we can choose appropriate distinct initial values $\zeta_1, \dots, \zeta_{d^m}$, which allow us to consider the above expression as a linear equation system which is uniquely solvable with respect to the iterated integrals of $X$. Solving this linear equation system will allow us to express all elements of the $m$'th signature component of $X$, in terms of the solutions $Z^{\zeta_1, r}, \dots, Z^{\zeta_{d^m}, r}$ for $r \in \mathbb{R}$ under $g$. Hence all signature components of $S(X)$ can be reconstructed from the solutions $(Z^{\zeta, r})_{\zeta \in G, r \in \mathbb{R}}$, as desired. The claim of linear independence should be provable in a similar fashion to the $\mathbb{R}^N$-case covered in Section \ref{sec:linindep} - we will however not go into the details of this here.
\end{proof}
\newpage
\newpage
\section{Further perspectives}

The findings and discussions presented in this paper lay the groundwork for future exploration in the understanding of \emph{randomized signatures}, of their expressive power and stability. Some promising directions for future research and practical applications can be identified:

\begin{itemize}
    \item \textbf{Log-Signature:} In this work, our primary goal was reconstructing the signature. A significant challenge we encountered stemmed from inherent, unavoidable algebraic relations between generic vector fields (\emph{cf.} Remark \ref{rmk:Dimensions}), which fundamentally arise from the properties of the algebra of words.
    These relations are closely tied to those of the signature entries themselves, which are not independent but exhibit similar algebraic dependencies. It is known (\cite{LyonsRoughDE, morrill2021neural}) that such dependencies can be eliminated by shifting the focus from the signature to the \emph{log-signature}. This raises a natural question about the recovery of this lossless compression.
    Would such an approach resolve the algebraic complexities encountered so far, or would it introduce new challenges arising from the Lie-algebraic nature of the log-signature? We leave this question open for further investigation.

    \item \textbf{The Role of Depth:} We have demonstrated that the notion of depth in vector fields—analogous to the depth of neural networks—serves as a powerful tool in obtaining explicit results. This depth provides a “leveled” structure, reminiscent of that found in tensor algebras.
    It seems plausible that greater depth corresponds to a higher degree of independence, thereby enabling the recovery of more signature terms for a fixed hidden dimension. We believe that a thorough understanding of this phenomenon would be of significant interest and value for future research.

    \item \textbf{Alternative approaches to signature reconstruction:} 
    It might be possible to approach the question of signature reconstruction from an entirely different direction than the one applied in this project, by avoiding the use of Taylor expansions, and thus circumventing the challenges of linear independence between vector fields. Some alternative approaches to signature reconstruction that were explored in the making of this paper, but which were ultimately abandoned, were approaches using Baire's category theorem or probabilistic arguments. A different line of inquiry might be looking into more direct constructions using Fourier coefficients.

\end{itemize}
\newpage
\section{Acknowledgements}
MG and NMC are supported by the EPSRC Centre for Doctoral
Training in Mathematics of Random Systems: Analysis, Modelling and Simulation (EP/S023925/1).\\

The authors would like to thank Cristopher Salvi for his feedback on an earlier version of this work, and for suggesting the investigation of randomized signatures with depth. The authors would also like to thank Thomas Mikosch for facilitating the joint master thesis of MG between University of Copenhagen and ETH Zürich, on which parts of this paper are loosely based.

\bibliographystyle{apalike}
\bibliography{references}
\newpage
\appendix
\section{Appendix: Vector fields}
\label{appendix:vectorfields}
\subsection{Tangent spaces and tangent mappings}
Let $M$ be a smooth manifold, such that $M$ is locally homeomorphic to $\mathbb{R}^n$ for some fixed $n \in \mathbb{N}$. Denote
$$ C^{\infty}(M) := \{ f: M \rightarrow \mathbb{R} ~|~ f \textrm{ is smooth} \}. $$ 
For any two maps $f,g \in C^{\infty}(M)$ we define addition, scalar multiplication and multiplication pointwise in $C^{\infty}(M)$ by
\begin{align*}
(f+g)(x) &= f(x) + g(x),\\
(\alpha f)(x) &= \alpha f(x), \\
(f \cdot g)(x) &= f(x)g(x),
\end{align*}
for all $x \in M$, $\alpha \in \mathbb{R}$. Equipped with these operations, $C^{\infty}(M)$ forms an algebra (for which $+$ and $\cdot$ are particularly commutative by construction).\\

For each $x \in M$ we associate a \emph{tangent space}
$$T_xM = \left\{ X_x: C^{\infty}(M) \rightarrow \mathbb{R} \textrm{ linear} ~\Big{|}~ \begin{array}{c}
X_x(f \cdot g) = X_x(f)g(x) + f(x) X_x(g), \\
X_x \textrm{ satisfies the germ condition}
\end{array} \right\}, $$
where the germ condition (referring to the formal definition of tangent spaces in terms of \emph{germs}) states that for all $f,g \in C^{\infty}(M)$:
$$\textrm{If } \exists U \subseteq M \textrm{ open neighbourhood of } x \textrm{ such that } f|_U = g|_U \textrm{ then } X_x(f) = X_x(g).$$
Suppose that $(U,u)$ is a chart around $x$, such that $U$ is open with $x \in U$, and $u: U \rightarrow u(U) \subseteq \mathbb{R}^n$ is a homeomorphism. Denoting the coordinate functions of $u$ by $u^1, \dots, u^n$, it is then known that the tangent space $T_xM$ has canonical basis vectors $\left( \frac{\partial}{\partial u^i} \right)_{i=1}^n$, so in particular $T_xM$ is linearly isomorphic to $\mathbb{R}^n$ for each $x \in M$.
\begin{defin}[Tangent mapping]
Let $M$ and $N$ be smooth manifolds, and let $\varphi: M \rightarrow N$ be a smooth mapping. For every $x \in M$ we define the \emph{tangent mapping induced by $\varphi$} as
$$ T_x\varphi: T_xM \rightarrow T_{\varphi(x)}N \quad \textrm{given by} \quad X_x \mapsto X_x(\cdot \circ \varphi). $$
\end{defin}
\begin{prop}[Properties of tangent mappings]
\label{prop:propertiestangent}
Let $M$, $N$ and $P$ be smooth manifolds and consider smooth mappings $\varphi: M \rightarrow N$ and $\psi: N \rightarrow P$. For every fixed $x \in M$ it holds that
\begin{itemize}
\item[\textbf{\emph{(i)}}] The tangent mapping $T_x\varphi: T_xM \rightarrow T_{\varphi(x)}N$ is linear;
\item[\textbf{\emph{(ii)}}] $T_x(\psi \circ \varphi) = T_{\varphi(x)}\psi \circ T_x\varphi$;
\item[\textbf{\emph{(iii)}}] $T_x\textrm{id}_M = \textrm{id}_{T_xM}$ where $\textrm{id}$ denotes the identity map on the indicated space.
\end{itemize}
\end{prop}
We especially notice that if
\begin{itemize}
\item $M$ is associated to $\mathbb{R}^m$ and $N$ is associated to $\mathbb{R}^n$;
\item $(U,u)$ is a chart on $M$ with $x \in U$ and $(V,v)$ is a chart on $N$ with $\varphi(x) \in V$;
\end{itemize}
we can consider the canonical bases $\left( \frac{\partial}{\partial u_i}|_x \right)_{i=1}^m$ for $T_xM$ and $\left(\frac{\partial}{\partial v_j}|_{\varphi(x)} \right)_{j=1}^n$ for $T_{\varphi(x)}N$. Hence, linearity of the tangent mapping $T_x\varphi: T_xM \rightarrow T_{\varphi(x)}N$ implies that it can be represented by a real-valued $n \times m$-matrix given with respect to the canonical bases.
\begin{prop}
Let $M$ and $N$ be smooth manifolds, fix some $x \in M$ and consider a smooth mapping $\varphi: M\rightarrow N$. Then the real-valued matrix representing the linear map $T_x\varphi: T_xM \rightarrow T_{\varphi(x)}N$ with respect to the canonical choice of basis vectors is precisely given by the Jacobi matrix
$$D(v \circ \varphi \circ u^{-1})(u(x)) \in \mathbb{R}^{n \times m}$$
for any choice of charts $(U,u)$ in $M$ and $(V,v)$ in $N$ with $x \in U$ and $\varphi(x) \in V$.\footnote{Note here that $v \circ \varphi \circ u^{-1}: \underbrace{u(U)}_{\subseteq \mathbb{R}^m} \rightarrow \mathbb{R}^n$ is a well-defined map between Euclidean spaces.}
\end{prop}
In particular this means that we obtain the linear combinations:
$$ T_x\varphi\left( \frac{\partial}{\partial u^i}|_x \right) = \sum_{j=1}^n \frac{\partial (v^j \circ \varphi \circ u^{-1})}{\partial x^i}(u(x)) \frac{\partial}{\partial v^j}|_{\varphi(x)} \quad \textrm{for } i=1, \dots, m. $$
\ \\
We denote by $TM = \{ (x,X_x) ~|~ x \in M, X_x \in T_xM\}$ the \emph{tangent bundle} on $M$. For any smooth mapping $\varphi: M \rightarrow N$ considered as above, we can associate the total mapping $T\varphi: TM \rightarrow TN$ given by $(x, X_x) \mapsto (\varphi(x), T_x\varphi(X_x))$.
\subsection{Vector fields and pushforwards}
\begin{defin}[Vector field]
A \emph{vector field} on $M$ is a map $X: M \rightarrow TM$ which associates to each $x \in M$ a tangent vector $X(x) \in T_xM$. I.e., the map $X$ is a vector field if it satisfies the identity
$$ \pi_M \circ X = \textrm{id}_M $$
where $\pi_M: TM \rightarrow M$ is the projection mapping $(x, X_x) \mapsto x$ and $\textrm{id}_M$ is the identity mapping on $M$.
\end{defin}
The set of all vector fields on $M$ is denoted
$$ \mathfrak{X}(M) := \{ X: M \rightarrow TM ~|~ X \textrm{ is a vector field}\}, $$
and $\mathfrak{X}(M)$ forms a vector space when equipped with pointwise addition and scalar multiplication.
\begin{prop}[The map $Xf$]
\label{prop:mapXf}
Let $X \in \mathfrak{X}(M)$ be a vector field and let $f \in C^{\infty}(M)$ be a smooth function. Then the map
$$ Xf: M \rightarrow \mathbb{R} \quad \textrm{given by} \quad Xf(x) := X(x)(f) $$
is smooth, i.e., $Xf \in C^{\infty}(M)$.
\end{prop}
For a vector field $X: M \rightarrow TM$ in $\mathfrak{X}(M)$, we define the map 
$$\mathcal{D}_X: C^{\infty}(M) \rightarrow C^{\infty}(M) \quad \textrm{given by} \quad f \mapsto Xf \quad \textrm{for all } x \in M, $$
and it can then be shown that $\mathscr{D}_X$ is a \emph{derivation}. Indeed, there is a unique correspondence between a vector field $X$ and its derivation, in the sense that the map $X \mapsto \mathscr{D}_X$ is bijective.\\

We will now define the notion of a \emph{pushforward} of a vector field.
\begin{defin}[$\varphi$-related vector fields]
Let $M$ and $N$ be smooth manifolds, and consider a smooth mapping $\varphi: M \rightarrow N$. If $X \in \mathfrak{X}(M)$ and $Y \in \mathfrak{X}(N)$ are vector fields on $M$ and $N$, respectively, we say that $X$ and $Y$ are \emph{$\varphi$-related} if
$$ T\varphi \circ X = Y \circ \varphi. $$
\end{defin}
\begin{prop}[Pushforward of a vector field]
\label{prop:pushforward}
Let $M$ and $N$ be smooth manifolds, and suppose that $\varphi: M \rightarrow N$ is a smooth diffeomorphism. Then for every $X \in \mathfrak{X}(M)$ there exists a unique vector field $Y \in \mathfrak{X}(N)$ such that $X$ and $Y$ are $\varphi$-related. We write $\varphi_*X := Y$ and call $\varphi_*X$ the \emph{pushforward} of $X$ by $\varphi$. In particular, the pushforward is explicitly given by
$$ \varphi_*X = T\varphi \circ X \circ \varphi^{-1}. $$
\end{prop}
\subsection{Velocity vectors and integral curves}
Let as usual $M$ be a smooth manifold and let $J \subseteq \mathbb{R}$ be some open interval. We will be interested in considering smooth curves $\gamma: J \rightarrow M$, and will typically assume that $0 \in J$, such that $\gamma(0)$ denotes the \emph{starting point} of the curve.
\begin{defin}[Velocity vector]
Consider some smooth curve $\gamma: J \rightarrow M$ and fix some time $t_0 \in J$. Then, considering the tangent mapping $T_{t_0} \gamma: T_{t_0} J \rightarrow T_{\gamma(t_0)} M$, we define the \emph{velocity vector} of $\gamma$ at $t_0$ as the tangent vector
$$ \gamma'(t_0) := T_{t_0} \gamma \left( \frac{d}{dt}|_{t_0} \right) \in T_{\gamma(t_0)} M. $$
\end{defin}
We note in particular that evaluating the tangent vector $\gamma'(t_0) \in T_{\gamma(t_0)} M$ in any function $f \in C^{\infty}(M)$ yields
$$ \gamma'(t_0)(f) = T_{t_0} \gamma \left( \frac{d}{dt}|_{t_0} \right)(f) = (f \circ \gamma)'(t_0).$$
In other words, the velocity vector can be considered as a derivation which maps every $f \in C^{\infty}(M)$ to its derivative along the curve $\gamma$.\\

One can now ask, whether we, given some vector field $X: M \rightarrow TM$ can construct a curve $\gamma: J \rightarrow M$ such that $X$ precisely maps every point of the curve to its corresponding velocity vector. Such a curve will be referred to as an \emph{integral curve}.
\begin{defin}[Integral curve]
Let $M$ be a smooth manifold and let $X \in \mathfrak{X}(M)$ be a vector field. We say that a smooth curve $\gamma: J \rightarrow M$ is an \emph{integral curve of $X$} if it satisfies that
$$ \gamma'(t) = X(\gamma(t)) \quad \textrm{for every } t \in J.$$
\end{defin}
\begin{thm}[Existence and uniqueness of integral curves]
Let $X \in \mathfrak{X}(M)$ be a vector field on the smooth manifold $M$ and fix some point $x \in M$. Then there exists an open interval $J \subseteq \mathbb{R}$ with $0 \in J$ and a smooth curve $\gamma: J \rightarrow M$ such that $\gamma$ is an integral curve of $X$ with starting point $\gamma(0) = x$. If we can choose $J = \mathbb{R}$, it furthermore holds that the curve $\gamma$ is unique.
\end{thm}
\begin{defin}[Maximal integral curve]
Let $X \in \mathfrak{X}(M)$ and let $\gamma: J \rightarrow \mathbb{R}$ be an integral curve of $X$. We say that $\gamma$ is a \emph{maximal integral curve of $X$} if it cannot be extended to a larger integral curve of $X$, i.e., if there exist no smooth curves $\tilde{\gamma}: \tilde{J} \rightarrow M$ satisfying that
\begin{enumerate}
    \item $\tilde{J} \subseteq \mathbb{R}$ is an open interval with $J \subseteq \tilde{J}$.
    \item $\tilde{\gamma}|_J = \gamma$.
    \item $\tilde{\gamma}$ is an integral curve of $X$.
\end{enumerate}
\end{defin}
\section{Lie groups and Lie algebras}
For this section it may be useful to refer to Appendix \ref{appendix:vectorfields} for conventions of notation. See \cite{Hall}, \cite{JohnLee}, \cite{Michor}, \cite{Samelson} and \cite{Warner}, for thorough introductions to Lie group theory.
\subsection{Basic definitions and relations}
\begin{defin}[Lie group]
Let $G$ be a smooth manifold, and let $\cdot: G \times G \rightarrow G$ be a smooth map on $G$. If the pair $(G, \cdot)$ constitutes a group, then we say that $(G, \cdot)$ is a \emph{Lie group}. We denote the identity element in $(G, \cdot)$ by $e$.
\end{defin}
Let $G$ be a Lie group. For any fixed $a \in G$, the left- and right-translation operators $L_a: G \rightarrow G$ and $R_a: G \rightarrow G$ are given by
$$ L_a(b) = a b \quad \textrm{and} \quad R_a(b) = b a \quad \textrm{for any } b \in G. $$
A map $\varphi: G \rightarrow H$ between two Lie groups $G$ and $H$ is called a \emph{Lie group homomorphism} if it is a smooth group homomorphism.
\begin{defin}[Lie algebra]
A vector space $\mathfrak{g}$, equipped with a bilinear operation $[\cdot, \cdot]: \mathfrak{g} \times \mathfrak{g} \rightarrow \mathfrak{g}$, satisfying the two properties
\begin{itemize}
\item[\textbf{\emph{(i)}}] $[X,Y] = - [Y,X]$;
\item[\textbf{\emph{(ii)}}] $[X, [Y,Z]] = [[X,Y], Z] + [Y, [X,Z]]$ (the Jacobi identity);
\end{itemize} 
for all $X,Y,Z \in \mathfrak{g}$ is called a \emph{Lie algebra}. The operation $[\cdot, \cdot]$ is called the \emph{Lie bracket} associated to the Lie algebra.
\end{defin}
\noindent As one would expect, a map $\phi: \mathfrak{g} \rightarrow \mathfrak{h}$ between two Lie algebras is called a \emph{Lie algebra homomorphism} if it is an algebra homomorphism between $\mathfrak{g}$ and $\mathfrak{h}$, i.e., if it is linear and satisfies the homomorphism relation
$$ \phi([X,Y]) = [\phi(X), \phi(Y)] \quad \textrm{for all } X,Y \in \mathfrak{g}.$$
\begin{prop}[Lie subalgebra]
Let $(\mathfrak{g}, [\cdot, \cdot])$ be a Lie algebra, and suppose that $\mathfrak{h} \subseteq \mathfrak{g}$ is a subspace. If $\mathfrak{h}$ is closed under the Lie bracket, i.e.,
$$ X,Y \in \mathfrak{h} \quad \Rightarrow \quad [X,Y] \in \mathfrak{h}, $$
then $(\mathfrak{h}, [\cdot, \cdot]|_{\mathfrak{h}})$ forms a Lie algebra, where $[\cdot, \cdot]|_{\mathfrak{h}}: \mathfrak{h} \times \mathfrak{h} \rightarrow \mathfrak{h}$ denotes the restriction of the Lie bracket to $\mathfrak{h}$. We call $(\mathfrak{h}, [\cdot, \cdot]|_{\mathfrak{h}})$ a \emph{Lie subalgebra} of $\mathfrak{g}$.
\end{prop}
Let $G$ be a Lie group. Then the space $\mathfrak{X}(G)$ of all vector fields on $G$ can be equipped with a canonical Lie bracket under which it forms a Lie algebra.
\begin{thm}[Lie bracket on $\mathfrak{X}(G)$]
The canonical choice of Lie bracket on $\mathfrak{X}(G)$ is the bilinear map $[\cdot, \cdot]: \mathfrak{X}(G) \times \mathfrak{X}(G) \rightarrow \mathfrak{X}(G)$ given by the representation 
$$ [X,Y](x) = X(x) \circ \mathcal{D}_Y - Y(x) \circ \mathcal{D}_X \quad \textrm{for all } x \in G \textrm{ and } X,Y \in \mathfrak{X}(G). $$
Equipped with this bracket, the pair $(\mathfrak{X}(G), [\cdot,\cdot])$ forms a Lie algebra. The maps $\mathcal{D}_X$ and $\mathcal{D}_Y$ are defined as in Appendix \ref{appendix:vectorfields}.
\end{thm}
Given a Lie group $(G, \cdot)$, we can now associate to it a canonical Lie algebra, namely, the Lie subalgebra of $(\mathfrak{X}(G), [\cdot, \cdot])$ consisting of all \emph{left-invariant} vector fields.
\begin{prop}[Left-invariant vector field]
Let $G$ be a Lie group, and let $X \in \mathfrak{X}(G)$ be a vector field on $G$. We say that the vector field $X$ is \emph{left-invariant} if
\begin{equation}
(L_a)_* X = TL_a \circ X \circ L_{a^{-1}} = X  \quad \textrm{for all } a \in G.
\end{equation}
We denote the set of left-invariant vector fields on $G$ by $\mathfrak{X}_L(G)$.
\end{prop}
By Proposition \ref{prop:pushforward} we see that a vector field is left-invariant if and only if it is $L_a$-related to itself for every $a \in G$. It can also easily be verified that the collection $\mathfrak{X}_L(G)$ of left-invariant vector fields on $G$ is a Lie subalgebra of $\mathfrak{X}(G)$. We shall refer to this Lie subalgebra as the \emph{Lie algebra generated by $G$}, and write $\textrm{Lie}(G) = \mathfrak{X}_L(G)$.
\begin{thm}
\label{thm:Liealgisomorph}
Let $G$ be a Lie group and denote by $\varepsilon: \textrm{Lie}(G) \rightarrow T_eG$ the evaluation map $X \mapsto X(e)$. Then $\varepsilon$ is a linear isomorphism between vector spaces, and we write
$$ \textrm{Lie}(G) \cong T_e G.$$
\end{thm}
From Theorem \ref{thm:Liealgisomorph} one immediately obtains a natural way of constructing a Lie algebra homomorphism from a Lie group homomorphism.
\begin{thm}[Induced Lie algebra homomorphism]
Let $\varphi: G \rightarrow H$ be a Lie group homomorphism between two Lie groups $G$ and $H$, and denote their Lie associated Lie algebras by $\mathfrak{g} = \textrm{Lie}(G)$ and $\mathfrak{h} = \textrm{Lie}(H)$, respectively. Then for every left-invariant vector field $X \in \mathfrak{g}$ there exists a unique $Y \in \mathfrak{h}$ such that $X$ and $Y$ are $\varphi$-related, and we write $\varphi_*X := Y$. In particular, the map $\varphi_*: \mathfrak{g} \rightarrow \mathfrak{h}$ obtained from this construction is a Lie algebra homomorphism.
\end{thm}
\subsection{The exponential map}
We start by introducing the notion of a \emph{one-parameter subgroup}.
\begin{defin}[One-parameter subgroup]
Let $G$ be a Lie group and note that the pair $(\mathbb{R}, +)$ naturally constitutes a Lie group as well. We will refer to any Lie group homomorphism $\gamma: \mathbb{R} \rightarrow G$ as a \emph{one-parameter subgroup of $G$}.
\end{defin}
\begin{thm}
Let $G$ be a Lie group with neutral element $e$. Then $\gamma: \mathbb{R} \rightarrow G$ is a one-parameter subgroup of $G$ if and only if there exists some left-invariant vector field $X \in \textrm{Lie}(G)$ such that $\gamma$ is the maximal integral curve of $X$ satisfying the initial condition $\gamma(0) = e$. In particular, we will say that $\gamma$ is the \emph{one-parameter subgroup generated by $X$}, and sometimes denote it by $\gamma_X$.
\end{thm}
In particular, we obtain the bijective correspondences
\begin{align*} 
\begin{array}{ccccc}
\{\textrm{one-parameter subgroups of } G\} & \longleftrightarrow & \textrm{Lie}(G) & \longleftrightarrow & T_eG, \\
\gamma_X & \leftrightarrow & X & \leftrightarrow & X(e).
\end{array}
\end{align*}
\begin{defin}[Exponential map]
Let $G$ be a Lie group with Lie algebra $\mathfrak{g} = \textrm{Lie}(G)$. Then the exponential map $\exp: \mathfrak{g} \rightarrow G$ is defined by
$$ \exp(X) = \gamma_X(1) \quad \textrm{for all } X \in \mathfrak{g},$$
where $\gamma_X$ denotes the one-parameter subgroup generated by $G$.
\end{defin}
The exponential map provides us with a canonical way of expressing the one-parameter subgroup associated to a left-invariant vector field.
\begin{prop}
Let $G$ be a Lie group. Then for every left-invariant vector field $X \in \textrm{Lie}(G)$, the one-parameter subgroup generated by $X$ is the curve $\gamma_X: \mathbb{R} \rightarrow G$ defined by
$$ \gamma_X(s) = \exp(sX) \quad \textrm{for every } s \in \mathbb{R}.$$
\end{prop}
We end this section by stating some useful properties of the exponential map.
\begin{prop}[Properties of the exponential map]
Let $G$ be a Lie group with Lie algebra $\mathfrak{g} = \textrm{Lie}(G)$. Then the exponential map $\exp: \mathfrak{g} \rightarrow G$ is smooth, and satisfies all of the following properties.
\begin{enumerate}
    \item For all $X \in \mathfrak{g}$: $\exp((s+t)X) = \exp(sX)\exp(tX)$ for all $s,t \in \mathbb{R}$-
    \item For all $X \in \mathfrak{g}$: $\exp(X)^{-1} = \exp(-X)$.
    \item For all $X \in \mathfrak{g}$: $\exp(X)^n = \exp(nX)$.
    \item The tangent mapping $T_0\exp: T_0\mathfrak{g} \rightarrow T_eG$ is the identity mapping.
    \item There exists an open neighbourhood $U \subseteq \mathfrak{g}$ with $0 \in U$ such that $\exp(U) \subseteq G$ is open with $e \in \exp(U)$, and
    $$ \exp|_U: U \rightarrow \exp(U) \quad \textrm{is a diffeomorphism}.$$
    \item If $\varphi: G \rightarrow H$ is a Lie group homomorphism into some other Lie group $H$ with Lie algebra $\mathfrak{h} = \textrm{Lie}(H)$, then it holds that
    $$ \exp \circ \varphi_* = \varphi \circ \exp \quad \textrm{as a map } \mathfrak{g} \rightarrow H,$$
    where $\varphi_*: \mathfrak{g} \rightarrow \mathfrak{h}$ denotes the Lie algebra homomorphism induced from $\varphi$.
\end{enumerate}
\end{prop}
\end{document}

%% file: Figures/planar_tree_ex2.tex
\tikzset{every picture/.style={line width=0.75pt}} 

\begin{tikzpicture}[x=0.75pt,y=0.75pt,yscale=-1,xscale=1]

\draw  [color={rgb, 255:red, 0; green, 0; blue, 0 }  ,draw opacity=1 ][fill={rgb, 255:red, 0; green, 0; blue, 0 }  ,fill opacity=1 ] (102.3,147.44) .. controls (102.3,146.53) and (102.94,145.79) .. (103.74,145.79) .. controls (104.53,145.79) and (105.18,146.53) .. (105.18,147.44) .. controls (105.18,148.35) and (104.53,149.09) .. (103.74,149.09) .. controls (102.94,149.09) and (102.3,148.35) .. (102.3,147.44) -- cycle ;
\draw    (90.69,133.28) -- (103.74,147.44) ;
\draw [shift={(94.91,137.86)}, rotate = 47.35] [color={rgb, 255:red, 0; green, 0; blue, 0 }  ][line width=0.75]    (4.37,-1.32) .. controls (2.78,-0.56) and (1.32,-0.12) .. (0,0) .. controls (1.32,0.12) and (2.78,0.56) .. (4.37,1.32)   ;
\draw  [color={rgb, 255:red, 0; green, 0; blue, 0 }  ,draw opacity=1 ][fill={rgb, 255:red, 0; green, 0; blue, 0 }  ,fill opacity=1 ] (89.25,133.28) .. controls (89.25,132.37) and (89.9,131.63) .. (90.69,131.63) .. controls (91.49,131.63) and (92.13,132.37) .. (92.13,133.28) .. controls (92.13,134.19) and (91.49,134.93) .. (90.69,134.93) .. controls (89.9,134.93) and (89.25,134.19) .. (89.25,133.28) -- cycle ;
\draw    (116.84,133.51) -- (103.74,147.44) ;
\draw [shift={(112.62,138)}, rotate = 133.27] [color={rgb, 255:red, 0; green, 0; blue, 0 }  ][line width=0.75]    (4.37,-1.32) .. controls (2.78,-0.56) and (1.32,-0.12) .. (0,0) .. controls (1.32,0.12) and (2.78,0.56) .. (4.37,1.32)   ;
\draw    (90.69,111.09) -- (90.69,133.28) ;
\draw [shift={(90.69,118.79)}, rotate = 90] [color={rgb, 255:red, 0; green, 0; blue, 0 }  ][line width=0.75]    (4.37,-1.32) .. controls (2.78,-0.56) and (1.32,-0.12) .. (0,0) .. controls (1.32,0.12) and (2.78,0.56) .. (4.37,1.32)   ;
\draw    (116.84,111.33) -- (116.84,133.51) ;
\draw [shift={(116.84,119.02)}, rotate = 90] [color={rgb, 255:red, 0; green, 0; blue, 0 }  ][line width=0.75]    (4.37,-1.32) .. controls (2.78,-0.56) and (1.32,-0.12) .. (0,0) .. controls (1.32,0.12) and (2.78,0.56) .. (4.37,1.32)   ;
\draw  [color={rgb, 255:red, 0; green, 0; blue, 0 }  ,draw opacity=1 ][fill={rgb, 255:red, 0; green, 0; blue, 0 }  ,fill opacity=1 ] (115.4,111.33) .. controls (115.4,110.42) and (116.05,109.68) .. (116.84,109.68) .. controls (117.64,109.68) and (118.28,110.42) .. (118.28,111.33) .. controls (118.28,112.24) and (117.64,112.98) .. (116.84,112.98) .. controls (116.05,112.98) and (115.4,112.24) .. (115.4,111.33) -- cycle ;
\draw    (103.8,97.17) -- (116.84,111.33) ;
\draw [shift={(108.02,101.75)}, rotate = 47.35] [color={rgb, 255:red, 0; green, 0; blue, 0 }  ][line width=0.75]    (4.37,-1.32) .. controls (2.78,-0.56) and (1.32,-0.12) .. (0,0) .. controls (1.32,0.12) and (2.78,0.56) .. (4.37,1.32)   ;
\draw    (129.95,97.4) -- (116.84,111.33) ;
\draw [shift={(125.73,101.89)}, rotate = 133.27] [color={rgb, 255:red, 0; green, 0; blue, 0 }  ][line width=0.75]    (4.37,-1.32) .. controls (2.78,-0.56) and (1.32,-0.12) .. (0,0) .. controls (1.32,0.12) and (2.78,0.56) .. (4.37,1.32)   ;
\draw  [color={rgb, 255:red, 0; green, 0; blue, 0 }  ,draw opacity=1 ][fill={rgb, 255:red, 0; green, 0; blue, 0 }  ,fill opacity=1 ] (115.4,133.51) .. controls (115.4,132.6) and (116.05,131.86) .. (116.84,131.86) .. controls (117.64,131.86) and (118.28,132.6) .. (118.28,133.51) .. controls (118.28,134.43) and (117.64,135.17) .. (116.84,135.17) .. controls (116.05,135.17) and (115.4,134.43) .. (115.4,133.51) -- cycle ;
\draw  [color={rgb, 255:red, 0; green, 0; blue, 0 }  ,draw opacity=1 ][fill={rgb, 255:red, 0; green, 0; blue, 0 }  ,fill opacity=1 ] (89.25,111.09) .. controls (89.25,110.18) and (89.9,109.44) .. (90.69,109.44) .. controls (91.49,109.44) and (92.13,110.18) .. (92.13,111.09) .. controls (92.13,112.01) and (91.49,112.75) .. (90.69,112.75) .. controls (89.9,112.75) and (89.25,112.01) .. (89.25,111.09) -- cycle ;
\draw  [color={rgb, 255:red, 0; green, 0; blue, 0 }  ,draw opacity=1 ][fill={rgb, 255:red, 0; green, 0; blue, 0 }  ,fill opacity=1 ] (102.36,97.17) .. controls (102.36,96.26) and (103,95.52) .. (103.8,95.52) .. controls (104.59,95.52) and (105.24,96.26) .. (105.24,97.17) .. controls (105.24,98.08) and (104.59,98.82) .. (103.8,98.82) .. controls (103,98.82) and (102.36,98.08) .. (102.36,97.17) -- cycle ;
\draw  [color={rgb, 255:red, 0; green, 0; blue, 0 }  ,draw opacity=1 ][fill={rgb, 255:red, 0; green, 0; blue, 0 }  ,fill opacity=1 ] (128.51,97.4) .. controls (128.51,96.49) and (129.16,95.75) .. (129.95,95.75) .. controls (130.75,95.75) and (131.39,96.49) .. (131.39,97.4) .. controls (131.39,98.32) and (130.75,99.06) .. (129.95,99.06) .. controls (129.16,99.06) and (128.51,98.32) .. (128.51,97.4) -- cycle ;
\draw  [color={rgb, 255:red, 0; green, 0; blue, 0 }  ,draw opacity=1 ][fill={rgb, 255:red, 0; green, 0; blue, 0 }  ,fill opacity=1 ] (195.18,148.87) .. controls (195.18,147.96) and (195.82,147.22) .. (196.62,147.22) .. controls (197.42,147.22) and (198.06,147.96) .. (198.06,148.87) .. controls (198.06,149.78) and (197.42,150.52) .. (196.62,150.52) .. controls (195.82,150.52) and (195.18,149.78) .. (195.18,148.87) -- cycle ;
\draw    (183.58,134.71) -- (196.62,148.87) ;
\draw [shift={(187.79,139.29)}, rotate = 47.35] [color={rgb, 255:red, 0; green, 0; blue, 0 }  ][line width=0.75]    (4.37,-1.32) .. controls (2.78,-0.56) and (1.32,-0.12) .. (0,0) .. controls (1.32,0.12) and (2.78,0.56) .. (4.37,1.32)   ;
\draw  [color={rgb, 255:red, 0; green, 0; blue, 0 }  ,draw opacity=1 ][fill={rgb, 255:red, 0; green, 0; blue, 0 }  ,fill opacity=1 ] (182.13,134.71) .. controls (182.13,133.8) and (182.78,133.06) .. (183.58,133.06) .. controls (184.37,133.06) and (185.02,133.8) .. (185.02,134.71) .. controls (185.02,135.62) and (184.37,136.36) .. (183.58,136.36) .. controls (182.78,136.36) and (182.13,135.62) .. (182.13,134.71) -- cycle ;
\draw    (209.73,134.95) -- (196.62,148.87) ;
\draw [shift={(205.5,139.43)}, rotate = 133.27] [color={rgb, 255:red, 0; green, 0; blue, 0 }  ][line width=0.75]    (4.37,-1.32) .. controls (2.78,-0.56) and (1.32,-0.12) .. (0,0) .. controls (1.32,0.12) and (2.78,0.56) .. (4.37,1.32)   ;
\draw    (183.58,112.52) -- (183.58,134.71) ;
\draw [shift={(183.58,120.22)}, rotate = 90] [color={rgb, 255:red, 0; green, 0; blue, 0 }  ][line width=0.75]    (4.37,-1.32) .. controls (2.78,-0.56) and (1.32,-0.12) .. (0,0) .. controls (1.32,0.12) and (2.78,0.56) .. (4.37,1.32)   ;
\draw    (209.73,112.76) -- (209.73,134.95) ;
\draw [shift={(209.73,120.45)}, rotate = 90] [color={rgb, 255:red, 0; green, 0; blue, 0 }  ][line width=0.75]    (4.37,-1.32) .. controls (2.78,-0.56) and (1.32,-0.12) .. (0,0) .. controls (1.32,0.12) and (2.78,0.56) .. (4.37,1.32)   ;
\draw  [color={rgb, 255:red, 0; green, 0; blue, 0 }  ,draw opacity=1 ][fill={rgb, 255:red, 0; green, 0; blue, 0 }  ,fill opacity=1 ] (208.29,112.76) .. controls (208.29,111.85) and (208.93,111.11) .. (209.73,111.11) .. controls (210.52,111.11) and (211.17,111.85) .. (211.17,112.76) .. controls (211.17,113.67) and (210.52,114.41) .. (209.73,114.41) .. controls (208.93,114.41) and (208.29,113.67) .. (208.29,112.76) -- cycle ;
\draw    (170.53,98.36) -- (183.58,112.52) ;
\draw [shift={(174.75,102.94)}, rotate = 47.35] [color={rgb, 255:red, 0; green, 0; blue, 0 }  ][line width=0.75]    (4.37,-1.32) .. controls (2.78,-0.56) and (1.32,-0.12) .. (0,0) .. controls (1.32,0.12) and (2.78,0.56) .. (4.37,1.32)   ;
\draw    (196.68,98.6) -- (183.58,112.52) ;
\draw [shift={(192.46,103.09)}, rotate = 133.27] [color={rgb, 255:red, 0; green, 0; blue, 0 }  ][line width=0.75]    (4.37,-1.32) .. controls (2.78,-0.56) and (1.32,-0.12) .. (0,0) .. controls (1.32,0.12) and (2.78,0.56) .. (4.37,1.32)   ;
\draw  [color={rgb, 255:red, 0; green, 0; blue, 0 }  ,draw opacity=1 ][fill={rgb, 255:red, 0; green, 0; blue, 0 }  ,fill opacity=1 ] (208.29,134.95) .. controls (208.29,134.03) and (208.93,133.29) .. (209.73,133.29) .. controls (210.52,133.29) and (211.17,134.03) .. (211.17,134.95) .. controls (211.17,135.86) and (210.52,136.6) .. (209.73,136.6) .. controls (208.93,136.6) and (208.29,135.86) .. (208.29,134.95) -- cycle ;
\draw  [color={rgb, 255:red, 0; green, 0; blue, 0 }  ,draw opacity=1 ][fill={rgb, 255:red, 0; green, 0; blue, 0 }  ,fill opacity=1 ] (182.13,112.52) .. controls (182.13,111.61) and (182.78,110.87) .. (183.58,110.87) .. controls (184.37,110.87) and (185.02,111.61) .. (185.02,112.52) .. controls (185.02,113.44) and (184.37,114.18) .. (183.58,114.18) .. controls (182.78,114.18) and (182.13,113.44) .. (182.13,112.52) -- cycle ;
\draw  [color={rgb, 255:red, 0; green, 0; blue, 0 }  ,draw opacity=1 ][fill={rgb, 255:red, 0; green, 0; blue, 0 }  ,fill opacity=1 ] (195.24,98.6) .. controls (195.24,97.69) and (195.89,96.95) .. (196.68,96.95) .. controls (197.48,96.95) and (198.12,97.69) .. (198.12,98.6) .. controls (198.12,99.51) and (197.48,100.25) .. (196.68,100.25) .. controls (195.89,100.25) and (195.24,99.51) .. (195.24,98.6) -- cycle ;
\draw  [color={rgb, 255:red, 0; green, 0; blue, 0 }  ,draw opacity=1 ][fill={rgb, 255:red, 0; green, 0; blue, 0 }  ,fill opacity=1 ] (169.09,98.36) .. controls (169.09,97.45) and (169.73,96.71) .. (170.53,96.71) .. controls (171.33,96.71) and (171.97,97.45) .. (171.97,98.36) .. controls (171.97,99.28) and (171.33,100.02) .. (170.53,100.02) .. controls (169.73,100.02) and (169.09,99.28) .. (169.09,98.36) -- cycle ;

\draw (105,93) node [anchor=north west][inner sep=0.75pt]  [font=\fontsize{0.24em}{0.28em}\selectfont]  {$w_6$};
\draw (131,93) node [anchor=north west][inner sep=0.75pt]  [font=\fontsize{0.24em}{0.28em}\selectfont]  {$w_7$};
\draw (119,132) node [anchor=north west][inner sep=0.75pt]  [font=\fontsize{0.24em}{0.28em}\selectfont]  {$w_4$};
\draw (105,147) node [anchor=north west][inner sep=0.75pt]  [font=\fontsize{0.24em}{0.28em}\selectfont]  {$w_1$};
\draw (92,132) node [anchor=north west][inner sep=0.75pt]  [font=\fontsize{0.24em}{0.28em}\selectfont]  {$w_2$};
\draw (92,110) node [anchor=north west][inner sep=0.75pt]  [font=\fontsize{0.24em}{0.28em}\selectfont]  {$w_3$};
\draw (198,147) node [anchor=north west][inner sep=0.75pt]  [font=\fontsize{0.24em}{0.28em}\selectfont]  {$w_1$};
\draw (211,132) node [anchor=north west][inner sep=0.75pt]  [font=\fontsize{0.24em}{0.28em}\selectfont]  {$w_2$};
\draw (211,110) node [anchor=north west][inner sep=0.75pt]  [font=\fontsize{0.24em}{0.28em}\selectfont]  {$w_3$};
\draw (119,110) node [anchor=north west][inner sep=0.75pt]  [font=\fontsize{0.24em}{0.28em}\selectfont]  {$w_5$};
\draw (186,132) node [anchor=north west][inner sep=0.75pt]  [font=\fontsize{0.24em}{0.28em}\selectfont]  {$w_4$};
\draw (186,110) node [anchor=north west][inner sep=0.75pt]  [font=\fontsize{0.24em}{0.28em}\selectfont]  {$w_5$};
\draw (171,93) node [anchor=north west][inner sep=0.75pt]  [font=\fontsize{0.24em}{0.28em}\selectfont]  {$w_7$};
\draw (197,93) node [anchor=north west][inner sep=0.75pt]  [font=\fontsize{0.24em}{0.28em}\selectfont]  {$w_6$};
\draw (144.11,119.81) node [anchor=north west][inner sep=0.75pt]  [font=\scriptsize]  {$=$};

\end{tikzpicture}

%% file: Figures/tree_spaces2.tex
\tikzset{every picture/.style={line width=0.75pt}} 

\begin{tikzpicture}[x=0.75pt,y=0.75pt,yscale=-1,xscale=1]

\draw  [color={rgb, 255:red, 0; green, 0; blue, 0 }  ,draw opacity=1 ][fill={rgb, 255:red, 0; green, 0; blue, 0 }  ,fill opacity=1 ] (58.8,90.96) .. controls (58.8,90.11) and (59.44,89.43) .. (60.23,89.43) .. controls (61.02,89.43) and (61.66,90.11) .. (61.66,90.96) .. controls (61.66,91.8) and (61.02,92.49) .. (60.23,92.49) .. controls (59.44,92.49) and (58.8,91.8) .. (58.8,90.96) -- cycle ;
\draw    (60.23,70.4) -- (60.23,90.96) ;
\draw [shift={(60.23,77.28)}, rotate = 90] [color={rgb, 255:red, 0; green, 0; blue, 0 }  ][line width=0.75]    (4.37,-1.32) .. controls (2.78,-0.56) and (1.32,-0.12) .. (0,0) .. controls (1.32,0.12) and (2.78,0.56) .. (4.37,1.32)   ;
\draw  [color={rgb, 255:red, 0; green, 0; blue, 0 }  ,draw opacity=1 ][fill={rgb, 255:red, 0; green, 0; blue, 0 }  ,fill opacity=1 ] (58.8,70.4) .. controls (58.8,69.55) and (59.44,68.87) .. (60.23,68.87) .. controls (61.02,68.87) and (61.66,69.55) .. (61.66,70.4) .. controls (61.66,71.24) and (61.02,71.93) .. (60.23,71.93) .. controls (59.44,71.93) and (58.8,71.24) .. (58.8,70.4) -- cycle ;
\draw    (117.49,76.58) -- (130.44,89.71) ;
\draw [shift={(121.57,80.73)}, rotate = 45.37] [color={rgb, 255:red, 0; green, 0; blue, 0 }  ][line width=0.75]    (4.37,-1.32) .. controls (2.78,-0.56) and (1.32,-0.12) .. (0,0) .. controls (1.32,0.12) and (2.78,0.56) .. (4.37,1.32)   ;
\draw    (143.45,76.8) -- (130.44,89.71) ;
\draw [shift={(139.36,80.86)}, rotate = 135.25] [color={rgb, 255:red, 0; green, 0; blue, 0 }  ][line width=0.75]    (4.37,-1.32) .. controls (2.78,-0.56) and (1.32,-0.12) .. (0,0) .. controls (1.32,0.12) and (2.78,0.56) .. (4.37,1.32)   ;
\draw  [color={rgb, 255:red, 0; green, 0; blue, 0 }  ,draw opacity=1 ][fill={rgb, 255:red, 0; green, 0; blue, 0 }  ,fill opacity=1 ] (129.01,89.71) .. controls (129.01,88.86) and (129.65,88.17) .. (130.44,88.17) .. controls (131.23,88.17) and (131.87,88.86) .. (131.87,89.71) .. controls (131.87,90.55) and (131.23,91.24) .. (130.44,91.24) .. controls (129.65,91.24) and (129.01,90.55) .. (129.01,89.71) -- cycle ;
\draw  [color={rgb, 255:red, 0; green, 0; blue, 0 }  ,draw opacity=1 ][fill={rgb, 255:red, 0; green, 0; blue, 0 }  ,fill opacity=1 ] (142.02,76.8) .. controls (142.02,75.96) and (142.66,75.27) .. (143.45,75.27) .. controls (144.24,75.27) and (144.88,75.96) .. (144.88,76.8) .. controls (144.88,77.65) and (144.24,78.33) .. (143.45,78.33) .. controls (142.66,78.33) and (142.02,77.65) .. (142.02,76.8) -- cycle ;
\draw  [color={rgb, 255:red, 0; green, 0; blue, 0 }  ,draw opacity=1 ][fill={rgb, 255:red, 0; green, 0; blue, 0 }  ,fill opacity=1 ] (116.06,76.58) .. controls (116.06,75.74) and (116.7,75.05) .. (117.49,75.05) .. controls (118.28,75.05) and (118.92,75.74) .. (118.92,76.58) .. controls (118.92,77.43) and (118.28,78.11) .. (117.49,78.11) .. controls (116.7,78.11) and (116.06,77.43) .. (116.06,76.58) -- cycle ;
\draw  [color={rgb, 255:red, 0; green, 0; blue, 0 }  ,draw opacity=1 ][fill={rgb, 255:red, 0; green, 0; blue, 0 }  ,fill opacity=1 ] (167.5,101.75) .. controls (167.5,100.9) and (168.14,100.22) .. (168.93,100.22) .. controls (169.72,100.22) and (170.36,100.9) .. (170.36,101.75) .. controls (170.36,102.6) and (169.72,103.28) .. (168.93,103.28) .. controls (168.14,103.28) and (167.5,102.6) .. (167.5,101.75) -- cycle ;
\draw    (168.93,81.19) -- (168.93,101.75) ;
\draw [shift={(168.93,88.07)}, rotate = 90] [color={rgb, 255:red, 0; green, 0; blue, 0 }  ][line width=0.75]    (4.37,-1.32) .. controls (2.78,-0.56) and (1.32,-0.12) .. (0,0) .. controls (1.32,0.12) and (2.78,0.56) .. (4.37,1.32)   ;
\draw  [color={rgb, 255:red, 0; green, 0; blue, 0 }  ,draw opacity=1 ][fill={rgb, 255:red, 0; green, 0; blue, 0 }  ,fill opacity=1 ] (167.5,81.19) .. controls (167.5,80.35) and (168.14,79.66) .. (168.93,79.66) .. controls (169.72,79.66) and (170.36,80.35) .. (170.36,81.19) .. controls (170.36,82.04) and (169.72,82.72) .. (168.93,82.72) .. controls (168.14,82.72) and (167.5,82.04) .. (167.5,81.19) -- cycle ;
\draw    (168.93,60.64) -- (168.93,81.19) ;
\draw [shift={(168.93,67.51)}, rotate = 90] [color={rgb, 255:red, 0; green, 0; blue, 0 }  ][line width=0.75]    (4.37,-1.32) .. controls (2.78,-0.56) and (1.32,-0.12) .. (0,0) .. controls (1.32,0.12) and (2.78,0.56) .. (4.37,1.32)   ;
\draw  [color={rgb, 255:red, 0; green, 0; blue, 0 }  ,draw opacity=1 ][fill={rgb, 255:red, 0; green, 0; blue, 0 }  ,fill opacity=1 ] (167.5,60.64) .. controls (167.5,59.79) and (168.14,59.1) .. (168.93,59.1) .. controls (169.72,59.1) and (170.36,59.79) .. (170.36,60.64) .. controls (170.36,61.48) and (169.72,62.17) .. (168.93,62.17) .. controls (168.14,62.17) and (167.5,61.48) .. (167.5,60.64) -- cycle ;
\draw    (273.11,90.23) -- (286.06,103.36) ;
\draw [shift={(277.19,94.38)}, rotate = 45.37] [color={rgb, 255:red, 0; green, 0; blue, 0 }  ][line width=0.75]    (4.37,-1.32) .. controls (2.78,-0.56) and (1.32,-0.12) .. (0,0) .. controls (1.32,0.12) and (2.78,0.56) .. (4.37,1.32)   ;
\draw    (299.07,90.45) -- (286.06,103.36) ;
\draw [shift={(294.98,94.51)}, rotate = 135.25] [color={rgb, 255:red, 0; green, 0; blue, 0 }  ][line width=0.75]    (4.37,-1.32) .. controls (2.78,-0.56) and (1.32,-0.12) .. (0,0) .. controls (1.32,0.12) and (2.78,0.56) .. (4.37,1.32)   ;
\draw  [color={rgb, 255:red, 0; green, 0; blue, 0 }  ,draw opacity=1 ][fill={rgb, 255:red, 0; green, 0; blue, 0 }  ,fill opacity=1 ] (284.63,103.36) .. controls (284.63,102.51) and (285.27,101.83) .. (286.06,101.83) .. controls (286.85,101.83) and (287.49,102.51) .. (287.49,103.36) .. controls (287.49,104.2) and (286.85,104.89) .. (286.06,104.89) .. controls (285.27,104.89) and (284.63,104.2) .. (284.63,103.36) -- cycle ;
\draw  [color={rgb, 255:red, 0; green, 0; blue, 0 }  ,draw opacity=1 ][fill={rgb, 255:red, 0; green, 0; blue, 0 }  ,fill opacity=1 ] (297.64,90.45) .. controls (297.64,89.61) and (298.28,88.92) .. (299.07,88.92) .. controls (299.86,88.92) and (300.5,89.61) .. (300.5,90.45) .. controls (300.5,91.3) and (299.86,91.98) .. (299.07,91.98) .. controls (298.28,91.98) and (297.64,91.3) .. (297.64,90.45) -- cycle ;
\draw  [color={rgb, 255:red, 0; green, 0; blue, 0 }  ,draw opacity=1 ][fill={rgb, 255:red, 0; green, 0; blue, 0 }  ,fill opacity=1 ] (271.68,90.23) .. controls (271.68,89.39) and (272.32,88.7) .. (273.11,88.7) .. controls (273.9,88.7) and (274.54,89.39) .. (274.54,90.23) .. controls (274.54,91.08) and (273.9,91.77) .. (273.11,91.77) .. controls (272.32,91.77) and (271.68,91.08) .. (271.68,90.23) -- cycle ;
\draw    (273.11,69.68) -- (273.11,90.23) ;
\draw [shift={(273.11,76.56)}, rotate = 90] [color={rgb, 255:red, 0; green, 0; blue, 0 }  ][line width=0.75]    (4.37,-1.32) .. controls (2.78,-0.56) and (1.32,-0.12) .. (0,0) .. controls (1.32,0.12) and (2.78,0.56) .. (4.37,1.32)   ;
\draw  [color={rgb, 255:red, 0; green, 0; blue, 0 }  ,draw opacity=1 ][fill={rgb, 255:red, 0; green, 0; blue, 0 }  ,fill opacity=1 ] (271.68,69.68) .. controls (271.68,68.83) and (272.32,68.15) .. (273.11,68.15) .. controls (273.9,68.15) and (274.54,68.83) .. (274.54,69.68) .. controls (274.54,70.52) and (273.9,71.21) .. (273.11,71.21) .. controls (272.32,71.21) and (271.68,70.52) .. (271.68,69.68) -- cycle ;
\draw    (323.77,90.69) -- (336.72,103.81) ;
\draw [shift={(327.86,94.83)}, rotate = 45.37] [color={rgb, 255:red, 0; green, 0; blue, 0 }  ][line width=0.75]    (4.37,-1.32) .. controls (2.78,-0.56) and (1.32,-0.12) .. (0,0) .. controls (1.32,0.12) and (2.78,0.56) .. (4.37,1.32)   ;
\draw    (349.74,90.91) -- (336.72,103.81) ;
\draw [shift={(345.65,94.96)}, rotate = 135.25] [color={rgb, 255:red, 0; green, 0; blue, 0 }  ][line width=0.75]    (4.37,-1.32) .. controls (2.78,-0.56) and (1.32,-0.12) .. (0,0) .. controls (1.32,0.12) and (2.78,0.56) .. (4.37,1.32)   ;
\draw  [color={rgb, 255:red, 0; green, 0; blue, 0 }  ,draw opacity=1 ][fill={rgb, 255:red, 0; green, 0; blue, 0 }  ,fill opacity=1 ] (335.29,103.81) .. controls (335.29,102.96) and (335.93,102.28) .. (336.72,102.28) .. controls (337.51,102.28) and (338.16,102.96) .. (338.16,103.81) .. controls (338.16,104.65) and (337.51,105.34) .. (336.72,105.34) .. controls (335.93,105.34) and (335.29,104.65) .. (335.29,103.81) -- cycle ;
\draw  [color={rgb, 255:red, 0; green, 0; blue, 0 }  ,draw opacity=1 ][fill={rgb, 255:red, 0; green, 0; blue, 0 }  ,fill opacity=1 ] (348.31,90.91) .. controls (348.31,90.06) and (348.95,89.38) .. (349.74,89.38) .. controls (350.53,89.38) and (351.17,90.06) .. (351.17,90.91) .. controls (351.17,91.75) and (350.53,92.44) .. (349.74,92.44) .. controls (348.95,92.44) and (348.31,91.75) .. (348.31,90.91) -- cycle ;
\draw  [color={rgb, 255:red, 0; green, 0; blue, 0 }  ,draw opacity=1 ][fill={rgb, 255:red, 0; green, 0; blue, 0 }  ,fill opacity=1 ] (322.34,90.69) .. controls (322.34,89.84) and (322.98,89.16) .. (323.77,89.16) .. controls (324.56,89.16) and (325.2,89.84) .. (325.2,90.69) .. controls (325.2,91.53) and (324.56,92.22) .. (323.77,92.22) .. controls (322.98,92.22) and (322.34,91.53) .. (322.34,90.69) -- cycle ;
\draw    (349.74,70.35) -- (349.74,90.91) ;
\draw [shift={(349.74,77.23)}, rotate = 90] [color={rgb, 255:red, 0; green, 0; blue, 0 }  ][line width=0.75]    (4.37,-1.32) .. controls (2.78,-0.56) and (1.32,-0.12) .. (0,0) .. controls (1.32,0.12) and (2.78,0.56) .. (4.37,1.32)   ;
\draw  [color={rgb, 255:red, 0; green, 0; blue, 0 }  ,draw opacity=1 ][fill={rgb, 255:red, 0; green, 0; blue, 0 }  ,fill opacity=1 ] (348.31,70.35) .. controls (348.31,69.5) and (348.95,68.82) .. (349.74,68.82) .. controls (350.53,68.82) and (351.17,69.5) .. (351.17,70.35) .. controls (351.17,71.19) and (350.53,71.88) .. (349.74,71.88) .. controls (348.95,71.88) and (348.31,71.19) .. (348.31,70.35) -- cycle ;
\draw  [color={rgb, 255:red, 0; green, 0; blue, 0 }  ,draw opacity=1 ][fill={rgb, 255:red, 0; green, 0; blue, 0 }  ,fill opacity=1 ] (430.48,104.61) .. controls (430.48,103.76) and (431.13,103.08) .. (431.92,103.08) .. controls (432.71,103.08) and (433.35,103.76) .. (433.35,104.61) .. controls (433.35,105.45) and (432.71,106.14) .. (431.92,106.14) .. controls (431.13,106.14) and (430.48,105.45) .. (430.48,104.61) -- cycle ;
\draw    (431.92,84.05) -- (431.92,104.61) ;
\draw [shift={(431.92,90.93)}, rotate = 90] [color={rgb, 255:red, 0; green, 0; blue, 0 }  ][line width=0.75]    (4.37,-1.32) .. controls (2.78,-0.56) and (1.32,-0.12) .. (0,0) .. controls (1.32,0.12) and (2.78,0.56) .. (4.37,1.32)   ;
\draw    (418.96,70.93) -- (431.92,84.05) ;
\draw [shift={(423.05,75.07)}, rotate = 45.37] [color={rgb, 255:red, 0; green, 0; blue, 0 }  ][line width=0.75]    (4.37,-1.32) .. controls (2.78,-0.56) and (1.32,-0.12) .. (0,0) .. controls (1.32,0.12) and (2.78,0.56) .. (4.37,1.32)   ;
\draw    (444.93,71.15) -- (431.92,84.05) ;
\draw [shift={(440.84,75.2)}, rotate = 135.25] [color={rgb, 255:red, 0; green, 0; blue, 0 }  ][line width=0.75]    (4.37,-1.32) .. controls (2.78,-0.56) and (1.32,-0.12) .. (0,0) .. controls (1.32,0.12) and (2.78,0.56) .. (4.37,1.32)   ;
\draw  [color={rgb, 255:red, 0; green, 0; blue, 0 }  ,draw opacity=1 ][fill={rgb, 255:red, 0; green, 0; blue, 0 }  ,fill opacity=1 ] (430.48,84.05) .. controls (430.48,83.2) and (431.13,82.52) .. (431.92,82.52) .. controls (432.71,82.52) and (433.35,83.2) .. (433.35,84.05) .. controls (433.35,84.9) and (432.71,85.58) .. (431.92,85.58) .. controls (431.13,85.58) and (430.48,84.9) .. (430.48,84.05) -- cycle ;
\draw  [color={rgb, 255:red, 0; green, 0; blue, 0 }  ,draw opacity=1 ][fill={rgb, 255:red, 0; green, 0; blue, 0 }  ,fill opacity=1 ] (443.5,71.15) .. controls (443.5,70.3) and (444.14,69.62) .. (444.93,69.62) .. controls (445.72,69.62) and (446.36,70.3) .. (446.36,71.15) .. controls (446.36,71.99) and (445.72,72.68) .. (444.93,72.68) .. controls (444.14,72.68) and (443.5,71.99) .. (443.5,71.15) -- cycle ;
\draw  [color={rgb, 255:red, 0; green, 0; blue, 0 }  ,draw opacity=1 ][fill={rgb, 255:red, 0; green, 0; blue, 0 }  ,fill opacity=1 ] (417.53,70.93) .. controls (417.53,70.08) and (418.17,69.4) .. (418.96,69.4) .. controls (419.75,69.4) and (420.39,70.08) .. (420.39,70.93) .. controls (420.39,71.77) and (419.75,72.46) .. (418.96,72.46) .. controls (418.17,72.46) and (417.53,71.77) .. (417.53,70.93) -- cycle ;
\draw  [color={rgb, 255:red, 0; green, 0; blue, 0 }  ,draw opacity=1 ][fill={rgb, 255:red, 0; green, 0; blue, 0 }  ,fill opacity=1 ] (467.69,112.8) .. controls (467.69,111.95) and (468.33,111.27) .. (469.12,111.27) .. controls (469.91,111.27) and (470.55,111.95) .. (470.55,112.8) .. controls (470.55,113.64) and (469.91,114.33) .. (469.12,114.33) .. controls (468.33,114.33) and (467.69,113.64) .. (467.69,112.8) -- cycle ;
\draw    (469.12,92.24) -- (469.12,112.8) ;
\draw [shift={(469.12,99.12)}, rotate = 90] [color={rgb, 255:red, 0; green, 0; blue, 0 }  ][line width=0.75]    (4.37,-1.32) .. controls (2.78,-0.56) and (1.32,-0.12) .. (0,0) .. controls (1.32,0.12) and (2.78,0.56) .. (4.37,1.32)   ;
\draw  [color={rgb, 255:red, 0; green, 0; blue, 0 }  ,draw opacity=1 ][fill={rgb, 255:red, 0; green, 0; blue, 0 }  ,fill opacity=1 ] (467.69,92.24) .. controls (467.69,91.39) and (468.33,90.71) .. (469.12,90.71) .. controls (469.91,90.71) and (470.55,91.39) .. (470.55,92.24) .. controls (470.55,93.09) and (469.91,93.77) .. (469.12,93.77) .. controls (468.33,93.77) and (467.69,93.09) .. (467.69,92.24) -- cycle ;
\draw    (469.12,71.68) -- (469.12,92.24) ;
\draw [shift={(469.12,78.56)}, rotate = 90] [color={rgb, 255:red, 0; green, 0; blue, 0 }  ][line width=0.75]    (4.37,-1.32) .. controls (2.78,-0.56) and (1.32,-0.12) .. (0,0) .. controls (1.32,0.12) and (2.78,0.56) .. (4.37,1.32)   ;
\draw  [color={rgb, 255:red, 0; green, 0; blue, 0 }  ,draw opacity=1 ][fill={rgb, 255:red, 0; green, 0; blue, 0 }  ,fill opacity=1 ] (467.69,71.68) .. controls (467.69,70.84) and (468.33,70.15) .. (469.12,70.15) .. controls (469.91,70.15) and (470.55,70.84) .. (470.55,71.68) .. controls (470.55,72.53) and (469.91,73.21) .. (469.12,73.21) .. controls (468.33,73.21) and (467.69,72.53) .. (467.69,71.68) -- cycle ;
\draw    (469.12,51.13) -- (469.12,71.68) ;
\draw [shift={(469.12,58)}, rotate = 90] [color={rgb, 255:red, 0; green, 0; blue, 0 }  ][line width=0.75]    (4.37,-1.32) .. controls (2.78,-0.56) and (1.32,-0.12) .. (0,0) .. controls (1.32,0.12) and (2.78,0.56) .. (4.37,1.32)   ;
\draw  [color={rgb, 255:red, 0; green, 0; blue, 0 }  ,draw opacity=1 ][fill={rgb, 255:red, 0; green, 0; blue, 0 }  ,fill opacity=1 ] (467.69,51.13) .. controls (467.69,50.28) and (468.33,49.59) .. (469.12,49.59) .. controls (469.91,49.59) and (470.55,50.28) .. (470.55,51.13) .. controls (470.55,51.97) and (469.91,52.66) .. (469.12,52.66) .. controls (468.33,52.66) and (467.69,51.97) .. (467.69,51.13) -- cycle ;
\draw [color={rgb, 255:red, 155; green, 155; blue, 155 }  ,draw opacity=1 ] [dash pattern={on 0.84pt off 2.51pt}]  (90.26,29.64) -- (90.26,131.06) ;
\draw [color={rgb, 255:red, 155; green, 155; blue, 155 }  ,draw opacity=1 ] [dash pattern={on 0.84pt off 2.51pt}]  (199.69,29.64) -- (199.69,131.06) ;
\draw  [color={rgb, 255:red, 0; green, 0; blue, 0 }  ,draw opacity=1 ][fill={rgb, 255:red, 0; green, 0; blue, 0 }  ,fill opacity=1 ] (19.37,80.1) .. controls (19.37,79.25) and (20.01,78.57) .. (20.8,78.57) .. controls (21.59,78.57) and (22.23,79.25) .. (22.23,80.1) .. controls (22.23,80.94) and (21.59,81.63) .. (20.8,81.63) .. controls (20.01,81.63) and (19.37,80.94) .. (19.37,80.1) -- cycle ;
\draw [color={rgb, 255:red, 155; green, 155; blue, 155 }  ,draw opacity=1 ] [dash pattern={on 0.84pt off 2.51pt}]  (39.98,30.21) -- (39.98,131.64) ;
\draw    (222.29,90.3) -- (235.24,103.42) ;
\draw [shift={(226.37,94.44)}, rotate = 45.37] [color={rgb, 255:red, 0; green, 0; blue, 0 }  ][line width=0.75]    (4.37,-1.32) .. controls (2.78,-0.56) and (1.32,-0.12) .. (0,0) .. controls (1.32,0.12) and (2.78,0.56) .. (4.37,1.32)   ;
\draw    (248.25,90.52) -- (235.24,103.42) ;
\draw [shift={(244.16,94.57)}, rotate = 135.25] [color={rgb, 255:red, 0; green, 0; blue, 0 }  ][line width=0.75]    (4.37,-1.32) .. controls (2.78,-0.56) and (1.32,-0.12) .. (0,0) .. controls (1.32,0.12) and (2.78,0.56) .. (4.37,1.32)   ;
\draw  [color={rgb, 255:red, 0; green, 0; blue, 0 }  ,draw opacity=1 ][fill={rgb, 255:red, 0; green, 0; blue, 0 }  ,fill opacity=1 ] (233.81,103.42) .. controls (233.81,102.57) and (234.45,101.89) .. (235.24,101.89) .. controls (236.03,101.89) and (236.67,102.57) .. (236.67,103.42) .. controls (236.67,104.27) and (236.03,104.95) .. (235.24,104.95) .. controls (234.45,104.95) and (233.81,104.27) .. (233.81,103.42) -- cycle ;
\draw  [color={rgb, 255:red, 0; green, 0; blue, 0 }  ,draw opacity=1 ][fill={rgb, 255:red, 0; green, 0; blue, 0 }  ,fill opacity=1 ] (246.82,90.52) .. controls (246.82,89.67) and (247.46,88.99) .. (248.25,88.99) .. controls (249.04,88.99) and (249.68,89.67) .. (249.68,90.52) .. controls (249.68,91.36) and (249.04,92.05) .. (248.25,92.05) .. controls (247.46,92.05) and (246.82,91.36) .. (246.82,90.52) -- cycle ;
\draw  [color={rgb, 255:red, 0; green, 0; blue, 0 }  ,draw opacity=1 ][fill={rgb, 255:red, 0; green, 0; blue, 0 }  ,fill opacity=1 ] (220.86,90.3) .. controls (220.86,89.45) and (221.5,88.77) .. (222.29,88.77) .. controls (223.08,88.77) and (223.72,89.45) .. (223.72,90.3) .. controls (223.72,91.14) and (223.08,91.83) .. (222.29,91.83) .. controls (221.5,91.83) and (220.86,91.14) .. (220.86,90.3) -- cycle ;
\draw    (235.24,86.92) -- (235.24,103.42) ;
\draw [shift={(235.24,91.77)}, rotate = 90] [color={rgb, 255:red, 0; green, 0; blue, 0 }  ][line width=0.75]    (4.37,-1.32) .. controls (2.78,-0.56) and (1.32,-0.12) .. (0,0) .. controls (1.32,0.12) and (2.78,0.56) .. (4.37,1.32)   ;
\draw  [color={rgb, 255:red, 0; green, 0; blue, 0 }  ,draw opacity=1 ][fill={rgb, 255:red, 0; green, 0; blue, 0 }  ,fill opacity=1 ] (233.81,86.92) .. controls (233.81,86.08) and (234.45,85.39) .. (235.24,85.39) .. controls (236.03,85.39) and (236.67,86.08) .. (236.67,86.92) .. controls (236.67,87.77) and (236.03,88.45) .. (235.24,88.45) .. controls (234.45,88.45) and (233.81,87.77) .. (233.81,86.92) -- cycle ;
\draw    (371.2,90.69) -- (384.15,103.81) ;
\draw [shift={(375.29,94.83)}, rotate = 45.37] [color={rgb, 255:red, 0; green, 0; blue, 0 }  ][line width=0.75]    (4.37,-1.32) .. controls (2.78,-0.56) and (1.32,-0.12) .. (0,0) .. controls (1.32,0.12) and (2.78,0.56) .. (4.37,1.32)   ;
\draw    (397.17,90.91) -- (384.15,103.81) ;
\draw [shift={(393.08,94.96)}, rotate = 135.25] [color={rgb, 255:red, 0; green, 0; blue, 0 }  ][line width=0.75]    (4.37,-1.32) .. controls (2.78,-0.56) and (1.32,-0.12) .. (0,0) .. controls (1.32,0.12) and (2.78,0.56) .. (4.37,1.32)   ;
\draw  [color={rgb, 255:red, 0; green, 0; blue, 0 }  ,draw opacity=1 ][fill={rgb, 255:red, 0; green, 0; blue, 0 }  ,fill opacity=1 ] (382.72,103.81) .. controls (382.72,102.96) and (383.36,102.28) .. (384.15,102.28) .. controls (384.94,102.28) and (385.58,102.96) .. (385.58,103.81) .. controls (385.58,104.65) and (384.94,105.34) .. (384.15,105.34) .. controls (383.36,105.34) and (382.72,104.65) .. (382.72,103.81) -- cycle ;
\draw  [color={rgb, 255:red, 0; green, 0; blue, 0 }  ,draw opacity=1 ][fill={rgb, 255:red, 0; green, 0; blue, 0 }  ,fill opacity=1 ] (395.74,90.91) .. controls (395.74,90.06) and (396.38,89.38) .. (397.17,89.38) .. controls (397.96,89.38) and (398.6,90.06) .. (398.6,90.91) .. controls (398.6,91.75) and (397.96,92.44) .. (397.17,92.44) .. controls (396.38,92.44) and (395.74,91.75) .. (395.74,90.91) -- cycle ;
\draw  [color={rgb, 255:red, 0; green, 0; blue, 0 }  ,draw opacity=1 ][fill={rgb, 255:red, 0; green, 0; blue, 0 }  ,fill opacity=1 ] (369.77,90.69) .. controls (369.77,89.84) and (370.41,89.16) .. (371.2,89.16) .. controls (371.99,89.16) and (372.63,89.84) .. (372.63,90.69) .. controls (372.63,91.53) and (371.99,92.22) .. (371.2,92.22) .. controls (370.41,92.22) and (369.77,91.53) .. (369.77,90.69) -- cycle ;
\draw    (397.17,70.35) -- (397.17,90.91) ;
\draw [shift={(397.17,77.23)}, rotate = 90] [color={rgb, 255:red, 0; green, 0; blue, 0 }  ][line width=0.75]    (4.37,-1.32) .. controls (2.78,-0.56) and (1.32,-0.12) .. (0,0) .. controls (1.32,0.12) and (2.78,0.56) .. (4.37,1.32)   ;
\draw  [color={rgb, 255:red, 0; green, 0; blue, 0 }  ,draw opacity=1 ][fill={rgb, 255:red, 0; green, 0; blue, 0 }  ,fill opacity=1 ] (395.74,70.35) .. controls (395.74,69.5) and (396.38,68.82) .. (397.17,68.82) .. controls (397.96,68.82) and (398.6,69.5) .. (398.6,70.35) .. controls (398.6,71.19) and (397.96,71.88) .. (397.17,71.88) .. controls (396.38,71.88) and (395.74,71.19) .. (395.74,70.35) -- cycle ;

\draw (62.23,92.83) node [anchor=north west][inner sep=0.75pt]  [font=\fontsize{0.35em}{0.42em}\selectfont]  {$w_{1}$};
\draw (62.23,72.27) node [anchor=north west][inner sep=0.75pt]  [font=\fontsize{0.35em}{0.42em}\selectfont]  {$w_{2}$};
\draw (21,81.97) node [anchor=north west][inner sep=0.75pt]  [font=\fontsize{0.35em}{0.42em}\selectfont]  {$w_{1}$};
\draw (338.72,105.68) node [anchor=north west][inner sep=0.75pt]  [font=\fontsize{0.35em}{0.42em}\selectfont]  {$w_{1}$};
\draw (288.06,105.23) node [anchor=north west][inner sep=0.75pt]  [font=\fontsize{0.35em}{0.42em}\selectfont]  {$w_{1}$};
\draw (170.93,103.62) node [anchor=north west][inner sep=0.75pt]  [font=\fontsize{0.35em}{0.42em}\selectfont]  {$w_{1}$};
\draw (132.44,91.57) node [anchor=north west][inner sep=0.75pt]  [font=\fontsize{0.35em}{0.42em}\selectfont]  {$w_{1}$};
\draw (433.92,106.48) node [anchor=north west][inner sep=0.75pt]  [font=\fontsize{0.35em}{0.42em}\selectfont]  {$w_{1}$};
\draw (471.12,114.67) node [anchor=north west][inner sep=0.75pt]  [font=\fontsize{0.35em}{0.42em}\selectfont]  {$w_{1}$};
\draw (170.93,83.06) node [anchor=north west][inner sep=0.75pt]  [font=\fontsize{0.35em}{0.42em}\selectfont]  {$w_{2}$};
\draw (105.49,79.02) node [anchor=north west][inner sep=0.75pt]  [font=\fontsize{0.35em}{0.42em}\selectfont]  {$w_{2}$};
\draw (170.93,62.5) node [anchor=north west][inner sep=0.75pt]  [font=\fontsize{0.35em}{0.42em}\selectfont]  {$w_{3}$};
\draw (145.45,78.67) node [anchor=north west][inner sep=0.75pt]  [font=\fontsize{0.35em}{0.42em}\selectfont]  {$w_{3}$};
\draw (314,93.69) node [anchor=north west][inner sep=0.75pt]  [font=\fontsize{0.35em}{0.42em}\selectfont]  {$w_{2}$};
\draw (263,93.69) node [anchor=north west][inner sep=0.75pt]  [font=\fontsize{0.35em}{0.42em}\selectfont]  {$w_{2}$};
\draw (471.12,94.11) node [anchor=north west][inner sep=0.75pt]  [font=\fontsize{0.35em}{0.42em}\selectfont]  {$w_{2}$};
\draw (433.92,85.92) node [anchor=north west][inner sep=0.75pt]  [font=\fontsize{0.35em}{0.42em}\selectfont]  {$w_{2}$};
\draw (348.74,92.78) node [anchor=north west][inner sep=0.75pt]  [font=\fontsize{0.35em}{0.42em}\selectfont]  {$w_{3}$};
\draw (471.12,73.55) node [anchor=north west][inner sep=0.75pt]  [font=\fontsize{0.35em}{0.42em}\selectfont]  {$w_{3}$};
\draw (446.93,73.02) node [anchor=north west][inner sep=0.75pt]  [font=\fontsize{0.35em}{0.42em}\selectfont]  {$w_{3}$};
\draw (301.07,92.32) node [anchor=north west][inner sep=0.75pt]  [font=\fontsize{0.35em}{0.42em}\selectfont]  {$w_{3}$};
\draw (235,80) node [anchor=north west][inner sep=0.75pt]  [font=\fontsize{0.35em}{0.42em}\selectfont]  {$w_{4}$};
\draw (471.12,52.99) node [anchor=north west][inner sep=0.75pt]  [font=\fontsize{0.35em}{0.42em}\selectfont]  {$w_{4}$};
\draw (409.52,73.11) node [anchor=north west][inner sep=0.75pt]  [font=\fontsize{0.35em}{0.42em}\selectfont]  {$w_{4}$};
\draw (351.74,72.22) node [anchor=north west][inner sep=0.75pt]  [font=\fontsize{0.35em}{0.42em}\selectfont]  {$w_{4}$};
\draw (275.11,71.55) node [anchor=north west][inner sep=0.75pt]  [font=\fontsize{0.35em}{0.42em}\selectfont]  {$w_{4}$};
\draw (237.24,105.29) node [anchor=north west][inner sep=0.75pt]  [font=\fontsize{0.35em}{0.42em}\selectfont]  {$w_{1}$};
\draw (211.29,92.74) node [anchor=north west][inner sep=0.75pt]  [font=\fontsize{0.35em}{0.42em}\selectfont]  {$w_{2}$};
\draw (250.25,92.39) node [anchor=north west][inner sep=0.75pt]  [font=\fontsize{0.35em}{0.42em}\selectfont]  {$w_{3}$};
\draw (386.15,105.68) node [anchor=north west][inner sep=0.75pt]  [font=\fontsize{0.35em}{0.42em}\selectfont]  {$w_{1}$};
\draw (399.17,92.78) node [anchor=north west][inner sep=0.75pt]  [font=\fontsize{0.35em}{0.42em}\selectfont]  {$w_{2}$};
\draw (383.74,72.2) node [anchor=north west][inner sep=0.75pt]  [font=\fontsize{0.35em}{0.42em}\selectfont]  {$w_{3}$};
\draw (361.74,93.07) node [anchor=north west][inner sep=0.75pt]  [font=\fontsize{0.35em}{0.42em}\selectfont]  {$w_{4}$};

\end{tikzpicture}

%% file: Figures/treeSketch1.tex
\tikzset{every picture/.style={line width=0.75pt}} 

\begin{tikzpicture}[x=0.75pt,y=0.75pt,yscale=-1.5,xscale=1.5]

\draw    (294.06,121.26) -- (307.01,134.38) ;
\draw [shift={(298.15,125.4)}, rotate = 45.37] [color={rgb, 255:red, 0; green, 0; blue, 0 }  ][line width=0.75]    (4.37,-1.32) .. controls (2.78,-0.56) and (1.32,-0.12) .. (0,0) .. controls (1.32,0.12) and (2.78,0.56) .. (4.37,1.32)   ;
\draw    (320.03,121.48) -- (307.01,134.38) ;
\draw [shift={(315.93,125.54)}, rotate = 135.25] [color={rgb, 255:red, 0; green, 0; blue, 0 }  ][line width=0.75]    (4.37,-1.32) .. controls (2.78,-0.56) and (1.32,-0.12) .. (0,0) .. controls (1.32,0.12) and (2.78,0.56) .. (4.37,1.32)   ;
\draw  [color={rgb, 255:red, 0; green, 0; blue, 0 }  ,draw opacity=1 ][fill={rgb, 255:red, 0; green, 0; blue, 0 }  ,fill opacity=1 ] (305.58,134.38) .. controls (305.58,133.54) and (306.22,132.85) .. (307.01,132.85) .. controls (307.8,132.85) and (308.44,133.54) .. (308.44,134.38) .. controls (308.44,135.23) and (307.8,135.91) .. (307.01,135.91) .. controls (306.22,135.91) and (305.58,135.23) .. (305.58,134.38) -- cycle ;
\draw  [color={rgb, 255:red, 0; green, 0; blue, 0 }  ,draw opacity=1 ][fill={rgb, 255:red, 0; green, 0; blue, 0 }  ,fill opacity=1 ] (318.59,121.48) .. controls (318.59,120.63) and (319.24,119.95) .. (320.03,119.95) .. controls (320.82,119.95) and (321.46,120.63) .. (321.46,121.48) .. controls (321.46,122.32) and (320.82,123.01) .. (320.03,123.01) .. controls (319.24,123.01) and (318.59,122.32) .. (318.59,121.48) -- cycle ;
\draw  [color={rgb, 255:red, 0; green, 0; blue, 0 }  ,draw opacity=1 ][fill={rgb, 255:red, 0; green, 0; blue, 0 }  ,fill opacity=1 ] (292.63,121.26) .. controls (292.63,120.41) and (293.27,119.73) .. (294.06,119.73) .. controls (294.85,119.73) and (295.49,120.41) .. (295.49,121.26) .. controls (295.49,122.1) and (294.85,122.79) .. (294.06,122.79) .. controls (293.27,122.79) and (292.63,122.1) .. (292.63,121.26) -- cycle ;
\draw    (320.03,100.92) -- (320.03,121.48) ;
\draw [shift={(320.03,107.8)}, rotate = 90] [color={rgb, 255:red, 0; green, 0; blue, 0 }  ][line width=0.75]    (4.37,-1.32) .. controls (2.78,-0.56) and (1.32,-0.12) .. (0,0) .. controls (1.32,0.12) and (2.78,0.56) .. (4.37,1.32)   ;
\draw  [color={rgb, 255:red, 0; green, 0; blue, 0 }  ,draw opacity=1 ][fill={rgb, 255:red, 0; green, 0; blue, 0 }  ,fill opacity=1 ] (318.59,100.92) .. controls (318.59,100.07) and (319.24,99.39) .. (320.03,99.39) .. controls (320.82,99.39) and (321.46,100.07) .. (321.46,100.92) .. controls (321.46,101.77) and (320.82,102.45) .. (320.03,102.45) .. controls (319.24,102.45) and (318.59,101.77) .. (318.59,100.92) -- cycle ;
\draw [color={rgb, 255:red, 0; green, 0; blue, 0 }  ,draw opacity=1 ]   (294.06,100.7) -- (294.06,121.26) ;
\draw [shift={(294.06,107.58)}, rotate = 90] [color={rgb, 255:red, 0; green, 0; blue, 0 }  ,draw opacity=1 ][line width=0.75]    (4.37,-1.32) .. controls (2.78,-0.56) and (1.32,-0.12) .. (0,0) .. controls (1.32,0.12) and (2.78,0.56) .. (4.37,1.32)   ;
\draw    (281.11,87.58) -- (294.06,100.7) ;
\draw [shift={(285.19,91.72)}, rotate = 45.37] [color={rgb, 255:red, 0; green, 0; blue, 0 }  ][line width=0.75]    (4.37,-1.32) .. controls (2.78,-0.56) and (1.32,-0.12) .. (0,0) .. controls (1.32,0.12) and (2.78,0.56) .. (4.37,1.32)   ;
\draw    (307.07,87.8) -- (294.06,100.7) ;
\draw [shift={(302.98,91.86)}, rotate = 135.25] [color={rgb, 255:red, 0; green, 0; blue, 0 }  ][line width=0.75]    (4.37,-1.32) .. controls (2.78,-0.56) and (1.32,-0.12) .. (0,0) .. controls (1.32,0.12) and (2.78,0.56) .. (4.37,1.32)   ;
\draw  [color={rgb, 255:red, 0; green, 0; blue, 0 }  ,draw opacity=1 ][fill={rgb, 255:red, 0; green, 0; blue, 0 }  ,fill opacity=1 ] (292.63,100.7) .. controls (292.63,99.86) and (293.27,99.17) .. (294.06,99.17) .. controls (294.85,99.17) and (295.49,99.86) .. (295.49,100.7) .. controls (295.49,101.55) and (294.85,102.23) .. (294.06,102.23) .. controls (293.27,102.23) and (292.63,101.55) .. (292.63,100.7) -- cycle ;
\draw  [color={rgb, 255:red, 0; green, 0; blue, 0 }  ,draw opacity=1 ][fill={rgb, 255:red, 0; green, 0; blue, 0 }  ,fill opacity=1 ] (305.64,87.8) .. controls (305.64,86.95) and (306.28,86.27) .. (307.07,86.27) .. controls (307.86,86.27) and (308.5,86.95) .. (308.5,87.8) .. controls (308.5,88.64) and (307.86,89.33) .. (307.07,89.33) .. controls (306.28,89.33) and (305.64,88.64) .. (305.64,87.8) -- cycle ;
\draw  [color={rgb, 255:red, 0; green, 0; blue, 0 }  ,draw opacity=1 ][fill={rgb, 255:red, 0; green, 0; blue, 0 }  ,fill opacity=1 ] (279.68,87.58) .. controls (279.68,86.73) and (280.32,86.05) .. (281.11,86.05) .. controls (281.9,86.05) and (282.54,86.73) .. (282.54,87.58) .. controls (282.54,88.43) and (281.9,89.11) .. (281.11,89.11) .. controls (280.32,89.11) and (279.68,88.43) .. (279.68,87.58) -- cycle ;

\draw (309.01,136.25) node [anchor=north west][inner sep=0.75pt]  [font=\fontsize{0.35em}{0.42em}\selectfont]  {$w_{1}$};
\draw (322.03,123.35) node [anchor=north west][inner sep=0.75pt]  [font=\fontsize{0.35em}{0.42em}\selectfont]  {$w_{2}$};
\draw (322.03,102.79) node [anchor=north west][inner sep=0.75pt]  [font=\fontsize{0.35em}{0.42em}\selectfont]  {$w_{3}$};
\draw (281.74,123.36) node [anchor=north west][inner sep=0.75pt]  [font=\fontsize{0.35em}{0.42em}\selectfont]  {$w_{4}$};
\draw (309.07,89.67) node [anchor=north west][inner sep=0.75pt]  [font=\fontsize{0.35em}{0.42em}\selectfont]  {$w_{7}$};
\draw (296.06,102.57) node [anchor=north west][inner sep=0.75pt]  [font=\fontsize{0.35em}{0.42em}\selectfont]  {$w_{5}$};
\draw (268.88,89.07) node [anchor=north west][inner sep=0.75pt]  [font=\fontsize{0.35em}{0.42em}\selectfont]  {$w_{6}$};
\draw (315.32,130) node [anchor=north west][inner sep=0.75pt]  [font=\fontsize{0.24em}{0.28em}\selectfont,color={rgb, 255:red, 208; green, 2; blue, 27 }  ,opacity=1 ]  {$j_{1}$};
\draw (322.7,113.79) node [anchor=north west][inner sep=0.75pt]  [font=\fontsize{0.24em}{0.28em}\selectfont,color={rgb, 255:red, 208; green, 2; blue, 27 }  ,opacity=1 ]  {$j_{2}$};
\draw (293.01,130) node [anchor=north west][inner sep=0.75pt]  [font=\fontsize{0.24em}{0.28em}\selectfont,color={rgb, 255:red, 208; green, 2; blue, 27 }  ,opacity=1 ]  {$j_{3}$};
\draw (282.09,110.48) node [anchor=north west][inner sep=0.75pt]  [font=\fontsize{0.24em}{0.28em}\selectfont,color={rgb, 255:red, 208; green, 2; blue, 27 }  ,opacity=1 ]  {$j_{4}$};
\draw (302.86,95) node [anchor=north west][inner sep=0.75pt]  [font=\fontsize{0.24em}{0.28em}\selectfont,color={rgb, 255:red, 208; green, 2; blue, 27 }  ,opacity=1 ]  {$j_{6}$};
\draw (279.01,95) node [anchor=north west][inner sep=0.75pt]  [font=\fontsize{0.24em}{0.28em}\selectfont,color={rgb, 255:red, 208; green, 2; blue, 27 }  ,opacity=1 ]  {$j_{5}$};

\end{tikzpicture}

%% file: Figures/tree_spaces.tex
\tikzset{every picture/.style={line width=0.75pt}} 

\begin{tikzpicture}[x=0.75pt,y=0.75pt,yscale=-1,xscale=1]

\draw  [color={rgb, 255:red, 0; green, 0; blue, 0 }  ,draw opacity=1 ][fill={rgb, 255:red, 0; green, 0; blue, 0 }  ,fill opacity=1 ] (58.8,90.96) .. controls (58.8,90.11) and (59.44,89.43) .. (60.23,89.43) .. controls (61.02,89.43) and (61.66,90.11) .. (61.66,90.96) .. controls (61.66,91.8) and (61.02,92.49) .. (60.23,92.49) .. controls (59.44,92.49) and (58.8,91.8) .. (58.8,90.96) -- cycle ;
\draw    (60.23,70.4) -- (60.23,90.96) ;
\draw [shift={(60.23,77.28)}, rotate = 90] [color={rgb, 255:red, 0; green, 0; blue, 0 }  ][line width=0.75]    (4.37,-1.32) .. controls (2.78,-0.56) and (1.32,-0.12) .. (0,0) .. controls (1.32,0.12) and (2.78,0.56) .. (4.37,1.32)   ;
\draw  [color={rgb, 255:red, 0; green, 0; blue, 0 }  ,draw opacity=1 ][fill={rgb, 255:red, 0; green, 0; blue, 0 }  ,fill opacity=1 ] (58.8,70.4) .. controls (58.8,69.55) and (59.44,68.87) .. (60.23,68.87) .. controls (61.02,68.87) and (61.66,69.55) .. (61.66,70.4) .. controls (61.66,71.24) and (61.02,71.93) .. (60.23,71.93) .. controls (59.44,71.93) and (58.8,71.24) .. (58.8,70.4) -- cycle ;
\draw    (117.49,76.58) -- (130.44,89.71) ;
\draw [shift={(121.57,80.73)}, rotate = 45.37] [color={rgb, 255:red, 0; green, 0; blue, 0 }  ][line width=0.75]    (4.37,-1.32) .. controls (2.78,-0.56) and (1.32,-0.12) .. (0,0) .. controls (1.32,0.12) and (2.78,0.56) .. (4.37,1.32)   ;
\draw    (143.45,76.8) -- (130.44,89.71) ;
\draw [shift={(139.36,80.86)}, rotate = 135.25] [color={rgb, 255:red, 0; green, 0; blue, 0 }  ][line width=0.75]    (4.37,-1.32) .. controls (2.78,-0.56) and (1.32,-0.12) .. (0,0) .. controls (1.32,0.12) and (2.78,0.56) .. (4.37,1.32)   ;
\draw  [color={rgb, 255:red, 0; green, 0; blue, 0 }  ,draw opacity=1 ][fill={rgb, 255:red, 0; green, 0; blue, 0 }  ,fill opacity=1 ] (129.01,89.71) .. controls (129.01,88.86) and (129.65,88.17) .. (130.44,88.17) .. controls (131.23,88.17) and (131.87,88.86) .. (131.87,89.71) .. controls (131.87,90.55) and (131.23,91.24) .. (130.44,91.24) .. controls (129.65,91.24) and (129.01,90.55) .. (129.01,89.71) -- cycle ;
\draw  [color={rgb, 255:red, 0; green, 0; blue, 0 }  ,draw opacity=1 ][fill={rgb, 255:red, 0; green, 0; blue, 0 }  ,fill opacity=1 ] (142.02,76.8) .. controls (142.02,75.96) and (142.66,75.27) .. (143.45,75.27) .. controls (144.24,75.27) and (144.88,75.96) .. (144.88,76.8) .. controls (144.88,77.65) and (144.24,78.33) .. (143.45,78.33) .. controls (142.66,78.33) and (142.02,77.65) .. (142.02,76.8) -- cycle ;
\draw  [color={rgb, 255:red, 0; green, 0; blue, 0 }  ,draw opacity=1 ][fill={rgb, 255:red, 0; green, 0; blue, 0 }  ,fill opacity=1 ] (116.06,76.58) .. controls (116.06,75.74) and (116.7,75.05) .. (117.49,75.05) .. controls (118.28,75.05) and (118.92,75.74) .. (118.92,76.58) .. controls (118.92,77.43) and (118.28,78.11) .. (117.49,78.11) .. controls (116.7,78.11) and (116.06,77.43) .. (116.06,76.58) -- cycle ;
\draw  [color={rgb, 255:red, 0; green, 0; blue, 0 }  ,draw opacity=1 ][fill={rgb, 255:red, 0; green, 0; blue, 0 }  ,fill opacity=1 ] (167.5,101.75) .. controls (167.5,100.9) and (168.14,100.22) .. (168.93,100.22) .. controls (169.72,100.22) and (170.36,100.9) .. (170.36,101.75) .. controls (170.36,102.6) and (169.72,103.28) .. (168.93,103.28) .. controls (168.14,103.28) and (167.5,102.6) .. (167.5,101.75) -- cycle ;
\draw    (168.93,81.19) -- (168.93,101.75) ;
\draw [shift={(168.93,88.07)}, rotate = 90] [color={rgb, 255:red, 0; green, 0; blue, 0 }  ][line width=0.75]    (4.37,-1.32) .. controls (2.78,-0.56) and (1.32,-0.12) .. (0,0) .. controls (1.32,0.12) and (2.78,0.56) .. (4.37,1.32)   ;
\draw  [color={rgb, 255:red, 0; green, 0; blue, 0 }  ,draw opacity=1 ][fill={rgb, 255:red, 0; green, 0; blue, 0 }  ,fill opacity=1 ] (167.5,81.19) .. controls (167.5,80.35) and (168.14,79.66) .. (168.93,79.66) .. controls (169.72,79.66) and (170.36,80.35) .. (170.36,81.19) .. controls (170.36,82.04) and (169.72,82.72) .. (168.93,82.72) .. controls (168.14,82.72) and (167.5,82.04) .. (167.5,81.19) -- cycle ;
\draw    (168.93,60.64) -- (168.93,81.19) ;
\draw [shift={(168.93,67.51)}, rotate = 90] [color={rgb, 255:red, 0; green, 0; blue, 0 }  ][line width=0.75]    (4.37,-1.32) .. controls (2.78,-0.56) and (1.32,-0.12) .. (0,0) .. controls (1.32,0.12) and (2.78,0.56) .. (4.37,1.32)   ;
\draw  [color={rgb, 255:red, 0; green, 0; blue, 0 }  ,draw opacity=1 ][fill={rgb, 255:red, 0; green, 0; blue, 0 }  ,fill opacity=1 ] (167.5,60.64) .. controls (167.5,59.79) and (168.14,59.1) .. (168.93,59.1) .. controls (169.72,59.1) and (170.36,59.79) .. (170.36,60.64) .. controls (170.36,61.48) and (169.72,62.17) .. (168.93,62.17) .. controls (168.14,62.17) and (167.5,61.48) .. (167.5,60.64) -- cycle ;
\draw    (273.11,90.23) -- (286.06,103.36) ;
\draw [shift={(277.19,94.38)}, rotate = 45.37] [color={rgb, 255:red, 0; green, 0; blue, 0 }  ][line width=0.75]    (4.37,-1.32) .. controls (2.78,-0.56) and (1.32,-0.12) .. (0,0) .. controls (1.32,0.12) and (2.78,0.56) .. (4.37,1.32)   ;
\draw    (299.07,90.45) -- (286.06,103.36) ;
\draw [shift={(294.98,94.51)}, rotate = 135.25] [color={rgb, 255:red, 0; green, 0; blue, 0 }  ][line width=0.75]    (4.37,-1.32) .. controls (2.78,-0.56) and (1.32,-0.12) .. (0,0) .. controls (1.32,0.12) and (2.78,0.56) .. (4.37,1.32)   ;
\draw  [color={rgb, 255:red, 0; green, 0; blue, 0 }  ,draw opacity=1 ][fill={rgb, 255:red, 0; green, 0; blue, 0 }  ,fill opacity=1 ] (284.63,103.36) .. controls (284.63,102.51) and (285.27,101.83) .. (286.06,101.83) .. controls (286.85,101.83) and (287.49,102.51) .. (287.49,103.36) .. controls (287.49,104.2) and (286.85,104.89) .. (286.06,104.89) .. controls (285.27,104.89) and (284.63,104.2) .. (284.63,103.36) -- cycle ;
\draw  [color={rgb, 255:red, 0; green, 0; blue, 0 }  ,draw opacity=1 ][fill={rgb, 255:red, 0; green, 0; blue, 0 }  ,fill opacity=1 ] (297.64,90.45) .. controls (297.64,89.61) and (298.28,88.92) .. (299.07,88.92) .. controls (299.86,88.92) and (300.5,89.61) .. (300.5,90.45) .. controls (300.5,91.3) and (299.86,91.98) .. (299.07,91.98) .. controls (298.28,91.98) and (297.64,91.3) .. (297.64,90.45) -- cycle ;
\draw  [color={rgb, 255:red, 0; green, 0; blue, 0 }  ,draw opacity=1 ][fill={rgb, 255:red, 0; green, 0; blue, 0 }  ,fill opacity=1 ] (271.68,90.23) .. controls (271.68,89.39) and (272.32,88.7) .. (273.11,88.7) .. controls (273.9,88.7) and (274.54,89.39) .. (274.54,90.23) .. controls (274.54,91.08) and (273.9,91.77) .. (273.11,91.77) .. controls (272.32,91.77) and (271.68,91.08) .. (271.68,90.23) -- cycle ;
\draw    (273.11,69.68) -- (273.11,90.23) ;
\draw [shift={(273.11,76.56)}, rotate = 90] [color={rgb, 255:red, 0; green, 0; blue, 0 }  ][line width=0.75]    (4.37,-1.32) .. controls (2.78,-0.56) and (1.32,-0.12) .. (0,0) .. controls (1.32,0.12) and (2.78,0.56) .. (4.37,1.32)   ;
\draw  [color={rgb, 255:red, 0; green, 0; blue, 0 }  ,draw opacity=1 ][fill={rgb, 255:red, 0; green, 0; blue, 0 }  ,fill opacity=1 ] (271.68,69.68) .. controls (271.68,68.83) and (272.32,68.15) .. (273.11,68.15) .. controls (273.9,68.15) and (274.54,68.83) .. (274.54,69.68) .. controls (274.54,70.52) and (273.9,71.21) .. (273.11,71.21) .. controls (272.32,71.21) and (271.68,70.52) .. (271.68,69.68) -- cycle ;
\draw    (323.77,90.69) -- (336.72,103.81) ;
\draw [shift={(327.86,94.83)}, rotate = 45.37] [color={rgb, 255:red, 0; green, 0; blue, 0 }  ][line width=0.75]    (4.37,-1.32) .. controls (2.78,-0.56) and (1.32,-0.12) .. (0,0) .. controls (1.32,0.12) and (2.78,0.56) .. (4.37,1.32)   ;
\draw    (349.74,90.91) -- (336.72,103.81) ;
\draw [shift={(345.65,94.96)}, rotate = 135.25] [color={rgb, 255:red, 0; green, 0; blue, 0 }  ][line width=0.75]    (4.37,-1.32) .. controls (2.78,-0.56) and (1.32,-0.12) .. (0,0) .. controls (1.32,0.12) and (2.78,0.56) .. (4.37,1.32)   ;
\draw  [color={rgb, 255:red, 0; green, 0; blue, 0 }  ,draw opacity=1 ][fill={rgb, 255:red, 0; green, 0; blue, 0 }  ,fill opacity=1 ] (335.29,103.81) .. controls (335.29,102.96) and (335.93,102.28) .. (336.72,102.28) .. controls (337.51,102.28) and (338.16,102.96) .. (338.16,103.81) .. controls (338.16,104.65) and (337.51,105.34) .. (336.72,105.34) .. controls (335.93,105.34) and (335.29,104.65) .. (335.29,103.81) -- cycle ;
\draw  [color={rgb, 255:red, 0; green, 0; blue, 0 }  ,draw opacity=1 ][fill={rgb, 255:red, 0; green, 0; blue, 0 }  ,fill opacity=1 ] (348.31,90.91) .. controls (348.31,90.06) and (348.95,89.38) .. (349.74,89.38) .. controls (350.53,89.38) and (351.17,90.06) .. (351.17,90.91) .. controls (351.17,91.75) and (350.53,92.44) .. (349.74,92.44) .. controls (348.95,92.44) and (348.31,91.75) .. (348.31,90.91) -- cycle ;
\draw  [color={rgb, 255:red, 0; green, 0; blue, 0 }  ,draw opacity=1 ][fill={rgb, 255:red, 0; green, 0; blue, 0 }  ,fill opacity=1 ] (322.34,90.69) .. controls (322.34,89.84) and (322.98,89.16) .. (323.77,89.16) .. controls (324.56,89.16) and (325.2,89.84) .. (325.2,90.69) .. controls (325.2,91.53) and (324.56,92.22) .. (323.77,92.22) .. controls (322.98,92.22) and (322.34,91.53) .. (322.34,90.69) -- cycle ;
\draw    (349.74,70.35) -- (349.74,90.91) ;
\draw [shift={(349.74,77.23)}, rotate = 90] [color={rgb, 255:red, 0; green, 0; blue, 0 }  ][line width=0.75]    (4.37,-1.32) .. controls (2.78,-0.56) and (1.32,-0.12) .. (0,0) .. controls (1.32,0.12) and (2.78,0.56) .. (4.37,1.32)   ;
\draw  [color={rgb, 255:red, 0; green, 0; blue, 0 }  ,draw opacity=1 ][fill={rgb, 255:red, 0; green, 0; blue, 0 }  ,fill opacity=1 ] (348.31,70.35) .. controls (348.31,69.5) and (348.95,68.82) .. (349.74,68.82) .. controls (350.53,68.82) and (351.17,69.5) .. (351.17,70.35) .. controls (351.17,71.19) and (350.53,71.88) .. (349.74,71.88) .. controls (348.95,71.88) and (348.31,71.19) .. (348.31,70.35) -- cycle ;
\draw  [color={rgb, 255:red, 0; green, 0; blue, 0 }  ,draw opacity=1 ][fill={rgb, 255:red, 0; green, 0; blue, 0 }  ,fill opacity=1 ] (430.48,104.61) .. controls (430.48,103.76) and (431.13,103.08) .. (431.92,103.08) .. controls (432.71,103.08) and (433.35,103.76) .. (433.35,104.61) .. controls (433.35,105.45) and (432.71,106.14) .. (431.92,106.14) .. controls (431.13,106.14) and (430.48,105.45) .. (430.48,104.61) -- cycle ;
\draw    (431.92,84.05) -- (431.92,104.61) ;
\draw [shift={(431.92,90.93)}, rotate = 90] [color={rgb, 255:red, 0; green, 0; blue, 0 }  ][line width=0.75]    (4.37,-1.32) .. controls (2.78,-0.56) and (1.32,-0.12) .. (0,0) .. controls (1.32,0.12) and (2.78,0.56) .. (4.37,1.32)   ;
\draw    (418.96,70.93) -- (431.92,84.05) ;
\draw [shift={(423.05,75.07)}, rotate = 45.37] [color={rgb, 255:red, 0; green, 0; blue, 0 }  ][line width=0.75]    (4.37,-1.32) .. controls (2.78,-0.56) and (1.32,-0.12) .. (0,0) .. controls (1.32,0.12) and (2.78,0.56) .. (4.37,1.32)   ;
\draw    (444.93,71.15) -- (431.92,84.05) ;
\draw [shift={(440.84,75.2)}, rotate = 135.25] [color={rgb, 255:red, 0; green, 0; blue, 0 }  ][line width=0.75]    (4.37,-1.32) .. controls (2.78,-0.56) and (1.32,-0.12) .. (0,0) .. controls (1.32,0.12) and (2.78,0.56) .. (4.37,1.32)   ;
\draw  [color={rgb, 255:red, 0; green, 0; blue, 0 }  ,draw opacity=1 ][fill={rgb, 255:red, 0; green, 0; blue, 0 }  ,fill opacity=1 ] (430.48,84.05) .. controls (430.48,83.2) and (431.13,82.52) .. (431.92,82.52) .. controls (432.71,82.52) and (433.35,83.2) .. (433.35,84.05) .. controls (433.35,84.9) and (432.71,85.58) .. (431.92,85.58) .. controls (431.13,85.58) and (430.48,84.9) .. (430.48,84.05) -- cycle ;
\draw  [color={rgb, 255:red, 0; green, 0; blue, 0 }  ,draw opacity=1 ][fill={rgb, 255:red, 0; green, 0; blue, 0 }  ,fill opacity=1 ] (443.5,71.15) .. controls (443.5,70.3) and (444.14,69.62) .. (444.93,69.62) .. controls (445.72,69.62) and (446.36,70.3) .. (446.36,71.15) .. controls (446.36,71.99) and (445.72,72.68) .. (444.93,72.68) .. controls (444.14,72.68) and (443.5,71.99) .. (443.5,71.15) -- cycle ;
\draw  [color={rgb, 255:red, 0; green, 0; blue, 0 }  ,draw opacity=1 ][fill={rgb, 255:red, 0; green, 0; blue, 0 }  ,fill opacity=1 ] (417.53,70.93) .. controls (417.53,70.08) and (418.17,69.4) .. (418.96,69.4) .. controls (419.75,69.4) and (420.39,70.08) .. (420.39,70.93) .. controls (420.39,71.77) and (419.75,72.46) .. (418.96,72.46) .. controls (418.17,72.46) and (417.53,71.77) .. (417.53,70.93) -- cycle ;
\draw  [color={rgb, 255:red, 0; green, 0; blue, 0 }  ,draw opacity=1 ][fill={rgb, 255:red, 0; green, 0; blue, 0 }  ,fill opacity=1 ] (467.69,112.8) .. controls (467.69,111.95) and (468.33,111.27) .. (469.12,111.27) .. controls (469.91,111.27) and (470.55,111.95) .. (470.55,112.8) .. controls (470.55,113.64) and (469.91,114.33) .. (469.12,114.33) .. controls (468.33,114.33) and (467.69,113.64) .. (467.69,112.8) -- cycle ;
\draw    (469.12,92.24) -- (469.12,112.8) ;
\draw [shift={(469.12,99.12)}, rotate = 90] [color={rgb, 255:red, 0; green, 0; blue, 0 }  ][line width=0.75]    (4.37,-1.32) .. controls (2.78,-0.56) and (1.32,-0.12) .. (0,0) .. controls (1.32,0.12) and (2.78,0.56) .. (4.37,1.32)   ;
\draw  [color={rgb, 255:red, 0; green, 0; blue, 0 }  ,draw opacity=1 ][fill={rgb, 255:red, 0; green, 0; blue, 0 }  ,fill opacity=1 ] (467.69,92.24) .. controls (467.69,91.39) and (468.33,90.71) .. (469.12,90.71) .. controls (469.91,90.71) and (470.55,91.39) .. (470.55,92.24) .. controls (470.55,93.09) and (469.91,93.77) .. (469.12,93.77) .. controls (468.33,93.77) and (467.69,93.09) .. (467.69,92.24) -- cycle ;
\draw    (469.12,71.68) -- (469.12,92.24) ;
\draw [shift={(469.12,78.56)}, rotate = 90] [color={rgb, 255:red, 0; green, 0; blue, 0 }  ][line width=0.75]    (4.37,-1.32) .. controls (2.78,-0.56) and (1.32,-0.12) .. (0,0) .. controls (1.32,0.12) and (2.78,0.56) .. (4.37,1.32)   ;
\draw  [color={rgb, 255:red, 0; green, 0; blue, 0 }  ,draw opacity=1 ][fill={rgb, 255:red, 0; green, 0; blue, 0 }  ,fill opacity=1 ] (467.69,71.68) .. controls (467.69,70.84) and (468.33,70.15) .. (469.12,70.15) .. controls (469.91,70.15) and (470.55,70.84) .. (470.55,71.68) .. controls (470.55,72.53) and (469.91,73.21) .. (469.12,73.21) .. controls (468.33,73.21) and (467.69,72.53) .. (467.69,71.68) -- cycle ;
\draw    (469.12,51.13) -- (469.12,71.68) ;
\draw [shift={(469.12,58)}, rotate = 90] [color={rgb, 255:red, 0; green, 0; blue, 0 }  ][line width=0.75]    (4.37,-1.32) .. controls (2.78,-0.56) and (1.32,-0.12) .. (0,0) .. controls (1.32,0.12) and (2.78,0.56) .. (4.37,1.32)   ;
\draw  [color={rgb, 255:red, 0; green, 0; blue, 0 }  ,draw opacity=1 ][fill={rgb, 255:red, 0; green, 0; blue, 0 }  ,fill opacity=1 ] (467.69,51.13) .. controls (467.69,50.28) and (468.33,49.59) .. (469.12,49.59) .. controls (469.91,49.59) and (470.55,50.28) .. (470.55,51.13) .. controls (470.55,51.97) and (469.91,52.66) .. (469.12,52.66) .. controls (468.33,52.66) and (467.69,51.97) .. (467.69,51.13) -- cycle ;
\draw [color={rgb, 255:red, 155; green, 155; blue, 155 }  ,draw opacity=1 ] [dash pattern={on 0.84pt off 2.51pt}]  (90.26,29.64) -- (90.26,131.06) ;
\draw [color={rgb, 255:red, 155; green, 155; blue, 155 }  ,draw opacity=1 ] [dash pattern={on 0.84pt off 2.51pt}]  (199.69,29.64) -- (199.69,131.06) ;
\draw    (222.29,90.3) -- (235.24,103.42) ;
\draw [shift={(226.37,94.44)}, rotate = 45.37] [color={rgb, 255:red, 0; green, 0; blue, 0 }  ][line width=0.75]    (4.37,-1.32) .. controls (2.78,-0.56) and (1.32,-0.12) .. (0,0) .. controls (1.32,0.12) and (2.78,0.56) .. (4.37,1.32)   ;
\draw    (248.25,90.52) -- (235.24,103.42) ;
\draw [shift={(244.16,94.57)}, rotate = 135.25] [color={rgb, 255:red, 0; green, 0; blue, 0 }  ][line width=0.75]    (4.37,-1.32) .. controls (2.78,-0.56) and (1.32,-0.12) .. (0,0) .. controls (1.32,0.12) and (2.78,0.56) .. (4.37,1.32)   ;
\draw  [color={rgb, 255:red, 0; green, 0; blue, 0 }  ,draw opacity=1 ][fill={rgb, 255:red, 0; green, 0; blue, 0 }  ,fill opacity=1 ] (233.81,103.42) .. controls (233.81,102.57) and (234.45,101.89) .. (235.24,101.89) .. controls (236.03,101.89) and (236.67,102.57) .. (236.67,103.42) .. controls (236.67,104.27) and (236.03,104.95) .. (235.24,104.95) .. controls (234.45,104.95) and (233.81,104.27) .. (233.81,103.42) -- cycle ;
\draw  [color={rgb, 255:red, 0; green, 0; blue, 0 }  ,draw opacity=1 ][fill={rgb, 255:red, 0; green, 0; blue, 0 }  ,fill opacity=1 ] (246.82,90.52) .. controls (246.82,89.67) and (247.46,88.99) .. (248.25,88.99) .. controls (249.04,88.99) and (249.68,89.67) .. (249.68,90.52) .. controls (249.68,91.36) and (249.04,92.05) .. (248.25,92.05) .. controls (247.46,92.05) and (246.82,91.36) .. (246.82,90.52) -- cycle ;
\draw  [color={rgb, 255:red, 0; green, 0; blue, 0 }  ,draw opacity=1 ][fill={rgb, 255:red, 0; green, 0; blue, 0 }  ,fill opacity=1 ] (220.86,90.3) .. controls (220.86,89.45) and (221.5,88.77) .. (222.29,88.77) .. controls (223.08,88.77) and (223.72,89.45) .. (223.72,90.3) .. controls (223.72,91.14) and (223.08,91.83) .. (222.29,91.83) .. controls (221.5,91.83) and (220.86,91.14) .. (220.86,90.3) -- cycle ;
\draw    (235.24,86.92) -- (235.24,103.42) ;
\draw [shift={(235.24,91.77)}, rotate = 90] [color={rgb, 255:red, 0; green, 0; blue, 0 }  ][line width=0.75]    (4.37,-1.32) .. controls (2.78,-0.56) and (1.32,-0.12) .. (0,0) .. controls (1.32,0.12) and (2.78,0.56) .. (4.37,1.32)   ;
\draw  [color={rgb, 255:red, 0; green, 0; blue, 0 }  ,draw opacity=1 ][fill={rgb, 255:red, 0; green, 0; blue, 0 }  ,fill opacity=1 ] (233.81,86.92) .. controls (233.81,86.08) and (234.45,85.39) .. (235.24,85.39) .. controls (236.03,85.39) and (236.67,86.08) .. (236.67,86.92) .. controls (236.67,87.77) and (236.03,88.45) .. (235.24,88.45) .. controls (234.45,88.45) and (233.81,87.77) .. (233.81,86.92) -- cycle ;
\draw    (371.2,90.69) -- (384.15,103.81) ;
\draw [shift={(375.29,94.83)}, rotate = 45.37] [color={rgb, 255:red, 0; green, 0; blue, 0 }  ][line width=0.75]    (4.37,-1.32) .. controls (2.78,-0.56) and (1.32,-0.12) .. (0,0) .. controls (1.32,0.12) and (2.78,0.56) .. (4.37,1.32)   ;
\draw    (397.17,90.91) -- (384.15,103.81) ;
\draw [shift={(393.08,94.96)}, rotate = 135.25] [color={rgb, 255:red, 0; green, 0; blue, 0 }  ][line width=0.75]    (4.37,-1.32) .. controls (2.78,-0.56) and (1.32,-0.12) .. (0,0) .. controls (1.32,0.12) and (2.78,0.56) .. (4.37,1.32)   ;
\draw  [color={rgb, 255:red, 0; green, 0; blue, 0 }  ,draw opacity=1 ][fill={rgb, 255:red, 0; green, 0; blue, 0 }  ,fill opacity=1 ] (382.72,103.81) .. controls (382.72,102.96) and (383.36,102.28) .. (384.15,102.28) .. controls (384.94,102.28) and (385.58,102.96) .. (385.58,103.81) .. controls (385.58,104.65) and (384.94,105.34) .. (384.15,105.34) .. controls (383.36,105.34) and (382.72,104.65) .. (382.72,103.81) -- cycle ;
\draw  [color={rgb, 255:red, 0; green, 0; blue, 0 }  ,draw opacity=1 ][fill={rgb, 255:red, 0; green, 0; blue, 0 }  ,fill opacity=1 ] (395.74,90.91) .. controls (395.74,90.06) and (396.38,89.38) .. (397.17,89.38) .. controls (397.96,89.38) and (398.6,90.06) .. (398.6,90.91) .. controls (398.6,91.75) and (397.96,92.44) .. (397.17,92.44) .. controls (396.38,92.44) and (395.74,91.75) .. (395.74,90.91) -- cycle ;
\draw  [color={rgb, 255:red, 0; green, 0; blue, 0 }  ,draw opacity=1 ][fill={rgb, 255:red, 0; green, 0; blue, 0 }  ,fill opacity=1 ] (369.77,90.69) .. controls (369.77,89.84) and (370.41,89.16) .. (371.2,89.16) .. controls (371.99,89.16) and (372.63,89.84) .. (372.63,90.69) .. controls (372.63,91.53) and (371.99,92.22) .. (371.2,92.22) .. controls (370.41,92.22) and (369.77,91.53) .. (369.77,90.69) -- cycle ;
\draw    (397.17,70.35) -- (397.17,90.91) ;
\draw [shift={(397.17,77.23)}, rotate = 90] [color={rgb, 255:red, 0; green, 0; blue, 0 }  ][line width=0.75]    (4.37,-1.32) .. controls (2.78,-0.56) and (1.32,-0.12) .. (0,0) .. controls (1.32,0.12) and (2.78,0.56) .. (4.37,1.32)   ;
\draw  [color={rgb, 255:red, 0; green, 0; blue, 0 }  ,draw opacity=1 ][fill={rgb, 255:red, 0; green, 0; blue, 0 }  ,fill opacity=1 ] (395.74,70.35) .. controls (395.74,69.5) and (396.38,68.82) .. (397.17,68.82) .. controls (397.96,68.82) and (398.6,69.5) .. (398.6,70.35) .. controls (398.6,71.19) and (397.96,71.88) .. (397.17,71.88) .. controls (396.38,71.88) and (395.74,71.19) .. (395.74,70.35) -- cycle ;

\draw (62.23,92.83) node [anchor=north west][inner sep=0.75pt]  [font=\fontsize{0.35em}{0.42em}\selectfont]  {$0$};
\draw (62.23,72.27) node [anchor=north west][inner sep=0.75pt]  [font=\fontsize{0.35em}{0.42em}\selectfont]  {$w_{1}$};
\draw (338.72,105.68) node [anchor=north west][inner sep=0.75pt]  [font=\fontsize{0.35em}{0.42em}\selectfont]  {$0$};
\draw (288.06,105.23) node [anchor=north west][inner sep=0.75pt]  [font=\fontsize{0.35em}{0.42em}\selectfont]  {$0$};
\draw (170.93,103.62) node [anchor=north west][inner sep=0.75pt]  [font=\fontsize{0.35em}{0.42em}\selectfont]  {$0$};
\draw (132.44,91.57) node [anchor=north west][inner sep=0.75pt]  [font=\fontsize{0.35em}{0.42em}\selectfont]  {$0$};
\draw (433.92,106.48) node [anchor=north west][inner sep=0.75pt]  [font=\fontsize{0.35em}{0.42em}\selectfont]  {$0$};
\draw (471.12,114.67) node [anchor=north west][inner sep=0.75pt]  [font=\fontsize{0.35em}{0.42em}\selectfont]  {$0$};
\draw (170.93,83.06) node [anchor=north west][inner sep=0.75pt]  [font=\fontsize{0.35em}{0.42em}\selectfont]  {$w_{1}$};
\draw (105.49,79.02) node [anchor=north west][inner sep=0.75pt]  [font=\fontsize{0.35em}{0.42em}\selectfont]  {$w_{1}$};
\draw (170.93,62.5) node [anchor=north west][inner sep=0.75pt]  [font=\fontsize{0.35em}{0.42em}\selectfont]  {$w_{2}$};
\draw (145.45,78.67) node [anchor=north west][inner sep=0.75pt]  [font=\fontsize{0.35em}{0.42em}\selectfont]  {$w_{2}$};
\draw (314,93.69) node [anchor=north west][inner sep=0.75pt]  [font=\fontsize{0.35em}{0.42em}\selectfont]  {$w_{1}$};
\draw (263,93.69) node [anchor=north west][inner sep=0.75pt]  [font=\fontsize{0.35em}{0.42em}\selectfont]  {$w_{1}$};
\draw (471.12,94.11) node [anchor=north west][inner sep=0.75pt]  [font=\fontsize{0.35em}{0.42em}\selectfont]  {$w_{1}$};
\draw (433.92,85.92) node [anchor=north west][inner sep=0.75pt]  [font=\fontsize{0.35em}{0.42em}\selectfont]  {$w_{1}$};
\draw (348.74,92.78) node [anchor=north west][inner sep=0.75pt]  [font=\fontsize{0.35em}{0.42em}\selectfont]  {$w_{2}$};
\draw (471.12,73.55) node [anchor=north west][inner sep=0.75pt]  [font=\fontsize{0.35em}{0.42em}\selectfont]  {$w_{2}$};
\draw (446.93,73.02) node [anchor=north west][inner sep=0.75pt]  [font=\fontsize{0.35em}{0.42em}\selectfont]  {$w_{2}$};
\draw (301.07,92.32) node [anchor=north west][inner sep=0.75pt]  [font=\fontsize{0.35em}{0.42em}\selectfont]  {$w_{2}$};
\draw (235,80) node [anchor=north west][inner sep=0.75pt]  [font=\fontsize{0.35em}{0.42em}\selectfont]  {$w_{3}$};
\draw (471.12,52.99) node [anchor=north west][inner sep=0.75pt]  [font=\fontsize{0.35em}{0.42em}\selectfont]  {$w_{3}$};
\draw (409.52,73.11) node [anchor=north west][inner sep=0.75pt]  [font=\fontsize{0.35em}{0.42em}\selectfont]  {$w_{3}$};
\draw (351.74,72.22) node [anchor=north west][inner sep=0.75pt]  [font=\fontsize{0.35em}{0.42em}\selectfont]  {$w_{3}$};
\draw (275.11,71.55) node [anchor=north west][inner sep=0.75pt]  [font=\fontsize{0.35em}{0.42em}\selectfont]  {$w_{3}$};
\draw (237.24,105.29) node [anchor=north west][inner sep=0.75pt]  [font=\fontsize{0.35em}{0.42em}\selectfont]  {$0$};
\draw (211.29,92.74) node [anchor=north west][inner sep=0.75pt]  [font=\fontsize{0.35em}{0.42em}\selectfont]  {$w_{1}$};
\draw (250.25,92.39) node [anchor=north west][inner sep=0.75pt]  [font=\fontsize{0.35em}{0.42em}\selectfont]  {$w_{2}$};
\draw (386.15,105.68) node [anchor=north west][inner sep=0.75pt]  [font=\fontsize{0.35em}{0.42em}\selectfont]  {$0$};
\draw (399.17,92.78) node [anchor=north west][inner sep=0.75pt]  [font=\fontsize{0.35em}{0.42em}\selectfont]  {$w_{1}$};
\draw (383.74,72.2) node [anchor=north west][inner sep=0.75pt]  [font=\fontsize{0.35em}{0.42em}\selectfont]  {$w_{2}$};
\draw (361.74,93.07) node [anchor=north west][inner sep=0.75pt]  [font=\fontsize{0.35em}{0.42em}\selectfont]  {$w_{3}$};

\end{tikzpicture}

%% file: Figures/treeSketch4.tex
\tikzset{every picture/.style={line width=0.75pt}} 

\begin{tikzpicture}[x=0.75pt,y=0.75pt,yscale=-1.5,xscale=1.5]

\draw  [color={rgb, 255:red, 0; green, 0; blue, 0 }  ,draw opacity=1 ][fill={rgb, 255:red, 0; green, 0; blue, 0 }  ,fill opacity=1 ] (292.63,121.26) .. controls (292.63,120.41) and (293.27,119.73) .. (294.06,119.73) .. controls (294.85,119.73) and (295.49,120.41) .. (295.49,121.26) .. controls (295.49,122.1) and (294.85,122.79) .. (294.06,122.79) .. controls (293.27,122.79) and (292.63,122.1) .. (292.63,121.26) -- cycle ;
\draw [color={rgb, 255:red, 0; green, 0; blue, 0 }  ,draw opacity=1 ]   (294.06,100.7) -- (294.06,121.26) ;
\draw [shift={(294.06,107.58)}, rotate = 90] [color={rgb, 255:red, 0; green, 0; blue, 0 }  ,draw opacity=1 ][line width=0.75]    (4.37,-1.32) .. controls (2.78,-0.56) and (1.32,-0.12) .. (0,0) .. controls (1.32,0.12) and (2.78,0.56) .. (4.37,1.32)   ;
\draw    (281.11,87.58) -- (294.06,100.7) ;
\draw [shift={(285.19,91.72)}, rotate = 45.37] [color={rgb, 255:red, 0; green, 0; blue, 0 }  ][line width=0.75]    (4.37,-1.32) .. controls (2.78,-0.56) and (1.32,-0.12) .. (0,0) .. controls (1.32,0.12) and (2.78,0.56) .. (4.37,1.32)   ;
\draw    (307.07,87.8) -- (294.06,100.7) ;
\draw [shift={(302.98,91.86)}, rotate = 135.25] [color={rgb, 255:red, 0; green, 0; blue, 0 }  ][line width=0.75]    (4.37,-1.32) .. controls (2.78,-0.56) and (1.32,-0.12) .. (0,0) .. controls (1.32,0.12) and (2.78,0.56) .. (4.37,1.32)   ;
\draw  [color={rgb, 255:red, 0; green, 0; blue, 0 }  ,draw opacity=1 ][fill={rgb, 255:red, 0; green, 0; blue, 0 }  ,fill opacity=1 ] (292.63,100.7) .. controls (292.63,99.86) and (293.27,99.17) .. (294.06,99.17) .. controls (294.85,99.17) and (295.49,99.86) .. (295.49,100.7) .. controls (295.49,101.55) and (294.85,102.23) .. (294.06,102.23) .. controls (293.27,102.23) and (292.63,101.55) .. (292.63,100.7) -- cycle ;
\draw  [color={rgb, 255:red, 0; green, 0; blue, 0 }  ,draw opacity=1 ][fill={rgb, 255:red, 0; green, 0; blue, 0 }  ,fill opacity=1 ] (305.64,87.8) .. controls (305.64,86.95) and (306.28,86.27) .. (307.07,86.27) .. controls (307.86,86.27) and (308.5,86.95) .. (308.5,87.8) .. controls (308.5,88.64) and (307.86,89.33) .. (307.07,89.33) .. controls (306.28,89.33) and (305.64,88.64) .. (305.64,87.8) -- cycle ;
\draw  [color={rgb, 255:red, 0; green, 0; blue, 0 }  ,draw opacity=1 ][fill={rgb, 255:red, 0; green, 0; blue, 0 }  ,fill opacity=1 ] (279.68,87.58) .. controls (279.68,86.73) and (280.32,86.05) .. (281.11,86.05) .. controls (281.9,86.05) and (282.54,86.73) .. (282.54,87.58) .. controls (282.54,88.43) and (281.9,89.11) .. (281.11,89.11) .. controls (280.32,89.11) and (279.68,88.43) .. (279.68,87.58) -- cycle ;
\draw [color={rgb, 255:red, 0; green, 0; blue, 0 }  ,draw opacity=1 ]   (307.07,67.24) -- (307.07,87.8) ;
\draw [shift={(307.07,74.12)}, rotate = 90] [color={rgb, 255:red, 0; green, 0; blue, 0 }  ,draw opacity=1 ][line width=0.75]    (4.37,-1.32) .. controls (2.78,-0.56) and (1.32,-0.12) .. (0,0) .. controls (1.32,0.12) and (2.78,0.56) .. (4.37,1.32)   ;
\draw  [color={rgb, 255:red, 0; green, 0; blue, 0 }  ,draw opacity=1 ][fill={rgb, 255:red, 0; green, 0; blue, 0 }  ,fill opacity=1 ] (305.64,67.24) .. controls (305.64,66.4) and (306.28,65.71) .. (307.07,65.71) .. controls (307.86,65.71) and (308.5,66.4) .. (308.5,67.24) .. controls (308.5,68.09) and (307.86,68.77) .. (307.07,68.77) .. controls (306.28,68.77) and (305.64,68.09) .. (305.64,67.24) -- cycle ;

\draw (281.74,123.36) node [anchor=north west][inner sep=0.75pt]  [font=\fontsize{0.35em}{0.42em}\selectfont]  {$w_{1}$};
\draw (309.07,89.67) node [anchor=north west][inner sep=0.75pt]  [font=\fontsize{0.35em}{0.42em}\selectfont]  {$w_{4}$};
\draw (296.06,102.57) node [anchor=north west][inner sep=0.75pt]  [font=\fontsize{0.35em}{0.42em}\selectfont]  {$w_{2}$};
\draw (268.88,89.07) node [anchor=north west][inner sep=0.75pt]  [font=\fontsize{0.35em}{0.42em}\selectfont]  {$w_{3}$};
\draw (286.09,113.48) node [anchor=north west][inner sep=0.75pt]  [font=\fontsize{0.24em}{0.28em}\selectfont,color={rgb, 255:red, 208; green, 2; blue, 27 }  ,opacity=1 ]  {$j_{1}$};
\draw (302.66,95.44) node [anchor=north west][inner sep=0.75pt]  [font=\fontsize{0.24em}{0.28em}\selectfont,color={rgb, 255:red, 208; green, 2; blue, 27 }  ,opacity=1 ]  {$j_{3}$};
\draw (279.01,96.83) node [anchor=north west][inner sep=0.75pt]  [font=\fontsize{0.24em}{0.28em}\selectfont,color={rgb, 255:red, 208; green, 2; blue, 27 }  ,opacity=1 ]  {$j_{2}$};
\draw (309.07,69.11) node [anchor=north west][inner sep=0.75pt]  [font=\fontsize{0.35em}{0.42em}\selectfont]  {$w_{5}$};
\draw (309.26,78.44) node [anchor=north west][inner sep=0.75pt]  [font=\fontsize{0.24em}{0.28em}\selectfont,color={rgb, 255:red, 208; green, 2; blue, 27 }  ,opacity=1 ]  {$j_{4}$};
\draw (340.27,90.95) node [anchor=north west][inner sep=0.75pt]  [font=\tiny]  {$( i_{1} ,i_{2} ,i_{3} ,i_{4}) \ =\ ( 1,2,2,4)$};

\end{tikzpicture}

%% file: Figures/TreeSketch5.tex
\tikzset{every picture/.style={line width=0.75pt}} 

\begin{tikzpicture}[x=0.75pt,y=0.75pt,yscale=-1.5,xscale=1.5]

\draw    (259.72,40.35) -- (240.44,40.35) ;
\draw [shift={(253.48,40.35)}, rotate = 180] [color={rgb, 255:red, 0; green, 0; blue, 0 }  ][line width=0.75]    (4.37,-1.32) .. controls (2.78,-0.56) and (1.32,-0.12) .. (0,0) .. controls (1.32,0.12) and (2.78,0.56) .. (4.37,1.32)   ;
\draw    (278.99,40.35) -- (259.72,40.35) ;
\draw [shift={(272.76,40.35)}, rotate = 180] [color={rgb, 255:red, 0; green, 0; blue, 0 }  ][line width=0.75]    (4.37,-1.32) .. controls (2.78,-0.56) and (1.32,-0.12) .. (0,0) .. controls (1.32,0.12) and (2.78,0.56) .. (4.37,1.32)   ;
\draw  [color={rgb, 255:red, 0; green, 0; blue, 0 }  ,draw opacity=1 ][fill={rgb, 255:red, 0; green, 0; blue, 0 }  ,fill opacity=1 ] (258.29,40.35) .. controls (258.29,39.5) and (258.93,38.82) .. (259.72,38.82) .. controls (260.51,38.82) and (261.15,39.5) .. (261.15,40.35) .. controls (261.15,41.19) and (260.51,41.88) .. (259.72,41.88) .. controls (258.93,41.88) and (258.29,41.19) .. (258.29,40.35) -- cycle ;
\draw  [color={rgb, 255:red, 0; green, 0; blue, 0 }  ,draw opacity=1 ][fill={rgb, 255:red, 0; green, 0; blue, 0 }  ,fill opacity=1 ] (239.01,40.35) .. controls (239.01,39.5) and (239.65,38.82) .. (240.44,38.82) .. controls (241.23,38.82) and (241.87,39.5) .. (241.87,40.35) .. controls (241.87,41.19) and (241.23,41.88) .. (240.44,41.88) .. controls (239.65,41.88) and (239.01,41.19) .. (239.01,40.35) -- cycle ;
\draw  [color={rgb, 255:red, 0; green, 0; blue, 0 }  ,draw opacity=1 ][fill={rgb, 255:red, 0; green, 0; blue, 0 }  ,fill opacity=1 ] (277.56,40.35) .. controls (277.56,39.5) and (278.2,38.82) .. (278.99,38.82) .. controls (279.78,38.82) and (280.42,39.5) .. (280.42,40.35) .. controls (280.42,41.19) and (279.78,41.88) .. (278.99,41.88) .. controls (278.2,41.88) and (277.56,41.19) .. (277.56,40.35) -- cycle ;
\draw  [dash pattern={on 0.84pt off 2.51pt}]  (310.22,40.35) -- (278.99,40.35) ;
\draw  [color={rgb, 255:red, 0; green, 0; blue, 0 }  ,draw opacity=1 ][fill={rgb, 255:red, 0; green, 0; blue, 0 }  ,fill opacity=1 ] (308.79,40.35) .. controls (308.79,39.5) and (309.43,38.82) .. (310.22,38.82) .. controls (311.01,38.82) and (311.65,39.5) .. (311.65,40.35) .. controls (311.65,41.19) and (311.01,41.88) .. (310.22,41.88) .. controls (309.43,41.88) and (308.79,41.19) .. (308.79,40.35) -- cycle ;
\draw    (330.61,40.35) -- (310.22,40.35) ;
\draw [shift={(323.82,40.35)}, rotate = 180] [color={rgb, 255:red, 0; green, 0; blue, 0 }  ][line width=0.75]    (4.37,-1.32) .. controls (2.78,-0.56) and (1.32,-0.12) .. (0,0) .. controls (1.32,0.12) and (2.78,0.56) .. (4.37,1.32)   ;
\draw    (349.88,40.35) -- (330.61,40.35) ;
\draw [shift={(343.65,40.35)}, rotate = 180] [color={rgb, 255:red, 0; green, 0; blue, 0 }  ][line width=0.75]    (4.37,-1.32) .. controls (2.78,-0.56) and (1.32,-0.12) .. (0,0) .. controls (1.32,0.12) and (2.78,0.56) .. (4.37,1.32)   ;
\draw  [color={rgb, 255:red, 0; green, 0; blue, 0 }  ,draw opacity=1 ][fill={rgb, 255:red, 0; green, 0; blue, 0 }  ,fill opacity=1 ] (348.45,40.35) .. controls (348.45,39.5) and (349.09,38.82) .. (349.88,38.82) .. controls (350.67,38.82) and (351.31,39.5) .. (351.31,40.35) .. controls (351.31,41.19) and (350.67,41.88) .. (349.88,41.88) .. controls (349.09,41.88) and (348.45,41.19) .. (348.45,40.35) -- cycle ;
\draw  [color={rgb, 255:red, 0; green, 0; blue, 0 }  ,draw opacity=1 ][fill={rgb, 255:red, 0; green, 0; blue, 0 }  ,fill opacity=1 ] (329.18,40.35) .. controls (329.18,39.5) and (329.82,38.82) .. (330.61,38.82) .. controls (331.4,38.82) and (332.04,39.5) .. (332.04,40.35) .. controls (332.04,41.19) and (331.4,41.88) .. (330.61,41.88) .. controls (329.82,41.88) and (329.18,41.19) .. (329.18,40.35) -- cycle ;
\draw    (259.94,69.9) -- (240.67,69.9) ;
\draw [shift={(253.7,69.9)}, rotate = 180] [color={rgb, 255:red, 0; green, 0; blue, 0 }  ][line width=0.75]    (4.37,-1.32) .. controls (2.78,-0.56) and (1.32,-0.12) .. (0,0) .. controls (1.32,0.12) and (2.78,0.56) .. (4.37,1.32)   ;
\draw    (279.22,69.9) -- (259.94,69.9) ;
\draw [shift={(272.98,69.9)}, rotate = 180] [color={rgb, 255:red, 0; green, 0; blue, 0 }  ][line width=0.75]    (4.37,-1.32) .. controls (2.78,-0.56) and (1.32,-0.12) .. (0,0) .. controls (1.32,0.12) and (2.78,0.56) .. (4.37,1.32)   ;
\draw  [color={rgb, 255:red, 0; green, 0; blue, 0 }  ,draw opacity=1 ][fill={rgb, 255:red, 0; green, 0; blue, 0 }  ,fill opacity=1 ] (258.51,69.9) .. controls (258.51,69.06) and (259.15,68.37) .. (259.94,68.37) .. controls (260.73,68.37) and (261.37,69.06) .. (261.37,69.9) .. controls (261.37,70.75) and (260.73,71.44) .. (259.94,71.44) .. controls (259.15,71.44) and (258.51,70.75) .. (258.51,69.9) -- cycle ;
\draw  [color={rgb, 255:red, 0; green, 0; blue, 0 }  ,draw opacity=1 ][fill={rgb, 255:red, 0; green, 0; blue, 0 }  ,fill opacity=1 ] (239.24,69.9) .. controls (239.24,69.06) and (239.88,68.37) .. (240.67,68.37) .. controls (241.46,68.37) and (242.1,69.06) .. (242.1,69.9) .. controls (242.1,70.75) and (241.46,71.44) .. (240.67,71.44) .. controls (239.88,71.44) and (239.24,70.75) .. (239.24,69.9) -- cycle ;
\draw  [color={rgb, 255:red, 0; green, 0; blue, 0 }  ,draw opacity=1 ][fill={rgb, 255:red, 0; green, 0; blue, 0 }  ,fill opacity=1 ] (277.79,69.9) .. controls (277.79,69.06) and (278.43,68.37) .. (279.22,68.37) .. controls (280.01,68.37) and (280.65,69.06) .. (280.65,69.9) .. controls (280.65,70.75) and (280.01,71.44) .. (279.22,71.44) .. controls (278.43,71.44) and (277.79,70.75) .. (277.79,69.9) -- cycle ;
\draw  [dash pattern={on 0.84pt off 2.51pt}]  (310.44,69.9) -- (279.22,69.9) ;
\draw  [color={rgb, 255:red, 0; green, 0; blue, 0 }  ,draw opacity=1 ][fill={rgb, 255:red, 0; green, 0; blue, 0 }  ,fill opacity=1 ] (309.01,69.9) .. controls (309.01,69.06) and (309.65,68.37) .. (310.44,68.37) .. controls (311.23,68.37) and (311.87,69.06) .. (311.87,69.9) .. controls (311.87,70.75) and (311.23,71.44) .. (310.44,71.44) .. controls (309.65,71.44) and (309.01,70.75) .. (309.01,69.9) -- cycle ;
\draw    (330.83,69.9) -- (310.44,69.9) ;
\draw [shift={(324.04,69.9)}, rotate = 180] [color={rgb, 255:red, 0; green, 0; blue, 0 }  ][line width=0.75]    (4.37,-1.32) .. controls (2.78,-0.56) and (1.32,-0.12) .. (0,0) .. controls (1.32,0.12) and (2.78,0.56) .. (4.37,1.32)   ;
\draw    (350.1,69.9) -- (330.83,69.9) ;
\draw [shift={(343.87,69.9)}, rotate = 180] [color={rgb, 255:red, 0; green, 0; blue, 0 }  ][line width=0.75]    (4.37,-1.32) .. controls (2.78,-0.56) and (1.32,-0.12) .. (0,0) .. controls (1.32,0.12) and (2.78,0.56) .. (4.37,1.32)   ;
\draw  [color={rgb, 255:red, 0; green, 0; blue, 0 }  ,draw opacity=1 ][fill={rgb, 255:red, 0; green, 0; blue, 0 }  ,fill opacity=1 ] (348.67,69.9) .. controls (348.67,69.06) and (349.31,68.37) .. (350.1,68.37) .. controls (350.89,68.37) and (351.53,69.06) .. (351.53,69.9) .. controls (351.53,70.75) and (350.89,71.44) .. (350.1,71.44) .. controls (349.31,71.44) and (348.67,70.75) .. (348.67,69.9) -- cycle ;
\draw  [color={rgb, 255:red, 0; green, 0; blue, 0 }  ,draw opacity=1 ][fill={rgb, 255:red, 0; green, 0; blue, 0 }  ,fill opacity=1 ] (329.4,69.9) .. controls (329.4,69.06) and (330.04,68.37) .. (330.83,68.37) .. controls (331.62,68.37) and (332.26,69.06) .. (332.26,69.9) .. controls (332.26,70.75) and (331.62,71.44) .. (330.83,71.44) .. controls (330.04,71.44) and (329.4,70.75) .. (329.4,69.9) -- cycle ;

\draw (241.01,43.75) node [anchor=north west][inner sep=0.75pt]  [font=\fontsize{0.35em}{0.42em}\selectfont]  {$w_{1}$};
\draw (247.45,32) node [anchor=north west][inner sep=0.75pt]  [font=\fontsize{0.24em}{0.28em}\selectfont,color={rgb, 255:red, 74; green, 144; blue, 226 }  ,opacity=1 ]  {$1$};
\draw (350.45,43.75) node [anchor=north west][inner sep=0.75pt]  [font=\fontsize{0.35em}{0.42em}\selectfont]  {$w_{m}$};
\draw (331.18,43.75) node [anchor=north west][inner sep=0.75pt]  [font=\fontsize{0.35em}{0.42em}\selectfont]  {$w_{m-1}$};
\draw (310.79,43.75) node [anchor=north west][inner sep=0.75pt]  [font=\fontsize{0.35em}{0.42em}\selectfont]  {$w_{m-2}$};
\draw (279.56,43.75) node [anchor=north west][inner sep=0.75pt]  [font=\fontsize{0.35em}{0.42em}\selectfont]  {$w_{3}$};
\draw (260.29,43.75) node [anchor=north west][inner sep=0.75pt]  [font=\fontsize{0.35em}{0.42em}\selectfont]  {$w_{2}$};
\draw (267.01,32) node [anchor=north west][inner sep=0.75pt]  [font=\fontsize{0.24em}{0.28em}\selectfont,color={rgb, 255:red, 74; green, 144; blue, 226 }  ,opacity=1 ]  {$2$};
\draw (315.01,32) node [anchor=north west][inner sep=0.75pt]  [font=\fontsize{0.24em}{0.28em}\selectfont,color={rgb, 255:red, 74; green, 144; blue, 226 }  ,opacity=1 ]  {$m-2$};
\draw (335.01,32) node [anchor=north west][inner sep=0.75pt]  [font=\fontsize{0.24em}{0.28em}\selectfont,color={rgb, 255:red, 74; green, 144; blue, 226 }  ,opacity=1 ]  {$m-1$};
\draw (241.24,73.3) node [anchor=north west][inner sep=0.75pt]  [font=\fontsize{0.35em}{0.42em}\selectfont]  {$w_{1}$};
\draw (242.12,63) node [anchor=north west][inner sep=0.75pt]  [font=\fontsize{0.24em}{0.28em}\selectfont,color={rgb, 255:red, 74; green, 144; blue, 226 }  ,opacity=1 ]  {$m-2$};
\draw (350.67,73.3) node [anchor=north west][inner sep=0.75pt]  [font=\fontsize{0.35em}{0.42em}\selectfont]  {$w_{m}$};
\draw (331.4,73.3) node [anchor=north west][inner sep=0.75pt]  [font=\fontsize{0.35em}{0.42em}\selectfont]  {$w_{2}$};
\draw (311.01,73.3) node [anchor=north west][inner sep=0.75pt]  [font=\fontsize{0.35em}{0.42em}\selectfont]  {$w_{3}$};
\draw (279.79,73.3) node [anchor=north west][inner sep=0.75pt]  [font=\fontsize{0.35em}{0.42em}\selectfont]  {$w_{m-2}$};
\draw (260.51,73.3) node [anchor=north west][inner sep=0.75pt]  [font=\fontsize{0.35em}{0.42em}\selectfont]  {$w_{m-1}$};
\draw (262.79,63) node [anchor=north west][inner sep=0.75pt]  [font=\fontsize{0.24em}{0.28em}\selectfont,color={rgb, 255:red, 74; green, 144; blue, 226 }  ,opacity=1 ]  {$m-3$};
\draw (317.01,63) node [anchor=north west][inner sep=0.75pt]  [font=\fontsize{0.24em}{0.28em}\selectfont,color={rgb, 255:red, 74; green, 144; blue, 226 }  ,opacity=1 ]  {$1$};
\draw (332.23,63) node [anchor=north west][inner sep=0.75pt]  [font=\fontsize{0.24em}{0.28em}\selectfont,color={rgb, 255:red, 74; green, 144; blue, 226 }  ,opacity=1 ]  {$m-1$};
\draw (381.77,37.75) node [anchor=north west][inner sep=0.75pt]  [font=\tiny,color={rgb, 255:red, 74; green, 144; blue, 226 }  ,opacity=1 ]  {$( \tau ,\ w,\ j\ )$};
\draw (381.27,67.8) node [anchor=north west][inner sep=0.75pt]  [font=\tiny,color={rgb, 255:red, 74; green, 144; blue, 226 }  ,opacity=1 ]  {$( \tau ',\ w',\ j'\ )$};

\end{tikzpicture}